\documentclass[12pt,CJK]{article}
\usepackage{geometry}
\usepackage{amsfonts,amsmath,amssymb,amsthm,mathrsfs}
\usepackage{latexsym}
\usepackage{graphicx}
\usepackage{indentfirst}
\usepackage{color}
\usepackage{float} %设置图片浮动位置的宏包
\usepackage{subfigure} %插入多图时用子图显示的宏包
\usepackage{fancyhdr}
\usepackage[colorlinks,linktocpage,linkcolor=blue,citecolor=red]{hyperref}
\usepackage{titlesec}

\usepackage{authblk}

\usepackage{pgf}
\usepackage{tikz}
\usepackage[utf8]{inputenc}
\usetikzlibrary{arrows,automata}
\usetikzlibrary{positioning}

\tikzset{
    state/.style={
           rectangle,
           rounded corners,
           draw=black, very thick,
           minimum height=2em,
           inner sep=2pt,
           text centered,
           },
}

\tikzset{global scale/.style={
    scale=#1,
    every node/.append style={scale=#1}
  }
}

\definecolor{darkgreen}{rgb}{0.00,0.50,0.25}
\definecolor{purple}{rgb}{0.54,0.17,0.89}
\definecolor{purple1}{rgb}{1.00,0.00,1.00}

\titleformat*{\subsection}{\bfseries}

\theoremstyle{plain}                       % default
\newtheorem{lemma}{Lemma}[section]
\newtheorem{theorem}[lemma]{Theorem}
\newtheorem{corollary}[lemma]{Corollary}
\newtheorem{remark}[lemma]{Remark}

\newtheorem{proposition}[lemma]{Proposition}

\theoremstyle{remark}

\numberwithin{equation}{section}

\def\Xint#1{\mathchoice
  {\XXint\displaystyle\textstyle{#1}}%
  {\XXint\textstyle\scriptstyle{#1}}%
  {\XXint\scriptstyle\scriptscriptstyle{#1}}%
  {\XXint\scriptscriptstyle\scriptscriptstyle{#1}}%
  \!\int}
\def\XXint#1#2#3{{\setbox0=\hbox{$#1{#2#3}{\int}$}
  \vcenter{\hbox{$#2#3$}}\kern-.5\wd0}}

\def\dashint{\Xint-}
\DeclareMathOperator*{\esssup}{ess\,sup}
\begin{document}
\allowdisplaybreaks
%\pagestyle{myheadings}
%\markboth{$~$ \hfill {\rm Q. Xu,} \hfill $~$}
%{$~$ \hfill {\rm  } \hfill$~$}
%\author{Li Wang
%%\thanks{Email: lwang10@lzu.edu.cn.}
%\quad Qiang Xu
%\thanks{Corresponding author}
%\thanks{Email: xuqiang09@lzu.edu.cn.}
%\quad Peihao Zhao
%%\thanks{Email: phzhao@lzu.edu.cn.}
%\\
%School of Mathematics and Statistics, Lanzhou University, \\
%Gansu, 730000, PR China.
%\vspace{0.5cm}
%}

%\author[a]{Claudia Raithel\thanks{Email: claudia.raithel@tuwien.ac.at.}}

\author[a]{Li Wang \thanks{Email:li$\_$wang@lzu.edu.cn}
}

\author[a]{Qiang Xu \thanks{Email: xuq@lzu.edu.cn}
}

%\author[a]{Zhifei Zhang\thanks{Email: zfzhang@math.pku.edu.cn.}}
%

%\affil[a]{TU Wien, Wiedner Hauptstrasse 8-10, 1040 Wien, Austria
%\authorcr Email:wangli@math.pku.edu.cn %\authorcr is for a new line.
%}

\affil[a]{School of Mathematics and Statistics, Lanzhou University, Lanzhou, 730000, China.}

%

%\author{Weiren Zhao
%\thanks{Email: xuqiang@math.pku.edu.cn.}
%\thanks{This work was supported by the National Natural Science Foundation of China (Grant No. 11471147).}

\title{
\textbf{Boundary estimates and Green function's expansion for
elliptic systems with random coefficients}
}

\maketitle
\begin{abstract}
We investigate boundary estimates for elliptic operators with stationary random coefficients exhibiting integrable correlations, arising from  stochastic homogenization theory.
As practical applications, we establish decay estimates for Green functions in both quenched and annealed senses. Furthermore, we derive notable annealed estimates for boundary correctors, including central limit theorem (CLT)-scaling type estimates.
By extending the lemma of Bella, Giunti, and Otto \cite{Bella-Giunti-Otto17} to accommodate boundary conditions, we ultimately obtain error estimates for the two-scale expansion of Green functions at the level of mixed derivatives, thereby establishing connections to other related fields.

\end{abstract}

\tableofcontents

\section{Introduction}

\subsection{Motivation and assumptions}

\noindent
From the perspective of quantitative stochastic homogenization theory,
the error estimate for two-scale expansion of Green functions plays a pivotal role in various applied research areas, including boundary layer problems \cite{Armstrong-Kuusi-Mourrat-Prange17,Gerard-Masmoudi12,Shen-Zhuge18},
homogenization error estimates \cite{Avellaneda-Lin91,Kenig-Lin-Shen14}, systematic error estimates in the Representative Volume Element (RVE) method \cite{Clozeau-Josien-Otto-Xu,Gloria-Neukamm-Otto15}, and
some artificial boundary condition problems \cite{Lu-Otto21,Lu-Otto-Wang21}.
For this application-oriented goal, the present work is
devoted to establishing the error estimates of two-scale expansion of Green functions at the level of the mixed derivatives.
To obtain an almost sharp decay estimate, we resort to boundary correctors, which in turn requires some annealed boundary estimates, and this requirement constitutes another focus of the paper.

%Unlike the framework based upon quantified ergodicity under a finite range assumption, established by S. Armstrong and C. Smart \cite{Armstrong-Smart16} and culminated with the monograph \cite{Armstrong-Kuusi-Mourrat19}, we adopt the Gloria-Neukamm-Otto's strategy \cite{Gloria-Neukamm-Otto20,Gloria-Neukamm-Otto21}, which seems to be applicable to more probabilistic models addressed in \cite{Torquato02}.
By redefining the \emph{minimal radius}, it is possible to treat
Gloria-Neukamm-Otto's results on the correctors as the ``input'', and then apply the deterministic arguments for boundary Lipschitz estimates
(developed by Shen \cite{Shen16,Shen18} originally inspired by Armstrong and Smart \cite{Armstrong-Smart16}) to the random setting.
This approach’s advantage lies in reducing the study of boundary correctors to annealed regularity estimates (with the aid of Green functions), rather than directly repeating the sensitive estimates established for correctors in \cite{Gloria-Otto11,Gloria-Neukamm-Otto20,Gloria-Neukamm-Otto21,Josien-Otto22}. Nevertheless, sensitive estimates remain indispensable for CLT-scaling type estimates.

%In a certain sense,
%this paper shows that some ideas and methods developed from periodic homogenization can also benefit the quantitative stochastic homogenization theory.

Precisely, we consider a family of elliptic operators in divergence form
\begin{equation*}
\mathcal{L}_\varepsilon :=-\nabla\cdot a^\varepsilon\nabla = -\nabla\cdot a(\cdot/\varepsilon)\nabla,\qquad \varepsilon>0,
\end{equation*}
where $a$ satisfies $\lambda$-uniformly elliptic conditions, i.e., for some
$\lambda\in(0,1)$ there holds
\begin{equation}\label{a:1}
\lambda|\xi|^2 \leq  \xi\cdot a(x)\xi \leq \lambda^{-1}|\xi|^2
\qquad
\forall\xi\in\mathbb{R}^{d}\text{~and~}x\in\mathbb{R}^d.
\end{equation}
Although the subject of this paper is elliptic systems, we shall adopt scalar notation and terminology throughout for the sake of expository clarity.

One can introduce the configuration space that is the set
of coefficient fields satisfying $\eqref{a:1}$, equipped with a
probability measure, referred to as an ensemble with
the expectation denoted by $\langle\cdot\rangle$.
This ensemble is assumed \emph{stationary} (i.e., for all shift vectors $z\in\mathbb{R}^d$, $a(\cdot+z)$ and $a(\cdot)$ have the same law under $\langle\cdot\rangle$), and additionally satisfies the \emph{spectral gap condition}
\begin{equation}\label{a:2}
 \big\langle (F-\langle F\rangle)^2 \big\rangle \leq \lambda_1
 \Big\langle\int_{\mathbb{R}^d}\Big(\dashint_{B_1(x)}\big|\frac{\partial F}{\partial a}\big|\Big)^2 dx\Big\rangle,
\end{equation}
where the random tensor field $\frac{\partial F}{\partial a}$ (depending on $(a,x)$) is the functional derivative of $F$ with respect to $a$, defined via
\begin{equation*}\label{functional}
      \delta F:= \lim_{\varepsilon\to 0} \frac{F(a+\varepsilon\delta a)
       -F(a)}{\varepsilon}
       =\int_{\mathbb{R}^d} \frac{\partial F(a)}{\partial a_{ij}(x)}
       (\delta a)_{ij}(x)dx,
\end{equation*}
and the Einstein summation convention for repeated indices is
applied throughout.
Moreover, the admissible coefficients satisfy the \emph{local smoothness} assumption, i.e., there exists $\bar{\sigma}\in(0,1)$ such that for any $0<\sigma<\bar{\sigma}$
and $1\leq \gamma<\infty$ there holds
\begin{equation}\label{a:3}
 \big\langle  \|a\|_{C^{0,\sigma}(B_1)}^\gamma
 \big\rangle \leq \lambda_2(d,\lambda,\sigma,\gamma).
\end{equation}

Let $B_r(x)$ denote the open ball centered at $x$ of radius $r$, abbreviated as $B$,
and we omit $x$ when $x=0$. For this shorthand,
let $x_B$ and $r_B$ represent the center and radius of $B$, respectively.
Let $\Omega\subset\mathbb{R}^d$ with $d\geq 2$ be a bounded domain, and
$R_0\geq 8$ represents the diameter of $\Omega$ throughout the paper
(where $R_0\geq 8$ is not essential but suited for later discussions).
The notation ``$\lesssim_{\lambda,d,\cdots}$'' means that the multiplicative constant depends on $\lambda,d,\cdots$ (see Subsection
$\ref{notation}$ for the detailed explanations and other notations used in the paper).

%where the symbol ``$\lesssim_{\lambda,\lambda_1,\cdots,\lambda_n}$'' reads ``$\leq C(\lambda,\lambda_1,\cdots,\lambda_n)$, and
%the definition of the functional derivative of $F$ with respect to $a$ can be found in $\eqref{functional}$ (or see \cite[pp.15]{Josien-Otto22}).

%Throughout the paper, let $\Omega\subset\mathbb{R}^d$ with $d\geq 2$
%be a bounded domain, and $R_0$ represents the diameter of $\Omega$.
%Let $B_r(x)$ be an open ball with radius $r$ and centered at $x$.
%To plainly describe a local boundary estimate, it is fine to assume that there is a new coordinate system in $\mathbb{R}^d$ obtained
%from the standard Euclidean coordinate system by translation and rotation so that
%\begin{equation*}
% B_r(0)\cap \Omega
% = B_r(0)\cap\big\{(x^\prime,x_d)\in\mathbb{R}^d:x^\prime\in\mathbb{R}^{d-1}
% ~\text{and}~ x_d>\psi(x^\prime)\big\},
%\end{equation*}
%where $\psi:\mathbb{R}^{d-1}\to\mathbb{R}$ is a boundary function, satisfying
%\begin{equation}\label{}
%\left\{\begin{aligned}
%&\psi(0) = 0,~\text{supp}(\psi)\subset\{x'\in\mathbb{R}^{d-1}:|x'|\leq 1\},
%\|\nabla\psi\|_{L^\infty(\mathbb{R}^{d-1})}\leq M_0;\\
%&|\nabla\psi(x')-\nabla\psi(y')|\leq \omega_{\text{B}}(|x'-y'|)
%\quad \forall x',y'\in\mathbb{R}^{d-1}.
%\end{aligned}\right.
%\end{equation}

\subsection{Main results}

\begin{theorem}[boundary regularity estimates at large scales]\label{thm:1}
Let $\Omega\subset\mathbb{R}^d$ with $d\geq 2$ be a bounded domain, and
$\varepsilon\in(0,1]$.
Suppose that $\langle\cdot\rangle$ satisfies the spectral gap condition $\eqref{a:2}$ and the
admissible coefficients satisfy
the regularity condition $\eqref{a:3}$. Let $u_\varepsilon$ be a weak solution to $\mathcal{L}_\varepsilon(u_\varepsilon) = 0$ in
$B_4\cap\Omega$ and $u_\varepsilon =0$ on $B_4\cap\partial\Omega$
with $\partial\Omega\ni\{0\}$.
We can obtain the following boundary regularity estimates:
\begin{itemize}
  \item [\emph{(i)}] \emph{H\"older's estimate}. If $\Omega$ is
    the bounded $C^1$ region and $\alpha\in(0,1)$. Then, there exists a stationary random field $\chi_*\geq 1$, satisfying  $\langle\chi_*^\beta\rangle\lesssim_{\lambda,\lambda_1,\lambda_2,d,\beta} 1$ for any $\beta\in[1,\infty)$, such that for all    $\varepsilon\leq\big(\inf_{x\in\mathbb{R}^d}
      \inf_{a}\chi_{*}(a,x)\big)^{-1}$, the following estimate
\begin{equation}\label{pri:hol}
\Big(\dashint_{B_r\cap\Omega}|\nabla u_\varepsilon|^2
\Big)^{1/2}
\lesssim_{\lambda,d,\partial\Omega,\alpha} r^{\alpha-1}
\Big(\dashint_{B_1\cap\Omega}|\nabla u_\varepsilon|^2
\Big)^{1/2}
\end{equation}
holds for any $r\in[\varepsilon\chi_*, 1]$.
  \item [\emph{(ii)}] \emph{Lipschitz estimate}. If
  $\Omega$ is the bounded $C^{1,\tau}$ region with $\tau\in(0,1]$.
  Then, there exists a stationary random field $\chi_{**}\geq 1$ with $\langle\chi_{**}^\beta\rangle\lesssim_{\lambda,\lambda_1,\lambda_2,d,\beta} 1$ for any $\beta\in[1,\infty)$, such that for all $\varepsilon\leq
  \big(\inf_{x\in\mathbb{R}^d}\inf_a\chi_{**}(a,x)\big)^{-1}$, there holds
\begin{equation}\label{pri:lip}
\Big(\dashint_{B_r\cap\Omega}|\nabla u_\varepsilon|^2
\Big)^{1/2}
\lesssim_{\lambda,d,\partial\Omega,\tau}
\Big(\dashint_{B_1\cap\Omega}|\nabla u_\varepsilon|^2
\Big)^{1/2}
\end{equation}
for any $r\in[\varepsilon\chi_{**},1]$.
\end{itemize}
\end{theorem}

Once boundary regularity estimates are established,
one can construct Green functions (see e.g. \cite{Dong-Kim09,Hofmann-Kim07,Taylor-Kim-Brow13}), denoted by $G_\varepsilon(x,\cdot)$, as follows:
\begin{equation}\label{pde:2}
\left\{\begin{aligned}
\mathcal{L}_\varepsilon^{*}(G_\varepsilon(x,\cdot))
&= \delta_x &\quad&\text{in}\quad\Omega;\\
G_\varepsilon(x,\cdot)&=0
&\quad&\text{on}\quad\partial\Omega,
\end{aligned}\right.
\end{equation}
where $\mathcal{L}_\varepsilon^{*}:=-\nabla\cdot a^{*}(\cdot/\varepsilon)\nabla$ is the adjoint operator of  $\mathcal{L}_\varepsilon$, and $\delta_x$ represents the Dirac delta function with pole at $x\in\Omega$.
Let $f\in C_0^\infty(\Omega;\mathbb{R}^d)$ and $u_\varepsilon$ be associated with $\mathcal{L}_\varepsilon(u_\varepsilon)=\nabla\cdot f$ in $\Omega$ and $u_\varepsilon = 0$ on $\partial\Omega$.
Then $u_\varepsilon$ can be plainly represented by
\begin{equation}\label{eq:2.1}
\begin{aligned}
u_\varepsilon(x)
=-\int_{\Omega} f(y)\cdot\nabla G_\varepsilon(x,y)dy.
\end{aligned}
\end{equation}

We now introduce
the so-called Dirichlet boundary correctors, i.e., for any  $i\in\{1,\cdots,d\}$, consider the following equations:
\begin{equation}\label{pde:10}
\left\{\begin{aligned}
\mathcal{L}_\varepsilon(\tilde{\Phi}_{\varepsilon,i})
 &= \nabla\cdot a^\varepsilon e_i
 \quad &\text{in}&\quad \Omega;\\
 \tilde{\Phi}_{\varepsilon,i} &= 0
 \quad &\text{on}&\quad \partial\Omega,
\end{aligned}\right.
\end{equation}
where $e_i$ is the canonical basis of $\mathbb{R}^d$.
In light of  the right-hand side of $\eqref{pde:10}$, it is natural to recall the definition of the corrector (denoted by $\phi$), i.e.,
\begin{equation}\label{corrector}
 \nabla\cdot a(\nabla\phi_i +e_i) = 0
 \quad\text{in}\quad\mathbb{R}^d,
\end{equation}
anchored at the zero point.
It is known from Gloria-Neukamm-Otto's strategy that, under the hypotheses $\eqref{a:2}$ and $\eqref{a:3}$, for any $p\in[1,\infty)$, one can derive higher moment estimates $\langle|\nabla\phi|^p\rangle^{\frac{1}{p}}
\lesssim_{\lambda,\lambda_1,\lambda_2,d,p} 1$ and the
fluctuation estimate (see Lemma $\ref{lemma:*3}$), i.e.,
\begin{equation}\label{pri:**3}
 \big\langle\big|\phi(x)
 -\phi(0)\big|^p \big\rangle^{\frac{1}{p}}
 \lesssim_{\lambda,\lambda_1,\lambda_2,d,p} \mu_d(x):=\left\{\begin{aligned}
% &\sqrt{2+|x|}; &~&d=1;\\
 &\ln^{\frac{1}{2}} (2+|x|); &~& d=2;\\
 & 1 &~& d>2.
 \end{aligned}\right.
\end{equation}

The following work focuses on transferring annealed estimates from correctors to Dirichlet boundary correctors, which plays a crucial role in
the subsequent study.

\begin{theorem}[boundary correctors]\label{thm:2}
Let $\Omega\ni \{0\}$ be a bounded $C^{1,\tau}$ domain with
the uniform interior ball condition and $\tau\in(0,1]$,
and $\varepsilon\in(0,1]$.
Suppose that $\langle\cdot\rangle$ satisfies the spectral gap condition $\eqref{a:2}$, and the
admissible coefficients satisfy the smoothness condition $\eqref{a:3}$. Let $\tilde{\Phi}_{\varepsilon}=
\{\tilde{\Phi}_{\varepsilon,i}\}_{i=1}^d$ be the solution of
the equations $\eqref{pde:10}$.
Then, for any $p\in[1,\infty)$, there hold the following annealed estimates:
\begin{subequations}
\begin{align}
&\esssup_{x\in\Omega} \big\langle|\tilde{\Phi}_{\varepsilon}(x)|^p\big\rangle^{\frac{1}{p}}
 \lesssim \varepsilon\mu_{d}(R_0/\varepsilon); \label{pri:1.a}\\
&\esssup_{x\in\Omega}\big\langle|\nabla\tilde{\Phi}_{\varepsilon}(x)
|^p\big\rangle^{\frac{1}{p}}
\lesssim\mu_d(R_0/\varepsilon).
\label{pri:1.b}
\end{align}
\end{subequations}
Moreover, let $Q_{\varepsilon} := \tilde{\Phi}_{\varepsilon}
-\varepsilon\phi(\cdot/\varepsilon)$ and $\delta(z):=\text{dist}(z,\partial\Omega)$ with $z\in\Omega$.
Then, we have
\begin{equation}\label{pri:1.c}
 \big\langle|\nabla Q_{\varepsilon}(z)|^p\big\rangle^{\frac{1}{p}}
 \lesssim \mu_d(R_0/\varepsilon)
 \min\Big\{1,\frac{\varepsilon}{\delta(z)}\Big\}
 \qquad \forall z\in\Omega,
\end{equation}
where the multiplicative constant depends on $\lambda,\lambda_1,\lambda_2,d,\Omega,\tau$, and $p$.
\end{theorem}

Similar to the CLT-scaling of $\nabla\phi$ (see e.g. \cite{Gloria-Neukamm-Otto21,Josien-Otto22}), with the help of Theorem
$\ref{thm:2}$, sensitive estimates, and
a bootstrap argument, one can derive CLT-scaling type estimates for Dirichlet boundary correctors.
The following estimate seems to be optimal up to a $\kappa$-loss at most.
The notation
``$\lfloor x\rfloor$'' denotes the largest integer not more than $x$, and ``$\lceil x\rceil$'' represents the smallest integer not less than $x$.

\begin{theorem}[CLT-scaling type estimates]\label{thm:2*}
Let $\Omega\subset\mathbb{R}^d$ (with $d\geq 2$) be a bounded
$C^{\lfloor d/2\rfloor+1}$ domain with $\Omega\ni\{0\}$.
Suppose that the ensemble $\langle\cdot\rangle$ satisfies
$\eqref{a:2}$ and $\eqref{a:3}$.
Let $p\in[1,\infty)$, $\varepsilon\in(0,1]$, and $Q_\varepsilon$ be given as in Theorem $\ref{thm:2}$.  Then, there exists $\kappa_d$, satisfying
$\kappa_2=0$ if $d=2$ and $\kappa_d=\kappa$ if $d\geq 3$ with
$0<\kappa\ll 1$,  such that the following estimate
\begin{equation}\label{pri:1.3}
 \big\langle|\nabla Q_\varepsilon(z)|^p\big\rangle^{\frac{1}{p}}
 \lesssim \mu_d^{\lceil\frac{d}{2}\rceil}(R_0/\varepsilon)
 \Big(\frac{\varepsilon}{\delta(z)}\Big)^{\frac{d}{2}-\kappa_d}
\end{equation}
holds for any $z\in\Omega$ and $\delta(z)\geq 2\varepsilon$, where
the multiplicative constant relies on
$\lambda,\lambda_1,\lambda_2,d,\Omega,p$, and $\kappa$ at most.
Moreover, let $\tilde{\Phi}_{\varepsilon}=
\{\tilde{\Phi}_{\varepsilon,i}\}_{i=1}^d$ be the solution of
the equations $\eqref{pde:10}$.
Then, for all $2B\subset\Omega$ with
$\text{dist}(x_B,\partial\Omega)\geq 2r_B$ and $r_B\geq 2\varepsilon$,
and any $g\in C_0^\infty(2B)$ satisfying $g=\textbf{1}_B$ on $B$
with $|\nabla^k g|\lesssim_{d,k}1/r_B^{k+d}$ for any positive integer
$k\in\mathbb{N}$, we obtain the CLT-scaling type estimate
\begin{equation}\label{pri:1.4}
\Big\langle\big|\int_{\Omega}\nabla\tilde{\Phi}_{\varepsilon}\cdot g\big|^{p}
\Big\rangle^{\frac{1}{p}}
\lesssim_{\lambda,\lambda_1,\lambda_2,d,\Omega,p}
\mu_d^{\lceil\frac{d}{2}\rceil+1}(R_0/\varepsilon)
\Big(\frac{\varepsilon}{r_B}\Big)^{\frac{d}{2}},
\end{equation}
where $\textbf{1}_B$ represents the indicator function of the region $B$.
\end{theorem}

For further application scenarios,
we study the asymptotic behavior of $\nabla_x\nabla_y G_\varepsilon(x,y)$ in the annealed sense, as $\varepsilon\to 0$, and we present it below.

\begin{theorem}[Green function's expansion]\label{thm:3}
Let $\Omega\ni \{0\}$ be a bounded $C^{2,\tau}$ domain with $\tau\in(0,1]$, and $\varepsilon\in(0,1]$.
Suppose that the ensemble $\langle\cdot\rangle$ satisfies $\eqref{a:2}$ and
$\eqref{a:3}$. Assume that $G_\varepsilon(x,\cdot)$ is the Green function given by $\eqref{pde:2}$. Let
$\tilde{\Phi}_{\varepsilon,i}$ be the solution of
the equations $\eqref{pde:10}$, and we define $\Phi_{\varepsilon,i}:=\tilde{\Phi}_{\varepsilon,i}+x_i$
(the adjoint boundary corrector is denoted by
$\Phi_{\varepsilon,i}^{*}$). We introduce the corresponding truncated
expansion, i.e.,
\begin{equation}\label{eq:6.1}
\mathcal{E}_{\varepsilon}(x,y)
:= \nabla_x\nabla_y G_\varepsilon(x,y)
+\nabla\Phi_{\varepsilon,i}(x)
\nabla\Phi_{\varepsilon,j}^{*}(y)\partial_{ij}\overline{G}(x,y),
\end{equation}
where the effective Green function $\overline{G}$ is determined by  $-\nabla\cdot\bar{a}\nabla \overline{G}(\cdot,0) = \delta_0$ in $\Omega$ with zero-Dirichlet boundary conditions.
Then, for any $p\in[1,\infty)$, there holds the annealed decay estimate
\begin{equation}\label{pri:6.1}
 |x-y|^{d+1}
 \big\langle|\mathcal{E}_{\varepsilon}(x,y)|^p\big\rangle^{\frac{1}{p}}
 \lesssim_{\lambda,\lambda_1,\lambda_2,d,\Omega,\tau,p} \varepsilon\mu_d^2(R_0/\varepsilon)\ln(|x-y|/\varepsilon+2)
\end{equation}
for any $x,y\in\Omega$ with $|x-y|\geq 2\varepsilon$.
\end{theorem}

A few remarks on theorems are in order.

\begin{remark}
\emph{To the authors' best knowledge,
with the help of boundary correctors originally developed by Avellaneda and Lin \cite{Avellaneda-Lin87} for quantitative periodic homogenization theory, Fischer and Raithel \cite{Fischer-Raithel17} first derived the intrinsic large-scale regularity theory for random elliptic operators on the half-space $\mathbb{R}^d_{+}$, inspired by Gloria, Neukamm, and Otto \cite{Gloria-Neukamm-Otto20}, and the estimate similar to $\eqref{pri:lip}$
was known as the mean-value property (see \cite[Theorem 2]{Fischer-Raithel17}).
Under the finite range condition on the ensemble $\langle\cdot\rangle$,
the estimate $\eqref{pri:lip}$ has already been shown by
Armstrong, Kuusi, and Mourrat \cite[Theorem 3.26]{Armstrong-Kuusi-Mourrat19}, traced back to Armstrong and Smart \cite{Armstrong-Smart16}. Basically, the methods mentioned above are different in nature. With the advantage of the boundary corrector, Fischer and Raithel's method does not rely on
any quantified ergodicity assumption, while the approach developed by  Armstrong and Smart \cite{Armstrong-Kuusi-Mourrat19,Armstrong-Smart16} starts from quantifying  the approximating error at each of large scales, which is equivalent to quantifying the sublinearity of correctors and therefore requires a  quantified ergodicity like $\eqref{a:2}$ as the precondition. However, a significant advantage of Armstrong and Smart's method is that it can avoid introducing boundary correctors to obtain boundary regularity estimates. So, we took their approach to prove Theorem $\ref{thm:1}$, while the present contribution is to
find a novel form of \emph{minimal radius} allowing
Gloria-Neukamm-Otto's results on the correctors \cite{Gloria-Neukamm-Otto20,Gloria-Neukamm-Otto21} to be further applied.
In addition, our proof is a little more concise, which is mainly due to the simplification of Campanato’s iteration made by Shen \cite[Lemma 8.5]{Shen16} and the fact that we did not pursue the optimal stochastic integrability of the minimum radius presented in Theorem $\ref{thm:1}$.
Meanwhile, we should mention that Armstrong and Kuusi have already upgraded their quantitative theory under very general mixing conditions in \cite{Armstrong-Kuusi22}, compared to the previous monograph \cite{Armstrong-Kuusi-Mourrat19}, which contains the assumption $\eqref{a:2}$. One can rewrite ``line-by-line'' from \cite{Armstrong-Kuusi-Mourrat19} under the assumptions of \cite{Armstrong-Kuusi22} to get Theorem $\ref{thm:1}$ (with optimal
stochastic integrability). However, we point out that
the assumption $\eqref{a:2}$ is not directly
involved in our proofs (except of CLT-scaling type estimates), and it is merely used
as one way to guarantee the almost-surely finiteness of the minimum radius stated in Theorem $\ref{thm:1}$ (different from $r_*$ in \cite[Theorem 1]{Gloria-Neukamm-Otto20} or
$\mathcal{X}$ in \cite[Theorem 1.21]{Armstrong-Kuusi22}). In fact,
it is still unclear for us whether the minimum radius given here
(i.e., $\chi_{**}$ for Lipschitz estimates)
can be weakened into some new one
independent of quantitative assumption on the ensemble $\langle\cdot\rangle$.
%In a certain sense, Theorem $\ref{thm:1}$ manages to fill this gap for the readers.
}
\end{remark}

%Finally, we mention that although S. Armstrong and his co-authors claimed in their monograph \cite{Armstrong-Kuusi-Mourrat19} that their arguments would be adapted for the ensemble satisfying the spectral gap or log-Sobolev inequalities\footnote{S. Armstrong et. al have further worked in this direction recently (see \cite{Armstrong-Kuusi22}).},
%non-experts may still be concerned about how this can be done easily, since A. Gloria, S. Neukamm, and F. Otto's work seems to follow a completely different set of ideas or languages.

\begin{remark}
\emph{Different from the idea of the massive extended correctors \cite{Gloria-Neukamm-Otto20} employed by Gloria et al. in a sensitive regime,
if taking $\Omega=\square_{R_0}:=(-R_0,R_0)^d$ in $\eqref{pde:10}$,
boundary correctors (defined on any large cube $\square_{mR_0}$ with $\varepsilon=1/m$ therein) have already been used
by Armstrong et al. \cite{Armstrong-Kuusi-Mourrat19} in renormalization arguments,
which start from energy-type quantities
that display a simpler dependence on the coefficients in the sense of  large scales and culminate with the optimal quantitative bounds on the first-order correctors (see \cite[Theorem 4.1]{Armstrong-Kuusi-Mourrat19}). Therefore, roughly speaking, the road map shown in Theorem $\ref{thm:2}$ is the other way around, and is reasonable to
acquire much stronger estimates $\eqref{pri:1.a}$ and $\eqref{pri:1.b}$ compared to the counterparts stated in \cite[Propositions 1.8,1.10]{Armstrong-Kuusi-Mourrat19}. However, the proof of
Theorem $\ref{thm:2}$ is totally deterministic, and the main idea
may be found in periodic homogenization theory (see e.g. \cite{Shen18}).}
\end{remark}

\begin{remark}
\emph{In periodic homogenization theory, boundary correctors are mainly used to eliminate the adverse effects of the boundary layer, playing a key role in deriving boundary regularity estimates and sharp homogenization errors (see e.g. \cite{Avellaneda-Lin87,Kenig-Lin-Shen14,Shen18}).
In stochastic homogenization, their fluctuation estimates pose
a tough problem. The first result analogous to $\eqref{pri:1.3}$ was recently established by Bella, Fischer, Josien, and Raithel \cite{Bella-Fischer-Josien-Raithel-24} for the half-space $\mathbb{R}^d_{+}$. Their method cannot be directly extended to general regions because their proof depends on the property that the boundary corrector is stationary along the tangential direction of the boundary of $\mathbb{R}^d_{+}$, which is not clearly true for the boundary corrector on general domains (even for the case of $d\geq 3$). In short, the estimates $\eqref{pri:1.3}$ and
$\eqref{pri:1.4}$ are based on a type of ``a bootstrap argument''
developed by Bella et al. \cite{Bella-Fischer-Josien-Raithel-24}, which includes three ingredients: an improved Caccioppoli's inequality\footnote{Its idea is similar to that shown in \cite[Lemma 4]{Bella-Fischer-Josien-Raithel-24}, which can be traced back to the contributions of Armstrong et al. \cite{Armstrong-Kuusi-Mourrat19,Armstrong-Kuusi-Mourrat-16} in the multiscale Poincaré inequality.}, one-order improvement, and sensitive estimates (see Section $\ref{sec:6}$). The statement of Theorem $\ref{thm:2*}$ may not surprise experts, while the novelty of Theorem $\ref{thm:2*}$ is more reflected in its proof. Besides, the estimate $\eqref{pri:1.4}$ may be applied to studying artificial boundary condition problems (avoiding imposing the second-order corrector
as in \cite{Lu-Otto-Wang21} for $d=3$), as well as upgrading heterogeneous
multiscale method (HMM) errors
stated in \cite[Theorem 1.3]{E-M-Zhang05}, and it will be addressed in a separate work.}
\end{remark}

\begin{remark}
\emph{The estimate analogous to $\eqref{pri:6.1}$ in a periodic setting was established
by Kenig, Lin, and Shen \cite{Kenig-Lin-Shen14} for a bounded $C^{3,\tau}$ domain, and by Avellaneda and Lin \cite{Avellaneda-Lin91} for the whole space, via a duality argument. This methodology can be extended to aperiodic cases, and we refer the reader to the work by Blanc, Josien, and Le Bris \cite{Blanc-Josien-LeBris-20} for further details on this aspect.
In a random setting, Bella et al. \cite{Bella-Giunti-Otto17} developed a similar result in the quenched sense. Recently, without boundary conditions, the previous conclusions in \cite{Bella-Giunti-Otto17} were extended to
a higher-order expansion in the annealed sense by Clozeau et al. \cite{Clozeau-Josien-Otto-Xu}, where the key ingredient is
a weak norm estimate (see Lemma $\ref{lemma:4}$). Besides,
with the aid of parabolic semigroup, Armstrong et al.
derived optimal convergence of Green functions, and we refer the reader
to \cite[Chapters 8, 9]{Armstrong-Kuusi-Mourrat19} for the details.}
\end{remark}

\subsection{Organization of the paper}

\noindent
In Section $\ref{sec:2}$, we first establish large-scale boundary H\"older estimates in Proposition $\ref{P:1}$ and give the related uniform estimates
in Corollary $\ref{corollary:2.1}$. Then, we discuss large-scale boundary
Lipschitz estimates in Proposition $\ref{P:3}$. Similarly, we state
the related uniform estimates in Corollary $\ref{corollary:2.2}$, which together with Corollary $\ref{corollary:2.1}$ benefits the study of Green functions.

In Section $\ref{sec:3}$, we show some annealed decay estimates of
Green functions in Proposition $\ref{P:4}$, which is due to the related
quenched decay estimates given in Lemma $\ref{lemma:3}$. Besides,
to handle the boundary correctors, we manage to establish the so-called
annealed layer-type estimate in Proposition $\ref{P:2}$, which is based upon the related decay estimates with a distance weight (see Lemmas $\ref{lemma:5}$, $\ref{lemma:3.1}$, and $\ref{lemma:3.2}$).

In Section $\ref{sec:4}$, we study boundary correctors. We first establish an interesting boundary estimate in Lemma $\ref{lemma:4.1}$, which is stronger than energy-type one but weaker than the boundary Lipschitz estimate. Based upon this lemma, we complete the proof of Proposition $\ref{P:5}$ by four steps, which implies all the results of Theorem $\ref{thm:2}$.

In Section $\ref{sec:5}$, we focus on the demonstration of Theorem $\ref{thm:3}$. There are two important ingredients involved. One is
the weak norm estimate $\eqref{pri:27}$ stated in Lemma $\ref{lemma:4}$
(we also call it Bella-Giunti-Otto's lemma),
the other is annealed Calder\'on-Zygmund estimates stated in Lemma $\ref{P:7*}$ whose proof has been shown in \cite{Wang-Xu23-1}.

In Section $\ref{sec:6}$, we introduce the bootstrap method\footnote{This is the only part of this article that involves nondeterministic arguments}. It consists of three parts: the improved Caccioppoli's inequality (see Lemma $\ref{lemma:7.1}$), one-order improvement (see Lemma $\ref{lemma:7.2}$), and sensitive estimates (see Lemma $\ref{lemma:7.3}$).
These ingredients are combined to establish the proof of Theorem $\ref{thm:2*}$.
The paper concludes with an appendix that presents the precise form of scale-invariant estimates (see Lemma $\ref{lemma:8.1}$). These estimates play a fundamental role in the analysis of Sections $\ref{sec:2}$ and $\ref{sec:6}$.

\subsection{Notations}\label{notation}

\begin{enumerate}
  \item Notation for estimates.
\begin{enumerate}
  \item $\lesssim$ and $\gtrsim$ stand for $\leq$ and $\geq$
  up to a multiplicative constant,
  which may depend on some given parameters in the paper,
  but never on $\varepsilon$. The subscript form $\lesssim_{\lambda,\cdots,\lambda_n}$ means that the constant depends only on parameters $\lambda,\cdots,\lambda_n$.
  In addition, we also use superscripts like $\lesssim^{(2.1)}$ to indicate the formula or estimate referenced.
  We write $\sim$ when both $\lesssim$ and $\gtrsim$ hold.
  \item We use $\gg$ instead of $\gtrsim$ to indicate that the multiplicative constant is much larger than 1 (but still finite),
      and it's similarly for $\ll$.
\end{enumerate}
  \item Notation for derivatives.
  \begin{enumerate}
    \item Spatial derivatives:
  $\nabla v = (\partial_1 v, \cdots, \partial_d v)$ is the gradient of $v$, where
  $\partial_i v = \partial v /\partial x_i$ denotes the
  $i^{\text{th}}$ derivative of $v$.
  $\nabla^2 v$  denotes the Hessian matrix of $v$;
  $\nabla\cdot v=\sum_{i=1}^d \partial_i v_i$
  denotes the divergence of $v$, where
  $v = (v_1,\cdots,v_d)$ is a vector-valued function.
   \item $\nabla_x$ denotes
  the gradient with respect to the variable $x$.
  Let $\alpha\in\mathbb{N}^d$ be a multi-index, i.e.,
  $\alpha=(\alpha_1,\alpha_2,\cdots,\alpha_d)$. We call
  the order of the multi-index as $|\alpha|=\sum_{i=1}^{d}\alpha_i$.
  For any $|\alpha|>1$, we define
  $\nabla^{|\alpha|} u=\nabla^\alpha u:=\frac{\partial^{|\alpha|}u}{\partial x_1^{\alpha_1}\cdots\partial x_d^{\alpha_d}}$.
  \end{enumerate}

  \item Geometric notation.
  \begin{enumerate}
  \item $d\geq 2$ is the dimension, and $R_0$ represents the diameter of $\Omega$.
  \item For a ball $B$, we set $B=B_{r_B}(x_B)$ and abusively write
  $\alpha B:=B_{\alpha r_B}(x_B)$.
%    \item Let $B_r(x)$ denote the open ball centered at $x$ of radius $r$. For any ball $B\subset\mathbb{R}^d$,
%        let $x_B$ and $r_B$ represent the center and radius of $B$, respectively.
%        Let $D_{*,\varepsilon}(x):=\Omega\cap B_{*,\varepsilon}(x)$ with $B_{*,\varepsilon}(x):=B_{\varepsilon\chi_{*}(x/\varepsilon)}(x)$, where $\chi_*$ is the minimum radius given in Lemma $\ref{lemma:1}$ or Corollary $\ref{cor:1}$.
%    \item The layer set of $\Omega$ is denoted
%    by $O_{n\varepsilon}
%    :=\{x\in\Omega:\text{dist}(x,\partial\Omega)<n\varepsilon\}$.
%    The co-layer set is defined by $\Sigma_{n\varepsilon}:=
%    \Omega\setminus O_{n\varepsilon}$.
\item For any $x\in\partial\Omega$ we set $D_r(x):= B_r(x)\cap\Omega$ and $\Delta_r(x):= B_r(x)\cap\partial\Omega$, and omit $x$ whenever  $x=0$. For any $x\in\Omega$, we set  $U_r(x):=B_r(x)\cap\Omega$.
\item $O_r:=\{x\in\Omega:\text{dist}(x,\partial\Omega)\leq r\}$
  denotes the layer part of $\Omega$ with $r>0$.
\item $\delta(x):=\text{dist}(x,\partial\Omega)$ represents
the distance from $x$ to $\partial\Omega$.
\item $dS$ is the surface measure on $\partial\Omega$.
  \end{enumerate}

\item Notation for functions.
\begin{enumerate}
    \item The function $\textbf{1}_{E}$ is the indicator function of the region $E$.
\item $a\vee b:=\max\{a,b\}$; and $a\wedge b:=\min\{a,b\}$;
  The notation
``$\lfloor x\rfloor$'' denotes the largest integer not more than $x$, and ``$\lceil x\rceil$'' represents the smallest integer not less than $x$.
\item We denote $F(\cdot/\varepsilon)$ by $F^{\varepsilon}(\cdot)$ for simplicity, and $$(f)_r = \dashint_{B_r} f  = \frac{1}{|B_r|}\int_{B_r} f(x) dx.$$
\item We denote the support of $f$ by $\text{supp}(f)$.
\item The nontangential maximal function of $u$ is defined by
\begin{equation}\label{mf}
(u)^{*}(Q):=\sup\big\{|u(x)|:x\in\Gamma_{N_{0}}(Q)\big\}
\qquad\forall Q\in\partial\Omega,
\end{equation}
  \end{enumerate}
where $\Gamma_{N_{0}}(Q):=\{x\in\Omega:|x-x_0|\leq N_{0}
\text{dist}(x,\partial\Omega)\}$ is the cone with vertex $Q$
and aperture $N_{0}$, and $N_{0}>1$ may be sufficiently large.
\end{enumerate}

Finally, we mention that:
(1)
When we say the multiplicative constant depends on $\partial\Omega$,
it means that
the constant relies on the character of the domain
(determined by a certain norm of functions
that locally characterize $\partial\Omega$. See e.g. $\eqref{boundary-1}$); (2) If we say the multiplicative constant depends on $\Omega$,
it means that the constant depends not only on $\partial\Omega$ but also on the diameter $R_0$. In particular,
if $\partial\Omega$ is uniformly regular (see the beginning of Section $\ref{sec:2}$),
it will rely on some scaling-invariant quantities such as $R_0^{\tau}M_0$ and $R_0^{k-1+\tau}M_0$ (see Lemmas $\ref{lemma:8.1}$,$\ref{lemma:8.2}$ for the details);
(3) The estimate of the pair $(X,Y)$ is equivalent to
the individual estimate of $X$ and $Y$.

\section{Quenched boundary estimates}\label{sec:2}

\noindent
We first introduce the term of uniform boundary regularity (or see e.g. \cite[pp.84]{Adams-03}). Let $k\geq 1$ be an integer, and $0\leq \tau\leq 1$.
For any ball $B$ with $x_B\in\partial\Omega$, let $D=B\cap\Omega$
and $\Delta=B\cap\partial\Omega$. There exists a $C^{k,\tau}$ function $\psi$ (up to
rotating and translating a standard Euclidean coordinate) such that $D=\{(x',x_d)\in\mathbb{R}^d:x'\in\mathbb{R}^{d-1},~x_d>\psi(x')\}\cap B_{r_B}(0)$,
where $\psi$ satisfies
\begin{equation}\label{boundary-1}
  \psi(0)=0,
%  \quad
%  \nabla\psi(0)=0,
  \quad
  \|\nabla\psi\|_{L^\infty(\mathbb{R}^{d-1})}
  \leq m_0,
  \quad
  \text{and}
  \quad
  [\nabla^k\psi]_{C^{0,\tau}(\mathbb{R}^{d-1})}\leq M_0.
\end{equation}
%$\psi(0)=0$, $\nabla\psi(0)=0$, and $[\nabla\psi]_{C^{0,\tau}(\mathbb{R}^{d-1})}\leq M_0$.

If the constants $(m_0,M_0)$ are independent of the choice of the ball $B$, then
$\Omega$ is called a uniform $C^{k,\tau}$ domain.
By interpolation inequalities (see e.g. \cite[Lemma 6.34]{Gilbarg-Trudinger-01}), there exists a constant $C$, depending only on $d,k,\tau$,  such that
$\|\psi\|_{C^{k,\tau}(\mathbb{R}^{d-1})}\leq C(m_0+M_0)$.
In particular, when $k=1$ and $\tau = 0$, we correspondingly refer to it as a uniform Lipschitz or $C^1$  region.
To characterize the $C^1$ property, we additionally introduce
\begin{equation}\label{boundary-2}
|\nabla\psi(x')-\nabla\psi(y')|\leq \omega(|x'-y'|)
\qquad \forall x',y'\in\mathbb{R}^{d-1},
\end{equation}
where $\omega$ is a fixed nondecreasing continuous function on $[0,\infty)$ with $\omega(0)=0$. We mention that a bounded $C^{1,\tau}$ region with
the uniform interior ball condition is the uniform $C^{1,\tau}$ domain.

\subsection{H\"older estimates}
\begin{proposition}\label{P:1}
Let $\Omega\subset\mathbb{R}^d$ (with $d\geq 2$) be a bounded $C^1$ domain.
Let $0<\varepsilon\leq 1$ and $0<\alpha<1$. Suppose that $\langle\cdot\rangle$ satisfies the spectral gap condition $\eqref{a:2}$ and
the admissible coefficients satisfy
the regularity condition $\eqref{a:3}$. Given
$x\in\partial\Omega$ and $R\geq 1$, let $u_\varepsilon$ be a weak solution to
\begin{equation}\label{pde:1}
\left\{\begin{aligned}
\mathcal{L}_\varepsilon(u_\varepsilon) &= F
&\quad&\text{in}~~ D_{2R}(x);\\
u_\varepsilon &= 0
&\quad&\text{on}~~ \Delta_{2R}(x).
\end{aligned}\right.
\end{equation}
Then, there exists a stationary random field $\chi_*\geq 1$, satisfying  $\langle\chi_*^\beta\rangle\lesssim_{\lambda,\lambda_1,\lambda_2,d,\beta} 1$
for any $\beta\in[1,\infty)$, such that for all $\varepsilon\leq \varepsilon_{1}:=\big(\inf_{y\in\mathbb{R}^d}
      \inf_{a}\chi_{*}(a,y)\big)^{-1}$, there holds
\begin{equation}\label{pri:2.1}
\Big(\dashint_{D_r(x)}|\nabla u_\varepsilon|^2
\Big)^{1/2}
\lesssim_{\lambda,d,\alpha,\partial\Omega,q} \Big(\frac{r}{R}\Big)^{\alpha-1}
\bigg\{\Big(\dashint_{D_R(x)}|\nabla u_\varepsilon|^2
\Big)^{\frac{1}{2}}
+ R^{1-\frac{d}{q}}\|F\|_{L^q(D_R(x))}\bigg\}
\end{equation}
for any $r\in[\varepsilon\chi_*(x/\varepsilon), R]$ and
$q>\frac{d}{2}$.
\end{proposition}

\begin{lemma}[extended correctors \cite{Gloria-Neukamm-Otto20}]\label{lemma:*3}
Assume that $\langle\cdot\rangle$ satisfies stationary and ergodic\footnote{In terms of ``ergodic'', we refer the reader to \cite[pp.105]{Gloria-Neukamm-Otto20} or \cite[pp.222-225]{Jikov-Kozlov-Oleinik94} for the definition.} conditions.
Then, there exist two random  fields $\{\phi_i\}_{i=1}^{d}$
and $\{\sigma_{ijk}\}_{i,j,k=1}^{d}$ with the following
properties: the gradient fields $\nabla\phi_i$ and
$\nabla\sigma_{ijk}$ are stationary in the sense of
$\nabla\phi_{i}(a;x+z)=\nabla\phi_i(a(\cdot+z);x)$ for any shift
vector $z\in\mathbb{R}^d$ (and likewise for $\nabla\sigma_{ijk}$).
Moreover, $(\phi,\sigma)$, which are called extended correctors, satisfy
\begin{equation*}%\label{eq:*2.1}
\begin{aligned}
\nabla\cdot a(\nabla\phi_i+e_i) &= 0;\\
\nabla\cdot\sigma_i &= q_i:= a(\nabla\phi_i+e_i)-\bar{a}e_i;\\
-\Delta\sigma_{ijk} &= \partial_j q_{ik} - \partial_k q_{ij},
\end{aligned}
\end{equation*}
in distributional sense on $\mathbb{R}^d$ with
$\bar{a}e_i:=\big\langle a(\nabla\phi_i +e_i)\big\rangle$ and
$(\nabla\cdot\sigma_i)_j:=\partial_k\sigma_{ijk}$. Also,
the field $\sigma$ is skew-symmetric in its last indices, that is
$\sigma_{ijk} = -\sigma_{ikj}$.
%\begin{equation}\label{}
% \sigma_{ijk} = -\sigma_{ikj}.
%\end{equation}
Furthermore, if $\langle\cdot\rangle$ additionally satisfies
the spectral gap condition $\eqref{a:2}$, as well as $\eqref{a:3}$.
Let $\mu_d$ be given as in $\eqref{pri:**3}$.
For any $p\in[1,\infty)$, there hold $\langle|\nabla\phi|^p\rangle^{\frac{1}{p}}
\lesssim_{\lambda,\lambda_1,\lambda_2,d,p} 1$,
\begin{equation}\label{pri:*2}
\Big\langle\big|\int_{\mathbb{R}^d}\nabla (\phi,\sigma)\cdot g\big|^{2p}\Big\rangle^{\frac{1}{p}}
\lesssim_{\lambda,\lambda_1,\lambda_2,d,p} \int_{\mathbb{R}^d}|g|^2
\end{equation}
for all deterministic vector fields $g$, and
\begin{equation*}
 \Big\langle\big|(\phi,\sigma)(x)
 -(\phi,\sigma)(0)\big|^p \Big\rangle^{\frac{1}{p}}
 \lesssim_{\lambda,\lambda_1,\lambda_2,d,p}
 \left\{\begin{aligned}
 &(1+|x|)^{\frac{1}{2}} &\quad& d=1;\\
 &\mu_d(x) &\quad& d\geq 2.
 \end{aligned}\right.
\end{equation*}
\end{lemma}

\begin{proof}
See \cite[Lemma 1]{Gloria-Neukamm-Otto20} coupled with \cite[Proposition 4.1]{Josien-Otto22}.
\end{proof}

\begin{lemma}[qualitative error]\label{lemma:1}
Let $\Omega$ be a bounded Lipschitz domain, and $\varepsilon\in(0,1]$.
Assume the same assumptions as in Proposition
$\ref{P:1}$. Let $p\in(1,\infty)$, $\gamma\in[1,\infty]$ and $p',\gamma'$ be the associated conjugate index satisfying
$0<p'-1\ll 1$ and $0<\gamma'-1\ll 1$. Then, for any $\nu>0$, there exist $\theta\in (0,1)$, depending on $\lambda,d,\gamma,p,m_0$, and $\nu$, and the stationary random field:
\begin{equation}\label{key-1}
\chi_{*}(0):=
 \inf\Big\{l\geq 1:
 \Big(\frac{1}{R}\Big)^{\frac{1}{\gamma p'}}
 \Big(\dashint_{B_{2R}}
 |(\phi,\sigma)-(\phi,\sigma)_{2R}|^{2p}\Big)^{\frac{1}{p}}\leq \theta, ~\forall R\geq l\Big\}\vee \theta^{-p},
\end{equation}
such that
%$\langle \chi_*^\beta\rangle\lesssim_{\lambda,\lambda_1,\lambda_2,d,\beta} 1$ for any $\beta<\infty
%$,
%\begin{equation}\label{key:1}
%\Big(\frac{1}{R}\Big)^{\frac{1}{\gamma p'}}
% \Big(\dashint_{B_{2R}}
% |(\phi,\sigma)-(\phi,\sigma)_{2R}|^{2p}\Big)^{\frac{1}{p}}\leq \theta
% \qquad \forall R\geq \chi_*,
%\end{equation}
for any solution $u_\varepsilon$ of $\eqref{pde:1}$ with $x=0$, there exists $\bar{u}_r\in H^1(D_{2r})$ satisfying
$\nabla\cdot \bar{a} \nabla \bar{u}_r = F$ in $D_{2r}$
with $\bar{u}_r = u_\varepsilon$ on $\partial D_{2r}$, such that
\begin{equation}\label{pri:1}
\dashint_{D_r}|u_\varepsilon - \bar{u}_r|^2 \leq \nu^2
\Big\{
 \dashint_{D_{2r}}|u_\varepsilon|^2
+ r^2\Big(\dashint_{D_{2r}}|F|^{2p'}\Big)^{\frac{1}{2p'}}
\Big\}
\end{equation}
holds for any $r\geq \varepsilon\chi_*(0)$. Moreover,
we have $\langle\chi_*^\beta\rangle\lesssim_{\lambda,\lambda_1,\lambda_2,d,\beta} 1$
for any $\beta\in[1,\infty)$.
\end{lemma}

\begin{remark}
\emph{If the admissible coefficient $a$ is symmetry, we can choose $\gamma=1$. The proof of Lemma $\ref{lemma:1}$ is similar to that given in \cite[Lemma 2.2]{Wang-Xu22}, so is omitted here. Based upon
the approximating error $\eqref{pri:1}$ at large scales,
it is not hard to
show the proof of Proposition $\ref{P:1}$ by an analogous argument
developed in \cite[Theorem 5.1]{Shen16}, and is therefore left to the reader.}
\end{remark}

%\medskip
%\textbf{Proof of Proposition $\ref{P:1}$.}
%The existence of the minimal radius $\chi_*$ is given by $\eqref{key-1}$, and its $\beta$-th moment estimate consequently follows from the fluctuation estimate $\eqref{pri:3}$. We first show
%the estimate $\eqref{pri:2.1}$ for the case $x=0$. On account of the estimate $\eqref{pri:1}$ and boundary H\"older regularity of $\bar{u}_r$, we can derive the desired estimate $\eqref{pri:2.1}$, and a rigorous proof can be found in \cite[Theorem 1]{Wang-Xu23-2}. For the case $x\not=0$, by
%rescaling argument, it is not hard to see the corresponding estimate $\eqref{pri:1}$ holds for $r\geq \varepsilon\chi_*(x/\varepsilon)$,
%and the rest of the proof is the same as the previous proof.

\begin{corollary}[Uniform H\"older estimates]\label{corollary:2.1}
Let $\alpha\in(0,1)$ and $\chi_*$ with $\varepsilon_1$ be given as in
Proposition $\ref{P:1}$.
Assume the same conditions as in Proposition $\ref{P:1}$. Let $u_\varepsilon$ be a weak solution to $\eqref{pde:1}$
with given $x\in\partial\Omega$ and $R\geq 1$. Then, there exists a stationary random field given by
$\mathcal{H}_*(z):=[a]_{C^{0,\sigma}(B_1(z/\varepsilon))}^{
\frac{1}{\sigma}(\frac{d}{2}+1)}
\chi_{*}^{\frac{d}{2}}(z/\varepsilon)$ with $z\in\mathbb{R}^d$
and $\varepsilon\in(0,\varepsilon_1]$,
such that
\begin{equation}\label{pri:3}
\Big(\dashint_{D_r(x)}|\nabla u_\varepsilon|^2
\Big)^{\frac{1}{2}}
\lesssim_{\lambda,d,\alpha,\partial\Omega,q}
\mathcal{H}_*(x) \Big(\frac{r}{R}\Big)^{\alpha-1}
\bigg\{
\Big(\dashint_{D_R(x)}|\nabla u_\varepsilon|^2
\Big)^{\frac{1}{2}}
+R\Big(\dashint_{D_R(x)}|F|^q
\Big)^{\frac{1}{q}}
\bigg\}
\end{equation}
holds for any $0<r\leq R$ and $q>(d/2)$. Also, for any $\beta\in[1,\infty)$,  we have
\begin{equation}\label{random-1}
 \big\langle\mathcal{H}_{*}^{\beta}(z)\big\rangle^{1/\beta}
 \lesssim_{\lambda,\lambda_1,\lambda_2,d,\beta} 1
 \qquad \forall z\in\mathbb{R}^d.
\end{equation}
Moreover, if $\Omega$ is the bounded uniform $C^1$ domain and $\varepsilon\in(0,1]$,
for any $B$ with $B\subset D_{R}$ or $x_B\in\Delta_R$, one can obtain that
%suppose that $u_\varepsilon$ is a solution to
%the equations $\mathcal{L}_\varepsilon(u_\varepsilon) = 0$ in
%$B_r(x)\cap\Omega$
%with $u_\varepsilon = 0$ on $B_r(x)\cap\partial\Omega$, then
\begin{equation}\label{pri:2.6}
\begin{aligned}
&|u_\varepsilon(x_B)|
\lesssim
\mathcal{H}_*(x_B)\bigg\{\Big(\dashint_{B\cap\Omega}|u_\varepsilon|^2
\Big)^{1/2}
+ r_B^2\Big(\dashint_{B\cap\Omega}|F|^q
\Big)^{1/q}\bigg\};\quad \text{and}\\
&[u_\varepsilon]_{C^{0,\alpha}(x_B)}
\lesssim r_B^{1-\alpha}\mathcal{H}_*(x_B)
\bigg\{\Big(\dashint_{B\cap\Omega}|\nabla u_\varepsilon|^2
\Big)^{1/2}
+ r_B\Big(\dashint_{B\cap\Omega}|F|^q
\Big)^{1/q}\bigg\},
\end{aligned}
\end{equation}
where $[u_\varepsilon]_{C^{0,\alpha}(x_B)}$ is the $\alpha$-H\"older
coefficient of $u_\varepsilon$ at $x_B\in\Delta_R$ with respect to $B\cap\Omega$ (see e.g. \cite[pp.52]{Gilbarg-Trudinger-01}). The multiplicative constants depend on
$\lambda, d$, $m_0$, $\omega(t)$ in $\eqref{boundary-2}$,  and $\alpha$ at most.
\end{corollary}

\begin{remark}\label{remark:2.1}
\emph{$\mathcal{H}_*$ given in Corollary $\ref{corollary:2.1}$ is a
scaling-invariant quantity\footnote{To see this, let $x=Ry$ with $x\in D_R$
and $\varepsilon'=\varepsilon/R$. One can observe that
$y/\varepsilon' = x/\varepsilon$ and therefore
$\mathcal{H}_{*}(y)=\mathcal{H}_{*}(x)$.
}, and so are the estimated constants in
the estimates $\eqref{pri:2.6}$.
}
\end{remark}

%$[a]_{C^{0,\sigma}(B_1(y/\varepsilon'))}
%\chi_{*}(y/\varepsilon')=
%[a]_{C^{0,\sigma}(B_1(x/\varepsilon))}
%\chi_{*}(x/\varepsilon)$
%Set $\tilde{u}_{\varepsilon'}(y)= u_\varepsilon(Ry)=u_\varepsilon(x)$, we have $\nabla\cdot a(\cdot/\varepsilon')\nabla \tilde{u}_{\varepsilon'} = 0$ in $\tilde{D}_1$ and
%$\tilde{u}_{\varepsilon'} = 0$ on $\tilde{\Delta}_1$, where the
%boundary function $y_d=\tilde{\psi}_{R}(y_1,\cdots,y_{d-1})
%=:\frac{1}{R}\psi(Ry_1,\cdots,Ry_{d-1})$ and
%$|(y_1,\cdots,y_{d-1})|\leq 1$.
%Applying the estimate $\eqref{pri:2.6}$
%to $\tilde{u}_{\varepsilon'}$, the estimated constant is given by
%$H_{*}(y)$

\begin{proof}
It suffices to show $\eqref{pri:3}$ and $\eqref{pri:2.6}$ in the case of $\varepsilon\in(0,\varepsilon_1]$, since the estimate $\eqref{pri:2.6}$ in
the case of $\varepsilon\in(\varepsilon_1,1]$ can be absorbed into
the former one by a blow up argument.
The arguments to the estimate $\eqref{pri:3}$
can be divided into three parts according to the scales:
For the ease of statement, it is fine to assume $F=0$.

Case 1: $\varepsilon\chi_*(x/\varepsilon)\leq r\leq R$ is for the large scales, which had done in $\eqref{pri:2.1}$;

Case 2: $\varepsilon\leq r<\varepsilon\chi_*(x/\varepsilon)$ is for the middle scales. We have
\begin{equation}\label{f:2.27}
\begin{aligned}
 \big[\chi_*(\frac{x}{\varepsilon})\big]^{-d}
& \dashint_{D_r(x)}|\nabla u_\varepsilon|^2
 \lesssim_{d} \Big(\frac{\varepsilon}{r}\Big)^{d}
\dashint_{D_{\varepsilon\chi_*(x/\varepsilon)}(x)}|\nabla u_\varepsilon|^2\\
& \lesssim^{\eqref{pri:2.1}}
\Big(\frac{\varepsilon\chi_*(x/\varepsilon)}{R}\Big)^{2(\alpha-1)}
 \dashint_{D_R(x)}|\nabla u_\varepsilon|^2
\lesssim
\Big(\frac{r}{R}\Big)^{2(\alpha-1)}\dashint_{D_R(x)}|\nabla u_\varepsilon|^2.
\end{aligned}
\end{equation}

Case 3: $0<r<\varepsilon$ is for the small scales. By applying classical H\"older estimates (the accurate dependence of the coefficient $a$ can be found by a covering technique (see e.g. \cite[Lemma A.3]{Josien-Otto22})), and Proposition $\ref{P:1}$, we obtain
\begin{equation*}
\begin{aligned}
 \dashint_{D_r(x)}|\nabla u_\varepsilon|^2
& \lesssim_{d} r^{2\alpha-2}[u_\varepsilon]_{C^{0,\alpha}(D_{\varepsilon}(x))}^2
\lesssim (r/\varepsilon)^{2\alpha-2}
[a]_{C^{0,\sigma}(B_1(x/\varepsilon))}^{\frac{2}{\sigma}(\frac{d}{2}+1)}
\dashint_{D_{2\varepsilon}(x)}|\nabla u_\varepsilon|^2\\
&\lesssim^{\eqref{f:2.27}}
[a]_{C^{0,\sigma}(B_1(x/\varepsilon))}^{\frac{2}{\sigma}(\frac{d}{2}+1)}
\big[\chi_*(x/\varepsilon)\big]^d
\Big(\frac{r}{R}\Big)^{2(\alpha-1)}\dashint_{D_R}|\nabla u_\varepsilon|^2,
\end{aligned}
\end{equation*}
where from the second inequality above, the estimated constant begins to depend on $\lambda, d$, $M_0$, $\omega(t)$ in $\eqref{boundary-2}$,  and $\alpha$ at most.

Combining the three cases, and taking $\mathcal{H}_*(z):=[a]_{C^{0,\sigma}(B_1(z/\varepsilon))}^{\frac{1}{\sigma}(\frac{d}{2}+1)}
\chi_{*}^{\frac{d}{2}}(z/\varepsilon)$ finally
leads to the desired estimate $\eqref{pri:3}$ in the case of $\varepsilon\in(0,\varepsilon_1]$. For the case of $\varepsilon\in(\varepsilon_1,1]$,
it is not hard to observe that
by setting $\alpha_{\varepsilon}:=\varepsilon_1/\varepsilon$ it concludes that $\alpha_{\varepsilon}\in[\varepsilon_1,1)$ and  $a(x/\varepsilon)=a(\alpha_{\varepsilon}x/\varepsilon_1)
=:\tilde{a}(x/\varepsilon_1)$. Thus, one can appeal
to the former case by taking
$\mathcal{H}_*(\tilde{z}):=
[\tilde{a}]_{C^{0,\sigma}(B_1(\tilde{z}/\varepsilon_1))
}^{\frac{1}{\sigma}(\frac{d}{2}+1)}
\chi_{*}^{\frac{d}{2}}(\tilde{z}/\varepsilon_1)$ with
$\tilde{z}=\alpha_\varepsilon z$. On the other hand,
we note that $[\tilde{a}]_{C^{0,\sigma}(B_{1}(\tilde{z}/\varepsilon_1))}
= \alpha_\varepsilon^\sigma[a]_{C^{0,\sigma}
(B_{\alpha_\varepsilon}(z/\varepsilon))}
\leq [a]_{C^{0,\sigma}(B_{1}(z/\varepsilon))}$, and this leads to
the desired result.
Concerned with the stated estimate $\eqref{pri:2.6}$,  we can
employ Morrey's theorem (see e.g.
\cite[Theorem 7.19]{Gilbarg-Trudinger-01}) together with the estimate $\eqref{pri:3}$
to achieve $\eqref{pri:2.6}$.
%and \textcolor[rgb]{1.00,0.00,0.00}{the multiplicative constant
%in $\eqref{pri:3}$ will additionally rely on $\varepsilon_1$.}
%As a result, the case $R\geq \varepsilon$ follows from a rescaling argument, while the case $0<R<\varepsilon$ can be derived from the classical regularity estimates.
Finally, the estimate $\eqref{random-1}$ follows from
the assumption $\eqref{a:3}$ and Proposition $\ref{P:1}$.
\end{proof}

\subsection{Lipschitz estimates}

\begin{proposition}\label{P:3}
Let $\Omega$ be a bounded $C^{1,\tau}$ domain with $\tau\in(0,1]$, and
$\varepsilon\in(0,1]$. Suppose that $\langle\cdot\rangle$ satisfies the spectral gap condition $\eqref{a:2}$
and the admissible coefficients satisfy the regularity condition $\eqref{a:3}$. Let $u_\varepsilon$ be a weak solution to $\eqref{pde:1}$ with
$x=0$, $F=0$, and $R\geq 1$ therein.
Then, there exists a stationary random field $\chi_{**}\geq 1$ with $\langle\chi_{**}^\beta\rangle\lesssim_{\lambda,\lambda_1,\lambda_2,d,\beta} 1$ for any $\beta\in[1,\infty)$, such that for all $\varepsilon\leq \varepsilon_{2}:=\big(\inf_{x\in\mathbb{R}^d}
      \inf_{a}\chi_{**}(a,x)\big)^{-1}$ we have
\begin{equation}\label{pri:2.4}
\Big(\dashint_{D_r}|\nabla u_\varepsilon|^2
\Big)^{1/2}
\lesssim_{\lambda,d,\Omega,\tau}
\Big(\dashint_{D_R}|\nabla u_\varepsilon|^2
\Big)^{1/2}
\end{equation}
for any $r\in[\varepsilon\chi_{**}(0),R]$.
\end{proposition}

\begin{lemma}[quantitative error]\label{lemma:2}
Let $\Omega$ be a bounded Lipschitz domain, and $\varepsilon\in(0,1]$.
Assume the same conditions on the ensemble $\langle\cdot\rangle$ as those in Proposition
$\ref{P:3}$. Let $p$ with $\gamma$ be given as in Lemma $\ref{lemma:1}$, and additionally satisfies $\gamma>\frac{1}{4(p'-1)}$ and $\sigma_0:=1/(4\gamma p')$.
Then, there exists a stationary random field
\begin{equation}\label{key-2}
  c_*(0):=\sup_{R\geq 1} R^{-2\sigma_0}\Big(\dashint_{B_{2R}}
  |(\phi,\sigma)-(\phi,\sigma)_{2R}|^{2p}\Big)^{\frac{1}{p}}\vee 1,
\end{equation}
 with $\langle c_{*}^\beta\rangle
\lesssim_{\lambda,\lambda_1,\lambda_2,d,\beta}1$
for any $\beta\in[1,\infty)$, such that for any
solution $u_\varepsilon$ to $\eqref{pde:1}$ with $x=0$ and $F=0$ therein,
there exists the approximating function $\bar{u}_r\in H^1(D_{2r})$ satisfying $\nabla\cdot \bar{a} \nabla \bar{u}_r = 0$ in $D_{2r}$
with $\bar{u}_r = u_\varepsilon$ on $\partial D_{2r}$, such that
\begin{equation}\label{pri:2.3}
\dashint_{D_r}|u_\varepsilon - \bar{u}_r|^2
\lesssim_{\lambda,d,p,\gamma,m_0}
 c_*(0)\Big(\frac{\varepsilon}{r}
 \Big)^{2\sigma_0}\dashint_{D_{2r}}|u_\varepsilon|^2
\end{equation}
holds for any $r\geq \varepsilon$, where $m_0$ as stated in $\eqref{boundary-1}$ denotes the (local) Lipschitz
character of $\Omega$.
\end{lemma}

\begin{proof}
The proof is similar to that given in \cite[Lemma 2.2]{Wang-Xu22},
and is therefore omitted.
\end{proof}

\medskip

\textbf{Proof of Proposition $\ref{P:3}$.}
Let $0<\alpha\leq \tau$ and $0<\vartheta\leq 1/4$, and
$\chi_{**}$ is the stationary random field whose expression is
determined by
$\chi_{**}:=(4C_{\vartheta}c_*)^{1/\sigma_0}$,
where $\sigma_0$ together with $c_*$
is given as in Lemma $\ref{lemma:2}$, and $C_{\vartheta}$
depends on $\lambda,d,p,\gamma,m_0$, and $\vartheta$.
The proof has been divided into three steps.

%Let $0<\alpha<1$ and $0<\varepsilon<\varepsilon_0$ (where $\varepsilon_0$
%is given in Lemma $\ref{lemma:2}$, and the case $\varepsilon\in(\varepsilon_0,1]$ is trivially reduced to the classical Schauder theory by rescaling arguments).
%Let $\chi_{**}$ be the stationary random field whose expression will be
%determined later on, and for any fixed $r$ satisfying $\varepsilon\chi_{**}\leq r\leq (1/4)$. The proof has been divided into
%three steps.

\textbf{Step 1.} Establish an iteration inequality.
We first construct
the approximating solution $\bar{u}_r$ to $\nabla\cdot\bar{a}\nabla \bar{u}_r = 0$ in $D_{2r}$ with $\bar{u}_r = u_\varepsilon$
on $\partial D_{2r}$.
For simplicity, we introduce the following
notation
\begin{equation*}
G(t;v):= \frac{1}{t}\inf_{M\in\mathbb{R}^{d\times d}}\Big(\dashint_{D_t}|v-Mx|^2\Big)^{\frac{1}{2}},
\end{equation*}
and let $M_t$ be such that $G(t;v) = \big(\dashint_{D_t}|v-M_tx|^2\big)^{\frac{1}{2}}$.
Then, by the boundary Schauder estimate for elliptic operator with constant coefficients (see e.g. \cite[Theorem 5.21]{Giaquinta-Martinazzi12} or
\cite[Theorem 4.16]{Gilbarg-Trudinger-01}), for any $0<\vartheta\leq1/4$, we have
\begin{equation}\label{f:2.9}
\begin{aligned}
G(\vartheta r;\bar{u}_r)
\leq^{\eqref{pri:8.1*}} C_\alpha \vartheta^{\alpha} G(r;\bar{u}_r),
\end{aligned}
\end{equation}
where $C_\alpha$ depends on $\lambda,d,\tau,\alpha$, and $\Omega$\footnote{When $R \leq 1$, the constant $C_{\alpha}$ on $\Omega$ only depends on the regularity of the local boundary of $\Omega$. However, as $R$ may approach the scale of $\Omega$, the constant $C_{\alpha}$ on $\Omega$ also acquires a dependence on the size of the domain, specifically on $R_0$. To facilitate scaling arguments in subsequent sections, we provide in the appendix a more precise characterization of how this estimate depends on $\Omega$, quantified in terms of
$R_0^{\tau}M_0$ (see Lemma $\ref{lemma:8.1}$).}.
Let $\epsilon:= c_*(0)(\varepsilon/r)^{\sigma_0}$. Appealing to
the approximating results stated in Lemma $\ref{lemma:2}$, there holds
\begin{equation*}
\begin{aligned}
 G(\vartheta r;u_\varepsilon)
&\leq \frac{1}{\vartheta r}
\Big(\dashint_{D_{\vartheta r}}|u_\varepsilon-\bar{u}_r|^2\Big)^{\frac{1}{2}}
+ G(\vartheta r;\bar{u}_r)\\
&\lesssim_d\vartheta^{-\frac{d}{2}-1}\frac{1}{r}
\Big(\dashint_{D_{r}}|u_\varepsilon-\bar{u}_r|^2\Big)^{\frac{1}{2}}
+ G(\vartheta r;\bar{u}_r)\\
&\leq^{\eqref{pri:2.3},\eqref{f:2.9}} C_{\vartheta}\epsilon
\Big(\dashint_{D_{2r}}|u_\varepsilon|^2\Big)^{\frac{1}{2}}
+ C_\alpha\vartheta^\alpha G(r;\bar{u}_r).
\end{aligned}
\end{equation*}
By using the triangle inequality and $\eqref{pri:2.3}$ again,
we obtain the iteration inequality
\begin{equation}\label{f:33}
\begin{aligned}
G(\vartheta r;u_\varepsilon)
%&\leq
%C_0\vartheta^\alpha G(r;u_\varepsilon)+
%C_{\vartheta}\frac{\epsilon}{2r}
%\Big(\dashint_{D_{2r}}|u_\varepsilon|^2\Big)^{\frac{1}{2}} \\
&\leq
C_\alpha\vartheta^\alpha G(r;u_\varepsilon)+
C_{\vartheta}\epsilon
\Big\{G(2r;u_\varepsilon)+|M_{2r}|\Big\}.
\end{aligned}
\end{equation}

\textbf{Step 2.} Control the slop $|M_{2r}|$ in $\eqref{f:33}$, i.e.,
for any $\rho\in[\varepsilon\chi_{**},R/2]$, we claim that
\begin{equation}\label{f:2.15}
  |M_{\rho}|
  \lesssim_{d} \bigg\{
   |M_R| + G(R,u_\varepsilon)
+  \int_{\rho/2}^{R}\frac{G(r,u_\varepsilon)}{r}dr\bigg\}.
\end{equation}

We start from the mean-value property of linear functions.
For any $r\leq s,t\leq 2r$, there holds
\begin{equation*}
 |M_s-M_t|\lesssim_{d} \frac{1}{r}
   \Big(\dashint_{D_{r}}
   |M_{s}x-M_{t}x|^2\Big)^{\frac{1}{2}}
   \lesssim G(2r;u_\varepsilon),
\end{equation*}
which further implies
\begin{equation}\label{f:2.10}
\sup_{r\leq s,t\leq 2r}\big||M_s|-|M_t|\big|
\lesssim G(2r;u_\varepsilon).
\end{equation}
Then, let $\rho\in [\varepsilon\chi_{**},R/2]$, and it follows from
$\eqref{f:2.10}$ that
\begin{equation*}
\begin{aligned}
\int_{\rho}^{R/2}\frac{|M_r|}{r}dr
&\leq \int_{\rho}^{R/2}\frac{|M_{2r}|}{r}dr
+ C\int_{\rho}^{R/2}\frac{G(2r;u_\varepsilon)}{r}dr,
\end{aligned}
\end{equation*}
where $C$ depends only on $d$.
By changing variables, the above inequality
gives us
\begin{equation}\label{f:2.11}
\begin{aligned}
\int_{\rho}^{2\rho}\frac{|M_r|}{r}dr
&\leq \int_{\frac{R}{2}}^{R}\frac{|M_r|}{r}dr
+ C \int_{2\rho}^{R}\frac{G(r,u_\varepsilon)}{r}dr\\
&\lesssim^{\eqref{f:2.10}} |M_R| + G(R,u_\varepsilon)
+ C \int_{2\rho}^{R}\frac{G(r,u_\varepsilon)}{r}dr.
\end{aligned}
\end{equation}

In light of $\eqref{f:2.10}$, there holds
$|M_{\rho}|\leq |M_r|+CG(2\rho,u_\varepsilon)$
for any fixed $\rho\in [\varepsilon\chi_{**},R/2]$ and $\rho\leq r\leq 2\rho$, and we further derive that
\begin{equation*}
\begin{aligned}
   |M_{\rho}| &\lesssim
   \int_{\rho}^{2\rho}\frac{|M_r|}{r}dr +
   G(2\rho,u_\varepsilon)\\
   &\lesssim^{\eqref{f:2.11}} \bigg\{
   |M_R| + G(R,u_\varepsilon)
+  G(2\rho,u_\varepsilon)
+  \int_{2\rho}^{R}\frac{G(r,u_\varepsilon)}{r}dr\bigg\}.
\end{aligned}
\end{equation*}

This together with the estimate
\begin{equation}\label{f:2.1}
G(2\rho,u_\varepsilon) \lesssim_{d} \int_{2\rho}^{R}\frac{G(r,u_\varepsilon)}{r}dr + G(R,u_\varepsilon)
\end{equation}
leads to the desired estimate $\eqref{f:2.15}$.
As a result, plugging the estimate $\eqref{f:2.15}$ back into
$\eqref{f:33}$, we further arrive at
\begin{equation}\label{f:2.12}
\begin{aligned}
G(\vartheta r;u_\varepsilon)
%&\leq
%C_0\vartheta^\alpha G(r;u_\varepsilon)+
%C\vartheta^{-\frac{d}{2}-1}\epsilon
%\Big\{G(2r;u_\varepsilon)+|M_{2r}|\Big\}\\
&\leq
C_\alpha\vartheta^\alpha G(r;u_\varepsilon)+
+ C_{\vartheta}\epsilon
\bigg\{
   |M_R| + G(R,u_\varepsilon)
+  \int_{r}^{R}\frac{G(s,u_\varepsilon)}{s}ds\bigg\}.
\end{aligned}
\end{equation}

\textbf{Step 3.} Establish the desired estimate $\eqref{pri:2.4}$.
We start from integrating both sides of $\eqref{f:2.12}$ from
$\rho$ to $R$ with respect to the measure $dr/r$,
\begin{equation}\label{f:2.13}
\begin{aligned}
\int_{\rho}^{R}
G(\vartheta r;u_\varepsilon)\frac{dr}{r}
&\leq
C_{\alpha}\vartheta^\alpha
\int_{\rho}^{R}
G(r;u_\varepsilon)\frac{dr}{r}
+ C_\vartheta
\int_{\rho}^{R}\epsilon
\Big\{|M_R| + G(R,u_\varepsilon)\Big\}\frac{dr}{r}\\
&+ C_\vartheta
\int_{\rho}^{R}\epsilon
\int_{r}^{R}\frac{G(s,u_\varepsilon)}{s}ds\frac{dr}{r}.
\end{aligned}
\end{equation}
Then, we first handle the last term in the right-hand side above. Recalling $\epsilon=c_*(0)(\varepsilon/r)^{\sigma_0}$ one can carried out a straightforward computation (and there is no confusion as we omitted $0$ in $c_*(0)$). By using
Fubini's theorem and the condition $\rho\geq \varepsilon\chi_{**}$
in order, we have
\begin{equation}\label{f:2.14}
\begin{aligned}
C_{\vartheta}c_*\varepsilon^{\sigma_{0}}
&\int_{\rho}^{R}
\int_{r}^{R}\frac{G(s,u_\varepsilon)}{s}
ds\frac{dr}{r^{1+\sigma_{0}}}
= C_{\vartheta}c_*\varepsilon^{\sigma_{0}}
\int_{\rho}^{R}\frac{G(s,u_\varepsilon)}{s}ds
\int_{\rho}^{s}\frac{dr}{r^{1+\sigma_{0}}}\\
&\lesssim  C_{\vartheta}c_*\varepsilon^{\sigma_{0}}\rho^{-\sigma_{0}}
\int_{\rho}^{R}\frac{G(r,u_\varepsilon)}{r}dr
\leq   C_{\vartheta}\chi_{**}^{-\sigma_{0}}
\int_{\rho}^{R}\frac{G(r,u_\varepsilon)}{r}dr
\leq\frac{1}{4}
\int_{\rho}^{R}\frac{G(r,u_\varepsilon)}{r}dr,
\end{aligned}
\end{equation}
where we also employ $\chi_{**} = (4C_\vartheta c_*)^{\frac{1}{\sigma_{0}}}$
in the last inequality. Combining the estimates $\eqref{f:2.13}$ and
$\eqref{f:2.14}$, choosing some $\vartheta\in(0,1/4]$ such
that $C_{\alpha}\vartheta^{\alpha}\leq (1/4)$,
and changing variables, one can obtain that
\begin{equation*}
\int_{\vartheta\rho}^{\vartheta R}
G(r;u_\varepsilon)\frac{dr}{r}
\leq \frac{1}{2}\int_{\rho}^{R}\frac{G(r,u_\varepsilon)}{r}dr
+\Big\{|M_R| + G(R,u_\varepsilon)\Big\},
\end{equation*}
which further implies
%\begin{equation}\label{}
%\int_{\vartheta\rho}^{\vartheta/2}
%G(r;u_\varepsilon)\frac{dr}{r}
%\leq \int_{\vartheta/2}^{1/2}\frac{G(r,u_\varepsilon)}{r}dr
%+\Big\{|M_1| + G(1,u_\varepsilon)\Big\},
%\end{equation}
%which implies
\begin{equation}\label{f:2.16}
\begin{aligned}
\int_{\vartheta\rho}^{R}
G(r;u_\varepsilon)\frac{dr}{r}
\leq 2\int_{\vartheta R}^{R}\frac{G(r,u_\varepsilon)}{r}dr
+2\Big\{|M_R| + G(R,u_\varepsilon)\Big\}
\lesssim_{\vartheta} \Big\{|M_R| + G(R,u_\varepsilon)\Big\}.
\end{aligned}
\end{equation}

Consequently, for any $\varepsilon\chi_{**}\leq r\leq R$ there holds
\begin{equation}\label{f:2.17}
G(r;u_\varepsilon)+ |M_r|
\lesssim^{\eqref{f:2.1},\eqref{f:2.15},\eqref{f:2.16}} \Big\{|M_R| + G(R,u_\varepsilon)\Big\}.
\end{equation}
Therefore, it follows from
Caccioppoli's inequality, the estimate $\eqref{f:2.17}$, and Poincar\'e's inequality that
\begin{equation*}
\begin{aligned}
\Big(\dashint_{D_r}|\nabla u_\varepsilon|^2\Big)^{\frac{1}{2}}
&\lesssim \frac{1}{r}\Big(\dashint_{D_r}|u_\varepsilon|^2\Big)^{\frac{1}{2}}
\leq  G(r,u_\varepsilon) + |M_r| \\
&\lesssim^{\eqref{f:2.17}} \Big\{|M_R| + G(R,u_\varepsilon)\Big\}
\lesssim \frac{1}{R}\Big(\dashint_{D_R}|u_\varepsilon|^2\Big)^{\frac{1}{2}}
+|M_R|
\lesssim \Big(\dashint_{D_R}|\nabla u_\varepsilon|^2\Big)^{\frac{1}{2}},
\end{aligned}
\end{equation*}
which leads to the desired estimate $\eqref{pri:2.4}$, and
this completes the whole proof.
\qed

\medskip

\begin{corollary}[uniform Lipschitz estimates]\label{corollary:2.2}
Let $\chi_{**}$, $\varepsilon_2$, $\Omega$, and $\tau$ be given as in
Proposition $\ref{P:3}$. Assume
the ensemble $\langle\cdot\rangle$ satisfies $\eqref{a:2}$ and
$\eqref{a:3}$.
Let $u_\varepsilon$ be a weak solution to $\eqref{pde:1}$
with $F=0$ and $R\geq 1$ therein.
Then, there exists a stationary random field given by
$\mathcal{C}_*(z):=
[a]_{C^{0,\sigma}(B_1(z/\varepsilon))}^{\frac{1}{\sigma}(\frac{d}{2}+1)}
\chi_{**}^{\frac{d}{2}}(z/\varepsilon)$ with $z\in\mathbb{R}^d$
and $\varepsilon\in(0,\varepsilon_2]$,
such that
\begin{equation}\label{pri:20}
\Big(\dashint_{D_r(x)}|\nabla u_\varepsilon|^2
\Big)^{\frac{1}{2}}
\lesssim_{\lambda,d,\Omega,\tau}
\mathcal{C}_*(x)
\Big(\dashint_{D_R(x)}|\nabla u_\varepsilon|^2
\Big)^{\frac{1}{2}}
\quad\text{with}\quad
\langle\mathcal{C}_{*}^\beta
\rangle^{\frac{1}{\beta}}\lesssim 1
\end{equation}
holds for all $0<r\leq R$ and $\beta\in[1,\infty)$.
Moreover, if $\Omega$ is the bounded uniform $C^{1,\tau}$ domain and
$\varepsilon\in(0,1]$, for any $B$ with $B\subset D_R$ or $x_B\in\Delta_R$, and for any $1\leq p<\bar{p}<\infty$, we obtain that
\begin{subequations}
\begin{align}
 &|\nabla u_\varepsilon(x_B)|\lesssim
 \mathcal{C}_*(x_B) \Big(\dashint_{B\cap\Omega}|\nabla u_\varepsilon|^2
\Big)^{1/2};\label{pri:2.11} \\
 &\big\langle |\nabla u_\varepsilon(x_B)|^p
\big\rangle^{\frac{1}{p}}
\lesssim
\Big\langle\Big(\dashint_{B\cap\Omega}|\nabla u_\varepsilon|^2
\Big)^{\bar{p}/2}\Big\rangle^{\frac{1}{\bar{p}}}, \label{pri:2.10}
\end{align}
\end{subequations}
where the multiplicative constant depends
on $\lambda,\lambda_1,\lambda_2,d,\tau,m_0,R_0^{\tau}M_0$, and $p$ at most.
\end{corollary}

\begin{proof}
The idea is analogous to that given for Corollary $\ref{corollary:2.1}$, so the proof is omitted.
\end{proof}

\begin{remark}
\emph{Both $\mathcal{C}_*$ and $(m_0,R_0^{\tau}M_0)$ shown in Corollary $\ref{corollary:2.2}$
are invariants under the scale transformation.
The estimated constants in $\eqref{pri:2.11}$ and
$\eqref{pri:2.10}$ depend on $m_0$ and $R_0^{\tau}M_0$ at most polynomially, and its specific form is identical to that given for $\eqref{pri:8.1}$ in Lemma $\ref{lemma:8.1}$.
%For consistency with traditional statements, in the following sections, we shortly use the multiplicative constant $C$ depending on $\Omega$ to represent the dependence on $R_0^{\tau}M_0$.
}
\end{remark}

\section{Green functions}\label{sec:3}

\noindent
The main ideas of this section may be found in
the literature \cite{Shen18} for the periodic homogenization.
Based on the work in the previous section,
the estimated constants obtained in this section are invariant under the scale transformation.
To avoid potential confusion, we mention that the notation  $\sigma$ appeared in this section (e.g. in Proposition $\ref{P:2}$) has no relation to that in $\eqref{a:3}$.

\begin{proposition}[annealed decays estimates]\label{P:4}
Let $\Omega\subset\mathbb{R}^d$ (with $d\geq 2$) be a bounded uniform $C^{1,\tau}$ domain with $\tau\in(0,1)$, and
$\varepsilon\in(0,1]$.
Suppose that the ensemble $\langle\cdot\rangle$ satisfies
$\eqref{a:2}$ and $\eqref{a:3}$, and $G_\varepsilon(\cdot,\cdot)$ is the Green function
associated with $\mathcal{L}_\varepsilon$ and $\Omega$.
Then, for any $p\in[1,\infty)$, there hold
\begin{equation}\label{pri:19}
\left\{\begin{aligned}
&\langle|\nabla_x G_\varepsilon(x,y)|^p\rangle^{\frac{1}{p}}
\lesssim
\frac{1}{|x-y|^{d-1}}
\min\Big\{1,\frac{\delta(y)}{|x-y|}\Big\};\\
&\langle|\nabla_y G_\varepsilon(x,y)|^p\rangle^{\frac{1}{p}}
\lesssim \frac{1}{|x-y|^{d-1}}
\min\Big\{1,\frac{\delta(x)}{|x-y|}\Big\};\\
&\langle|\nabla_y\nabla_x G_\varepsilon(x,y)|^p\rangle^{\frac{1}{p}}
\lesssim \frac{1}{|x-y|^{d}}
\end{aligned}\right.
\end{equation}
for any $x,y\in\Omega$ with $x\not=y$,
where we recall that $\delta(x):=\text{dist}(x,\partial\Omega)$, and
the multiplicative constants depend on $\lambda$, $\lambda_1$, $\lambda_2$, $d$, $m_0$, $R_0^{\tau}M_0$, $\tau$, and $p$.
\end{proposition}

\begin{lemma}[quenched decay estimates I]\label{lemma:3}
Assume the same conditions as in Proposition $\ref{P:4}$.
%Let the stationary random fields $\mathcal{H}_*$, $\mathcal{C}_*$ be given as in Corollaries $\ref{corollary:2.1}$ and $\ref{corollary:2.2}$, respectively.
Then, for any $x,y\in\Omega$ with $x\not=y$, there exist
random fields $C_*$ and $H_*$, satisfying
\begin{equation}\label{pri:3.9}
 \big\langle|(C_*,H_*)|^{\beta}\big\rangle^{\frac{1}{\beta}}
 \lesssim_{\lambda,\lambda_1,\lambda_2,d,\beta} 1  \quad\forall \beta<\infty,
\end{equation}
such that there hold the following decay estimates:
\begin{subequations}
\begin{align}
  &|\nabla_x G_\varepsilon(x,y)|
\lesssim
\frac{C_*(x)}{|x-y|^{d-1}}
\min\Big\{H_*(y),\frac{C_*(y)\delta(y)}{|x-y|}\Big\};
\label{pri:2.5-b}\\
  &|\nabla_y G_\varepsilon(x,y)|
\lesssim \frac{C_*(y)}{|x-y|^{d-1}}
\min\Big\{H_*(x),\frac{C_*(x)\delta(x)}{|x-y|}\Big\}; \label{pri:2.5-c}\\
  &|\nabla_x\nabla_y G_\varepsilon(x,y)|
\lesssim \frac{C_*(x)C_*(y)}{|x-y|^{d}},
\label{pri:2.5-d}
\end{align}
\end{subequations}
where
the multiplicative constants depend on $\lambda$, $d$, $m_0$, $R_0^{\tau}M_0$, and $\tau$.
\end{lemma}

\begin{proof}
The basic idea is based upon the interplay between duality and
uniform regularities, which had been leaked out in periodic homogenization
theory (see e.g. \cite{Avellaneda-Lin87}). We provide a proof
for the reader's convenience, and split the proof into two steps.
For any $x_0,y_0\in\Omega$, we set $r=|x_0-y_0|$, and $\text{supp}(f)\subset U_{r/3}(y_0):=B_{r/3}(y_0)\cap\Omega$ (it is fine to assume
$U_{r/3}(y_0)\not=\emptyset$ and merely handle the cases $\delta(y_0)\ll r$ and $\delta(x_0)\ll r$. Otherwise, we just employ the corresponding interior estimates).

\textbf{Step 1.} Show the arguments for $\eqref{pri:2.5-d}$.
On account of the formula $\eqref{eq:2.1}$,
taking the gradient $\nabla_x$ on the both sides of $\eqref{eq:2.1}$, we have
\begin{equation*}
 -\int_{\Omega}\nabla_x\nabla G_\varepsilon(x,\cdot) f
 = \nabla_x u_\varepsilon(x)
 \qquad \forall x\in\Omega.
\end{equation*}
Appealing to Corollary $\ref{corollary:2.2}$ and
energy estimates, it follows that
\begin{equation*}
\begin{aligned}
\Big|\int_{U_{r/3}(y_0)}
&\nabla_x\nabla G_\varepsilon(x_0,\cdot) f\Big|
\leq |\nabla_x u_\varepsilon(x_0)|\\
&\lesssim^{\eqref{pri:2.11},\eqref{pri:20}} \mathcal{C}_*(x_0)
\mathcal{C}_*(\bar{x}_0)
\Big(\dashint_{D_{r/3}(\bar{x}_0)}|\nabla u_\varepsilon|^2\Big)^{\frac{1}{2}}
\lesssim \frac{\mathcal{C}_{*}(x_0,\bar{x}_0)}{r^{\frac{d}{2}}}
\Big(\int_{\Omega}|f|^2\Big)^{\frac{1}{2}},
\end{aligned}
\end{equation*}
where $\mathcal{C}_{*}(x_0,\bar{x}_0):=\mathcal{C}_*(x_0)
\mathcal{C}_*(\bar{x}_0)$ with $\bar{x}_0\in\partial\Omega$ being
such that $\delta(x_0)=|x_0-\bar{x}_0|$. By a duality argument, this further implies
\begin{equation}\label{f:2.18}
\begin{aligned}
\Big(\dashint_{U_{r/3}(y_0)}|\nabla_x\nabla G_\varepsilon(x_0,\cdot)|^2\Big)^{\frac{1}{2}}
\lesssim \frac{\mathcal{C}_{*}(x_0,\bar{x}_0)}{r^{d}}.
\end{aligned}
\end{equation}
Since $\nabla_x G_\varepsilon(x_0,\cdot)$ satisfies
$\mathcal{L}_{\varepsilon}^{*}[
\nabla_x G_\varepsilon(x_0,\cdot)] = 0$ in $U_{\frac{r}{3}}(y_0)$
and $\nabla_x G_\varepsilon(x_0,\cdot) = 0$ on
$B_{\frac{r}{3}}(y_0)\cap\partial\Omega$, by using Corollary $\ref{corollary:2.2}$ again, we consequently derive that
\begin{equation*}
\begin{aligned}
 |\nabla_y \nabla_xG_\varepsilon(x_0,y_0)|
& \lesssim^{\eqref{pri:2.11},\eqref{pri:20}} \mathcal{C}_{*}(y_0,\bar{y}_0)
 \Big(\dashint_{U_{\frac{r}{3}}(y_0)}|\nabla \nabla_x G_\varepsilon(x_0,\cdot)|^2\Big)^{\frac{1}{2}}\\
& \lesssim^{\eqref{f:2.18}} \frac{\mathcal{C}_{*}(x_0,\bar{x}_0)\mathcal{C}_{*}(y_0,\bar{y}_0)}{r^{d}}
=:\frac{C_*(x_0)C_*(y_0)}{r^d}.
\end{aligned}
\end{equation*}
This gives us the decay estimate $\eqref{pri:2.5-d}$.

\textbf{Step 2.} Show the arguments for $\eqref{pri:2.5-c}$
and $\eqref{pri:2.5-b}$.
By noting that $\nabla_x G_\varepsilon(x_0,\cdot) = 0$ on $B_{\frac{r}{3}}(y_0)\cap\partial\Omega$, there exists $\bar{y}_0\in\partial\Omega$ such
that $\delta(y_0)=|y_0-\bar{y}_0|$, and we can arrive at
\begin{equation*}
  |\nabla_x G_\varepsilon(x_0,y_0)|
  \lesssim |\nabla_y\nabla_x G_\varepsilon(x_0,y_0)||y_0-\bar{y}_0|
  \lesssim^{\eqref{pri:2.5-d}} \frac{C_{*}(x_0)C_{*}(y_0)
  \delta(y_0)}{|x_0-y_0|^d},
\end{equation*}
which is the part of the estimate $\eqref{pri:2.5-b}$.
By the same token, we can also obtain
\begin{equation}\label{f:2.2}
 |\nabla_y G_\varepsilon(x_0,y_0)|
 \lesssim \frac{C_{*}(x_0)C_{*}(y_0)
  \delta(x_0)}{|x_0-y_0|^d}.
\end{equation}

Moreover,
inspired by the same argument given for $\eqref{pri:2.5-d}$, we start from
the representation $\eqref{eq:2.1}$ and obtain
\begin{equation*}
\begin{aligned}
\Big|\int_{U_{r/3}(y_0)}
&\nabla G_\varepsilon(x_0,\cdot) f\Big|
\leq |u_\varepsilon(x_0)|
\lesssim^{\eqref{pri:2.6}} \mathcal{H}_*(x_0)\delta(x_0)
\Big(\dashint_{D_{2\delta(x_0)}(\bar{x}_0)}|\nabla u_\varepsilon|^2\Big)^{\frac{1}{2}}\\
&\lesssim^{\eqref{pri:3}} r\mathcal{H}_*(x_0,\bar{x}_0)
\Big(\dashint_{D_{r/6}(\bar{x}_0)}|\nabla u_\varepsilon|^2\Big)^{\frac{1}{2}}
\lesssim \frac{H_{*}(x_0)}{r^{\frac{d}{2}-1}}
\Big(\int_{\Omega}|f|^2\Big)^{\frac{1}{2}},
\end{aligned}
\end{equation*}
where $\mathcal{H}_*(x_0,\bar{x}_0):=\mathcal{H}_*(x_0)
\mathcal{H}_*(\bar{x}_0)$ with $\bar{x}_0\in\partial\Omega$ being such
that $\delta(x_0)=|x_0-\bar{x}_0|$, and we merely set  $H_*(x_0):=\mathcal{H}_*(x_0,\bar{x}_0)$.
This implies
\begin{equation}\label{f:2.28}
\Big(\dashint_{U_{r/3}(y_0)}|\nabla G_\varepsilon(x_0,\cdot)|^2\Big)^{\frac{1}{2}}
\lesssim \frac{H_{*}(x_0)}{r^{d-1}}.
\end{equation}
Therefore, by using Corollary $\ref{corollary:2.2}$ again,
we have
\begin{equation}\label{f:2.29}
\begin{aligned}
& |\nabla_y G_\varepsilon(x_0,y_0)|\\
& \lesssim^{\eqref{pri:2.11},\eqref{pri:20}}
 \mathcal{C}_{*}(y_0)\mathcal{C}_{*}(\bar{y}_0)
 \Big(\dashint_{D_{\frac{r}{3}}(y_0)}|\nabla G_\varepsilon(x_0,\cdot)|^2\Big)^{\frac{1}{2}}
 \lesssim^{\eqref{f:2.28}} \frac{H_{*}(x_0)C_{*}(y_0)}{|x_0-y_0|^{d-1}},
\end{aligned}
\end{equation}
where $\bar{y}_0\in\partial\Omega$ is such that
$\delta(y_0)=|y_0-\bar{y}_0|$,
which coupled with $\eqref{f:2.2}$ gives us the estimate $\eqref{pri:2.5-c}$.

To see the rest part of $\eqref{pri:2.5-b}$,
let $G_\varepsilon^*(y_0,\cdot)$ be the Green function associated with $\mathcal{L}_\varepsilon[G_\varepsilon^*(y_0,\cdot)] = \delta_{y_0}$ in $\Omega$ and
$G_\varepsilon^*(y_0,\cdot) = 0$ on $\partial\Omega$. The counterpart  estimate of $\eqref{f:2.29}$
will hold below, i.e.,
\begin{equation*}
|\nabla_x G_\varepsilon^*(y_0,x_0)|
 \lesssim \frac{C_{*}(x_0)H_{*}(y_0)}{|x_0-y_0|^{d-1}}.
\end{equation*}
This coupled with the property: $G^*_\varepsilon(y,x) = G_\varepsilon(x,y)$ for any
$x,y\in\Omega$ finally leads to
the desired estimate $\eqref{pri:2.5-b}$.
Finally, the estimate $\eqref{pri:3.9}$ follows from
Corollaries $\ref{corollary:2.1}$ and $\ref{corollary:2.2}$.
We have completed the whole proof.
\end{proof}

\begin{remark}
\emph{We actually abuse the notations in Lemma $\ref{lemma:3}$.
Since $a=a^{*}$ is not necessary in this paper, compared with  $\mathcal{H}_*(x)$ and $\mathcal{C}_*(x)$,
$\mathcal{H}_*(y)$ and $\mathcal{C}_*(y)$ presented above should be determined by $a^*$. However, it has little effect on later annealed estimates so we prefer a passive way to show their difference.}
\end{remark}

\medskip

\textbf{Proof of Proposition $\ref{P:4}$.}
Based upon Lemma $\ref{lemma:3}$, a routine computation leads to
the stated estimates $\eqref{pri:19}$, and we are done.
\qed

\begin{proposition}[annealed layer-type estimates]\label{P:2}
Let $\Omega\subset\mathbb{R}^d$ with $d\geq 2$ be a bounded uniform $C^{1,\tau}$ domain with $\tau\in(0,1)$,
and $\varepsilon\in(0,1]$. Suppose that the ensemble $\langle\cdot\rangle$ satisfies $\eqref{a:2}$ and $\eqref{a:3}$,
and $G_\varepsilon(\cdot,\cdot)$ is the Green function
associated with $\mathcal{L}_\varepsilon$ and $\Omega$.
Let $\sigma\in(1/2,1)$ and $R\geq \varepsilon$,
Then, for any $p\in[1,\infty)$, we have
\begin{equation}\label{pri:3.5}
\int_{O_R}\langle|\nabla_y G_\varepsilon(x,y)|^{p}\rangle^{\frac{1}{p}}
dy \lesssim_{\lambda,\lambda_1,\lambda_2,d,m_0,R_0^{\tau}M_0,\tau,\sigma,p} R^{1-\sigma}[\delta(x)]^{\sigma},
\end{equation}
where we recall the notations $O_R:=\{x\in\Omega:\delta(x)\leq R\}$ and
$\delta(x):=\text{dist}(x,\partial\Omega)$.
\end{proposition}

%Here,
%The results stated in Proposition $\ref{P:4}$ would be easily derived, once people established the corresponding decay estimates of quenched version as follows.

\begin{lemma}[quenched decay estimates II]\label{lemma:5}
Let $\Omega\subset\mathbb{R}^d$ with $d\geq 2$ be a bounded uniform $C^{1}$ domain, and $\varepsilon\in(0,1]$.
Suppose the ensemble $\langle\cdot\rangle$ satisfies $\eqref{a:2}$ and $\eqref{a:3}$.
Then, there exists a random field
$H_{*}(\cdot,\cdot)$ satisfying
\begin{equation}\label{pri:3.8}
\big\langle|H_{*}(x,y)|^{\beta}
\big\rangle^{\frac{1}{\beta}}\lesssim_{\lambda,\lambda_1,\lambda_2,d,\beta} 1 \quad\forall \beta\in[1,\infty),~\forall x,y\in\mathbb{R}^d,
\end{equation}
such that, for any $\sigma,\sigma'\in(0,1)$, we have
\begin{subequations}
\begin{align}
&~|G_\varepsilon(x,y)|\lesssim \frac{H_*(y,x)[\delta(x)]^{\sigma}}{|x-y|^{d-2+\sigma}}
\qquad \text{if}~\delta(x)\leq |x-y|/2;\label{pri:3.1a}\\
&~|G_\varepsilon(x,y)|\lesssim \frac{H_*(x,y)[\delta(y)]^{\sigma'}}{|x-y|^{d-2+\sigma'}}
\qquad \text{if}~\delta(y)\leq |x-y|/2; \label{pri:3.1b}\\
&
\left\{\begin{aligned}
&|G_\varepsilon(x,y)|
\lesssim
\frac{\tilde{H}_*(x,y)[\delta(x)]^{\sigma}
[\delta(y)]^{\sigma'}}{|x-y|^{d-2+\sigma+\sigma'}}\\
&\text{if}~\delta(x)\leq |x-y|/2
~\text{or}~\delta(y)\leq |x-y|/2,
\end{aligned}\right.
\label{pri:3.1c}
\end{align}
\end{subequations}
where $\tilde{H}_*(x,y)
:= H_*(y,x)
H_*(x,y)$, and the multiplicative constants depend
on $\lambda, d, \sigma, \sigma'$, $m_0$, and $\omega(t)$ in $\eqref{boundary-2}$ at most.
\end{lemma}

\begin{proof}
The proof is based upon Corollary $\ref{corollary:2.1}$, and the existence
of Green functions would be find in \cite{Dong-Kim09,Hofmann-Kim07}, as well as in \cite{Conlon-Giunti-Otto17,Giunti-Otto22} for a different approach. We plan to show the main arguments in the case of $d\geq 3$, while the case of $d=2$ follows from a few of modifications (see e.g. \cite{Wang-Zhang23}).
The proof is divided into three steps.

\textbf{Step 1.} Show the estimates $\eqref{pri:3.1a}$, $\eqref{pri:3.1b}$,
and $\eqref{pri:3.1c}$.
We start from the estimate
\begin{equation}\label{f:3.15}
|G_{\varepsilon}(x,y)|\lesssim
\left\{\begin{aligned}
& \mathcal{H}_*(x)\mathcal{H}_*(y)|x-y|^{2-d} &\quad& d\geq 3;\\
& \mathcal{H}_*(x)\mathcal{H}_*(y)|x-y|^{-\sigma_0} &\quad& d=2\text{~with~any~}\sigma_0\in(0,1).
\end{aligned}\right.
\end{equation}

For any $x,y\in\Omega$, let $r:=|x-y|$ and $\delta(y)\leq r/2$. Since $\mathcal{L}_\varepsilon^{*}(G_\varepsilon(x,\cdot)) = 0$
in $\Omega\setminus\{y\}$ with $G_\varepsilon(x,\bar{y})=0$, where
$\bar{y}\in\partial\Omega$ is such that $\delta(y)=|y-\bar{y}|$.
In view of Corollary $\ref{corollary:2.1}$, for any $\sigma'\in(0,1)$ we have
\begin{equation}\label{f:3.14}
\begin{aligned}
|G_\varepsilon(x,y)|
&\leq[\delta(y)]^{\sigma'}[G_\varepsilon(x,\cdot)]_{C^{0,\sigma'}(\bar{y})}\\
&\lesssim^{\eqref{pri:2.6},\eqref{pri:3}}
\big(\frac{\delta(y)}{r}\big)^{\sigma'}
\mathcal{H}_{*}(\bar{y})\Big(\dashint_{B_{\frac{r}{2}}(y)\cap\Omega}
|G_\varepsilon(x,\cdot)|^2\Big)^{\frac{1}{2}}\\
&\lesssim^{\eqref{f:3.15}\text{~for~}d\geq3} \frac{\mathcal{H}_{*}(x)\mathcal{H}_{*}(\bar{y})
[\delta(y)]^{\sigma'}}{r^{d-2+\sigma'}}
\Big(\dashint_{B_{\frac{r}{2}}(y)}\mathcal{H}_{*}^2\Big)^{\frac{1}{2}}
=:\frac{H_*(x,y)[\delta(y)]^{\sigma'}}{r^{d-2+\sigma'}},
\end{aligned}
\end{equation}
which gives the estimate $\eqref{pri:3.1b}$ in the case of $d>2$ (The case $d=2$ is discussed in \textbf{Step 3}). Moreover, in view of $\eqref{random-1}$, we also have $\eqref{pri:3.8}$.
We now turn to the estimate $\eqref{pri:3.1a}$. Consider $\mathcal{L}_\varepsilon(G_\varepsilon^*(y,\cdot)) = 0$
in $\Omega\setminus\{x\}$ with $G_\varepsilon^*(y,\bar{x})=0$, where
$\bar{x}\in\partial\Omega$ is such that $\delta(x)=|x-\bar{x}|$.
By the analogous computations given in $\eqref{f:3.14}$,
for any $\sigma\in(0,1)$, we have
\begin{equation}\label{f:3.16}
\begin{aligned}
|G_\varepsilon^*(y,x)|
&\lesssim \frac{\mathcal{H}_\varepsilon(y)\mathcal{H}_\varepsilon(\bar{x})
[\delta(x)]^{\sigma}}{r^{d-2+\sigma}}
\Big(\dashint_{B_{\frac{r}{2}}(x)}\mathcal{H}_{*}^2\Big)^{\frac{1}{2}}
=:\frac{H_*(y,x)[\delta(x)]^{\sigma}}{r^{d-2+\sigma}}.
\end{aligned}
\end{equation}
This together with $G_\varepsilon^*(y,x)=G_\varepsilon(x,y)$ leads to
the stated estimate $\eqref{pri:3.1a}$. In the end, one may plug
the estimate $\eqref{f:3.16}$ instead of $\eqref{f:3.15}$ back into
$\eqref{f:3.14}$ to get the desired estimate $\eqref{pri:3.1c}$.

\textbf{Step 2.} Show the estimate $\eqref{f:3.15}$.
Consider the equations $\mathcal{L}_\varepsilon(v_\varepsilon) = F$
in $\Omega$ with $v_\varepsilon=0$ on $\partial\Omega$. Similar
to the representation $\eqref{eq:2.1}$, we obtain
\begin{equation*}
 v_\varepsilon(x) = \int_{\Omega} G_\varepsilon(x,y)F(y)dy.
\end{equation*}
It's fine to assume that $\text{supp}(F)\subset U_R(x):=B_R(x)\cap\Omega$,
and it follows from Corollary $\ref{corollary:2.1}$ that
\begin{equation*}
\big|\int_{\Omega} G_\varepsilon(x,y)F(y)dy\big|
\leq |v_\varepsilon(x)|
\lesssim^{\eqref{pri:2.6}} \mathcal{H}_{*}(x)
\left\{\begin{aligned}
&R^2\Big(\dashint_{U_R(x)}|F|^q\Big)^{\frac{1}{q}}&~&d\geq 3;\\
&R^{-\frac{2}{s}+\frac{2}{t}}
\Big(\dashint_{U_R(x)}|F|^2\Big)^{\frac{1}{2}}&~&d=2,
\end{aligned}\right.
\end{equation*}
where $q>(d/2)$, $s>2$, and $1<t<2$ are arbitrarily fixed, and we also employ Sobolev's inequality for
the last inequality above\footnote{We refer the reader to \cite[Lemma 3.12]{Xu16}
and \cite[Theorem 1.3]{Wang-Zhang23} for the details.}.
This implies
\begin{equation}\label{f:3.17}
 \Big(\dashint_{U_R(x)}|G_\varepsilon(x,\cdot)|^p\Big)^{\frac{1}{p}}
\lesssim
\mathcal{H}_{*}(x)
\left\{\begin{aligned}
&R^{2-d}&~&d\geq 3,~p\in[1,\frac{d}{d-2});\\
&R^{-\frac{2}{s}+\frac{2}{t}-2}
&~&d=2,~p=2.
\end{aligned}\right.
\end{equation}
Therefore, for any $x,y\in\Omega$, let $r:=|x-y|$. Noting that $\mathcal{L}_\varepsilon^{*}
(G_\varepsilon(x,\cdot))= 0$ in $\Omega\setminus B_{r/2}(x)$,
and by using Corollary $\ref{corollary:2.1}$ again, there holds
\begin{equation}\label{f:3.18}
|G_\varepsilon(x,y)|\lesssim \mathcal{H}_{*}(y)\dashint_{U_{\frac{r}{2}}(y)}
|G_\varepsilon(x,\cdot)|
\lesssim \mathcal{H}_{*}(y)
\dashint_{U_{2r}(x)}
|G_\varepsilon(x,\cdot)|.
\end{equation}
Then, by setting $R=2r$ and plugging the estimate $\eqref{f:3.17}$
back into $\eqref{f:3.18}$, we can derive the stated estimate $\eqref{f:3.15}$ by denoting $\sigma_0 := \frac{2}{s}-\frac{2}{t}+2$, and one can prefer $s>2$ and $t\in(1,2)$ such that $\sigma_0\in(0,1)$.

\textbf{Step 3.} Some modifications for the case $d=2$ in \textbf{Step 1}.
We first remark that if $\delta(y)\leq r/2$, we have
$|G_\varepsilon(x,y)|\lesssim_{d,\lambda,m_0} 1$ almost surely with respect to the measure introduced by $\langle\cdot\rangle$
(see \cite[Theorem 2.21]{Dong-Kim09} or \cite[Theorem 4.1]{Taylor-Kim-Brow13} for general cases), where $m_0$ is the Lipschitz character of $\Omega$ (here it could be very small due to $\partial\Omega\in C^1$).  Thus,
the estimate $\eqref{f:3.14}$ is correspondingly changed into
$|G_\varepsilon(x,y)|\lesssim \mathcal{H}_{*}(\bar{y})(\delta(y)/r)^{\sigma'}$, and
we can similarly modify the proofs of $\eqref{pri:3.1a}$ and
$\eqref{pri:3.1c}$.
We have completed the whole proof.
\end{proof}

%\begin{remark}\label{remark:1}
%\emph{If $\delta(x)\leq |x-y|$, the estimate $\eqref{pri:3.1a}$ is still true. Obviously, we merely need to handle the case $|x-y|/2< \delta(x)\leq |x-y|$, and it is not hard to get it from
%\begin{equation*}
% |G_\varepsilon(x,y)|
% \lesssim \frac{|x-y|^\sigma}{|x-y|^{d-2+\sigma}}
% \lesssim \frac{[\delta(x)]^\sigma}{|x-y|^{d-2+\sigma}}.
%\end{equation*}
%Similarly, we can obtain for other cases. }
%\end{remark}

%\begin{lemma}
%Furthermore, we have
%\begin{equation}\label{pri:2.5-a}
%\begin{aligned}
%&\langle|G_\varepsilon(x,y)|^p\rangle^{\frac{1}{p}}\lesssim \frac{[\delta(x)]^{\sigma}}{|x-y|^{d-2+\sigma}}
%&\quad&\text{if}~\delta(x)\leq |x-y|/2;\\
%&\langle|G_\varepsilon(x,y)|^p\rangle^{\frac{1}{p}}\lesssim \frac{[\delta(y)]^{\sigma'}}{|x-y|^{d-2+\sigma'}}
%&\quad&\text{if}~\delta(y)\leq |x-y|/2; \\
%& \langle|G_\varepsilon(x,y)|^p\rangle^{\frac{1}{p}}
%\lesssim
%\frac{[\delta(x)]^{\sigma}[\delta(y)]^{\sigma'}}{|x-y|^{d-2+\sigma+\sigma'}}
%&\quad&\text{if}~\delta(x)\leq |x-y|/2
%~\text{and}~\delta(y)\leq |x-y|/2.
%\end{aligned}
%\end{equation}
%\end{lemma}

\begin{lemma}\label{lemma:3.1}
Let $\sigma\in(0,1)$ and $\varepsilon\in(0,1]$.
Assume the same conditions as in Proposition $\ref{P:2}$.
Let $R\geq \varepsilon$ and $p\in[1,\infty)$. Then,
for any $x,y\in O_R$ with $|x-y|\geq 2R$,
we have
\begin{equation}\label{pri:3.3}
 \langle|\nabla_y G_\varepsilon(x,y)|^p\rangle^{\frac{1}{p}}
 \lesssim_{\lambda,\lambda_1,\lambda_2,d,m_0,R_0^{\tau}M_0,\tau,\sigma,p} \frac{R^{\sigma-1}[\delta(x)]^{\sigma}}{|x-y|^{d-2+2\sigma}}.
\end{equation}
\end{lemma}

\begin{proof}
Fix any $x,y\in O_R$ with $|x-y|\geq 2R$,
and it follows from Corollary $\ref{corollary:2.2}$ and
Lemma $\ref{lemma:5}$ that
\begin{equation}\label{f:3.13}
\begin{aligned}
 |\nabla_y G_\varepsilon(x,y)|
& \lesssim^{\eqref{pri:2.11},\eqref{pri:20}} \frac{\mathcal{C}_*(y,\bar{y})}{R}
 \Big(\dashint_{D_{R/4}(\bar{y})}
 |G_\varepsilon(x,\cdot)|^2\Big)^{1/2}\\
& \lesssim^{\eqref{pri:3.1c}} \frac{\mathcal{C}_*(y,\bar{y})[\delta(x)]^{\sigma}}{R}
 \bigg(\dashint_{D_{R/4}(\bar{y})}
 \Big(\frac{\tilde{H}_*(x,z)
 [\delta(z)]^{\sigma}}{|x-z|^{d-2+2\sigma}}\Big)^2
 dz\bigg)^{\frac{1}{2}},
\end{aligned}
\end{equation}
where $\mathcal{C}_*(y,\bar{y}):=\mathcal{C}_*(y)\mathcal{C}_*(\bar{y})$, and
$\bar{y}\in\partial\Omega$ is such that $\delta(y)=|y-\bar{y}|$.
Since $|x-y|\sim |x-z|$ for any $z\in B_{\frac{R}{4}}(\bar{y})$ due
to $|x-y|\geq 2R$, we further derive that
\begin{equation}\label{f:3.9}
 |\nabla_y G_\varepsilon(x,y)|
 \lesssim  \frac{\mathcal{C}_*(y,\bar{y})[\delta(x)]^{\sigma}R^{\sigma-1}}{|x-y|^{d-2+2\sigma}}
 \Big(\dashint_{B_{R/4}(\bar{y})}\tilde{H}_*^2(x,z)dz\Big)^{\frac{1}{2}},
\end{equation}
where we also employ the fact $\delta(z)\leq R$ for any $z\in B_{\frac{R}{4}}(\bar{y})$. Then, taking $\langle(\cdot)^p\rangle^{1/p}$ on the both sides of $\eqref{f:3.9}$, the desired estimate $\eqref{pri:3.3}$ follows.
\end{proof}

\begin{lemma}\label{lemma:3.2}
Let $\sigma\in(0,1)$ and $\varepsilon\in(0,1]$.
Assume the same conditions as in Proposition $\ref{P:2}$.
Let $R\geq \varepsilon$ and $p\in[1,\infty)$. Then,
for any $x,y\in O_R$ with $|x-y|< 2R$,
we have
 \begin{equation}\label{pri:3.2}
  \langle|\nabla_y G_\varepsilon(x,y)|^{p}\rangle^{\frac{1}{p}}
  \lesssim_{\lambda,\lambda_1,\lambda_2,d,m_0,R_0^{\tau}M_0,\tau,\sigma,p} \frac{[\delta(x)]^{\sigma}}{|x-y|^{d-1+\sigma}}.
 \end{equation}
\end{lemma}

\begin{proof}
For any $x,y\in O_R$ with $|x-y|< 2R$,
we set $r:=|x-y|$, and the proof is divided into two cases:
(1) $r\geq \delta(x)/2$; (2) $r<\delta(x)/2$.
We first handle the case (1). If $r/2\leq \delta(x)\leq 2r$, it follows from Proposition $\ref{P:4}$ that
\begin{equation*}
 \langle|\nabla_y G_\varepsilon(x,y)|^p\rangle^{\frac{1}{p}}
 \lesssim^{\eqref{pri:19}}\frac{1}{|x-y|^{d-1}}
 = \frac{|x-y|^{\sigma}}{|x-y|^{d-1+\sigma}}
 \lesssim \frac{[\delta(x)]^{\sigma}}{|x-y|^{d-1+\sigma}}.
\end{equation*}
In this regard, it suffices to show the case $\delta(x) <r/2$.
By using Corollary $\ref{corollary:2.2}$ and Lemma $\ref{lemma:5}$
in order (similar to the computations as in $\eqref{f:3.13}$), we obtain
\begin{equation*}
  |\nabla_y G_\varepsilon(x,y)|
 \lesssim^{\eqref{pri:2.11},\eqref{pri:20}} \frac{\mathcal{C}_*(y,\bar{y})}{r}
 \Big(\dashint_{D_{\frac{r}{2}}(\bar{y})}|G_\varepsilon(x,z)|^2dz\Big)^{\frac{1}{2}}
 \lesssim^{\eqref{pri:3.1a}} \frac{\tilde{C}_*(x,y)[\delta(x)]^{\sigma}}{r^{d-1+\sigma}},
\end{equation*}
where
$\mathcal{C}_*(y,\bar{y}):=\mathcal{C}_*(y)\mathcal{C}_*(\bar{y})$ is taken
with $\bar{y}\in\partial\Omega$ being such that $\delta(y)=|y-\bar{y}|$,
and
\begin{equation*}
\tilde{C}(x,y):=\mathcal{C}_*(y,\bar{y})
\Big(\dashint_{B_{r/2}(\bar{y})}|H_*(z,x)|^2dz\Big)^{1/2}.
\end{equation*}
Moreover, for any
$p\in[1,\infty)$,
we have
\begin{equation*}
\big\langle|\tilde{C}_*(x,y)|^p\big\rangle^{\frac{1}{p}}
\lesssim^{\eqref{random-1},\eqref{pri:20}} 1\quad~\forall x,y\in\mathbb{R}^d.
\end{equation*}
This consequently leads to the stated estimate $\eqref{pri:3.2}$.

Now, we proceed to address the case (2) $r<\delta(x)/2$.
We also start from the Lipschitz estimates (in Corollary  $\ref{corollary:2.2}$), and then use the fact that
$G_\varepsilon(\bar{x},z) =0$ with $\bar{x}\in\partial\Omega$
satisfying $\delta(x)=|x-\bar{x}|$ (where we note that the adjoint Green function satisfies $G_\varepsilon^{*}(z,\cdot)=0$ on $\partial\Omega$ and
$G_\varepsilon^{*}(z,x)=G_\varepsilon(x,z)$ for any
$x,y\in\bar{\Omega}\times\bar{\Omega}$), as well as the
boundary H\"older estimates (in Corollary $\ref{corollary:2.1}$)
in order, we can derive that
\begin{equation*}
\begin{aligned}
 &|\nabla_y G_\varepsilon(x,y)|
 \lesssim^{\eqref{pri:2.11}} \frac{\mathcal{C}_*(y)}{r}
 \Big(\dashint_{B_{\frac{r}{5}}(y)}|G_\varepsilon(x,z)|^2dz\Big)^{\frac{1}{2}}\\
 &\lesssim \frac{\mathcal{C}_*(y)[\delta(x)]^{\sigma}}{r}
 \Big(\dashint_{B_{\frac{r}{5}}(y)}
 [G_\varepsilon(\cdot,z)]_{C^{0,\sigma}(\bar{x})}^2
  dz\Big)^{\frac{1}{2}}\\
 &\lesssim^{\eqref{pri:2.6}} \frac{\mathcal{C}_*(y)\mathcal{H}_*(\bar{x})[\delta(x)]^{\sigma}
 }{r^{\sigma}}
  \Big(\dashint_{B_{\frac{r}{5}}(y)}\dashint_{D_{\frac{r}{5}}(\bar{x})}
 |\nabla G_\varepsilon(\cdot,z)|^2
  dz\Big)^{\frac{1}{2}}
  \lesssim^{\eqref{pri:2.5-b}} \frac{\tilde{C}_{**}(\bar{x},y)[\delta(x)]^{\sigma}}{r^{d-1+\sigma}},
\end{aligned}
\end{equation*}
where $\tilde{C}_{**}(\bar{x},y):=
\mathcal{C}_*(y)\mathcal{H}_*(\bar{x})
(\dashint_{B_{\frac{r}{5}}(y)}\dashint_{D_{\frac{r}{5}}(\bar{x})}
 |C_{*}(\tilde{x})H_{*}(z)|^2
  d\tilde{x}dz)^{\frac{1}{2}}$, and we also employ
$|\bar{x}-y|\sim \delta(x)>2r$. Moreover, there holds
$\langle|\tilde{C}_{**}(\bar{x},y)|^p\rangle^{1/p}\lesssim 1$
for any $p\in[1,\infty)$.
This together with the case (1)
leads to the desired estimate $\eqref{pri:3.2}$.
\end{proof}

\medskip

\textbf{Proof of Proposition $\ref{P:2}$}.
The proof is based upon Lemmas $\ref{lemma:3.1}$ and $\ref{lemma:3.2}$,
and we start from a simple decomposition, i.e., for any $x\in O_{R}$ we
have
\begin{equation}\label{f:3.10}
\int_{O_R}\langle|\nabla_y G_\varepsilon(x,y)|^{p}\rangle^{\frac{1}{p}}
dy \leq \Big\{\int_{O_R\cap B_{2R}(x)}
+ \int_{O_R\cap B_{2R}^{c}(x)}\Big\}
\langle|\nabla_y G_\varepsilon(x,y)|^{p}\rangle^{\frac{1}{p}}
dy,
\end{equation}
where $B_{2R}^{c}(x)$ is the complement of $B_{2R}(x)$.

On the one hand, it follows from Lemma $\ref{lemma:3.2}$ that
\begin{equation}\label{f:3.11}
\begin{aligned}
\int_{O_R\cap B_R(x)}
\langle|\nabla_y G_\varepsilon(x,y)|^{p}\rangle^{\frac{1}{p}}dy
&\lesssim^{\eqref{pri:3.2}}
[\delta(x)]^{\sigma}\int_{B_{2R}(x)}
\frac{dy}{|x-y|^{d-1+\sigma}}\\
&\lesssim R^{1-\sigma}[\delta(x)]^{\sigma}.
\end{aligned}
\end{equation}

On the other hand, in view of Lemma $\ref{lemma:3.1}$, we obtain
\begin{equation}\label{f:3.12}
\begin{aligned}
\int_{O_R\cap B_{2R}^{c}(x)}
\langle|\nabla_y G_\varepsilon(x,y)|^{p}\rangle^{\frac{1}{p}}dy
&\lesssim^{\eqref{pri:3.3}}
R^{\sigma-1}[\delta(x)]^{\sigma}\int_{O_R\cap B_{2R}^{c}(x)}
\frac{dy}{|x-y|^{d-2+2\sigma}}\\
&\lesssim R^{\sigma}[\delta(x)]^{\sigma}
\int_{R}^{\infty}\frac{dr}{r^{2\sigma}}
\lesssim R^{1-\sigma}[\delta(x)]^{\sigma},
\end{aligned}
\end{equation}
where we employ the co-area formula for the second inequality above,
and require the condition $2\sigma>1$ for the convergence.
Plugging the estimates $\eqref{f:3.11}$ and $\eqref{f:3.12}$
back into $\eqref{f:3.10}$ leads to the stated estimate
$\eqref{pri:3.5}$, and we have completed the whole proof.
\qed

\section{Boundary correctors}\label{sec:4}
\begin{proposition}\label{P:5}
Let $\Omega\ni \{0\}$ be a bounded $C^{1,\tau}$ domain with
the uniform interior ball condition and $\tau\in(0,1]$,
and $\varepsilon\in(0,1]$.
Suppose that $\langle\cdot\rangle$ satisfies the spectral gap condition $\eqref{a:2}$, and the
admissible coefficients satisfy the smoothness condition $\eqref{a:3}$. Let $\{\tilde{\Phi}_{\varepsilon,i}\}_{i=1}^d$ be the solution of
the equations $\eqref{pde:10}$.
Then, for any $p\in[1,\infty)$ and $\eta\in(0,1]$, there hold
\begin{subequations}
\begin{align}
&\esssup_{x\in\Omega\atop
i=1,\cdots,d} \big\langle|\tilde{\Phi}_{\varepsilon,i}(x)|^p\big\rangle^{\frac{1}{p}}
 \lesssim \varepsilon\mu_{d}(R_0/\varepsilon); \label{pri:21-1} \\
&\esssup_{x\in\Omega\atop i=1,\cdots,d}\Big\langle\big[\tilde{\Phi}_{\varepsilon,i}
\big]_{C^{0,\eta}(U_\varepsilon(x))}^p
\Big\rangle^{\frac{1}{p}}
\lesssim
\varepsilon^{1-\eta}\mu_d(R_0/\varepsilon).
\label{pri:21-2}
\end{align}
\end{subequations}
Moreover, let $Q_{\varepsilon,i} := \tilde{\Phi}_{\varepsilon,i}-\varepsilon\phi_i^\varepsilon$
with $i=1,\cdots,d$. Then we have
\begin{equation}\label{pri:18}
 \max_{i=1\cdots,d}\big\langle|\nabla Q_{\varepsilon,i}(z)|^p\big\rangle^{\frac{1}{p}}
 \lesssim \mu_d(R_0/\varepsilon)
 \min\Big\{1,\frac{\varepsilon}{\delta(z)}\Big\}
 \qquad \forall z\in\Omega,
\end{equation}
where the multiplicative constants depend on
$\lambda,\lambda_1,\lambda_2,d,m_0,R_0^{\tau}M_0,\tau,p$, and $\eta$ at most.
\end{proposition}

\begin{lemma}\label{lemma:4.1}
Let $\varepsilon\in(0,1]$, and
assume the same conditions as in Proposition $\ref{P:5}$.
%Let $\Omega\ni\{0\}$ be a bounded $C^{1,\tau}$ domain with $\tau\in(0,1)$,
%and $\varepsilon\in(0,1]$. Suppose that $\langle\cdot\rangle$ satisfies the spectral gap condition $\eqref{a:2}$, and
%(admissible) coefficients satisfy the smoothness condition $\eqref{a:3}$.
Let $\phi$ be the corrector given by the equation $\eqref{corrector}$,
and $v_\varepsilon$ is associated with $\phi$ by the following equations:
\begin{equation}\label{pde:6*}
\left\{\begin{aligned}
\mathcal{L}_\varepsilon(v_{\varepsilon})
&= 0 \quad&\text{in}&\quad \Omega;\\
v_{\varepsilon}
&= -\varepsilon\phi^\varepsilon \quad&\text{on}&\quad \partial\Omega.
\end{aligned}\right.
\end{equation}
Then, for any $p\in[1,\infty)$, $x\in\partial\Omega$, and $r\geq \varepsilon$, there holds
\begin{equation}\label{pri:3.6}
  \Big\langle\Big(\dashint_{B_r(x)\cap\Omega}|\nabla v_\varepsilon|^2\Big)^{\frac{p}{2}}\Big\rangle^{\frac{1}{p}}
  \lesssim_{\lambda,\lambda_1,\lambda_2,d,m_0,R_0^{\tau}M_0,\tau,p} \mu_d(R_0/\varepsilon).
\end{equation}
\end{lemma}

\begin{proof}
The main idea is inspired by Shen \cite[Lemma 5.4.5]{Shen18}, and
the proof is divided into two steps. Let $\varepsilon\leq r\leq (R_0/4)$
and it suffices to show the desired estimate $\eqref{pri:3.6}$ for the case of $p\geq 2$.
Let $\varphi\in C^1(\Omega)$ be a cut-off function, satisfying
$\varphi =1$ on $O_{r}$, $\varphi =0$ on $\Omega\setminus O_{2r}$ and $|\nabla\varphi|\lesssim 1/r$.

\textbf{Step 1.} Reduction and outline the proof.
For the ease of the statement, let $g := -\varepsilon\phi^\varepsilon\varphi$, and we consider $z_\varepsilon = v_\varepsilon - g$. Then, it follows from the equations $\eqref{pde:6*}$ that
\begin{equation*}
\left\{\begin{aligned}
-\nabla\cdot a^\varepsilon \nabla z_\varepsilon
&= \nabla\cdot a^\varepsilon \nabla g
&\quad&\text{in}\quad\Omega;\\
z_\varepsilon
&= 0
&\quad&\text{on}\quad\partial\Omega.
\end{aligned}\right.
\end{equation*}
For any fixed $x\in\partial\Omega$, it follows from Caccioppoli's inequality (see e.g. \cite{Giaquinta-Martinazzi12}) that
\begin{equation}\label{f:3.4}
 \dashint_{D_r(x)}|\nabla z_\varepsilon|^2
 \lesssim_{\lambda,d}\frac{1}{r^2}\dashint_{D_{2r}(x)}|z_\varepsilon|^2
 + \dashint_{D_{2r}(x)}|\nabla g|^2,
\end{equation}
where $r>0$ is arbitrary.
On the one hand, it is known that
\begin{equation}\label{f:3.2}
\begin{aligned}
   &\Big\langle
   \Big(\dashint_{D_{2r}(x)}|\nabla g|^2\Big)^{\frac{p}{2}}\Big\rangle^{\frac{2}{p}}
   \leq
   \dashint_{D_{2r}(x)}\langle|\nabla g|^p\rangle^{\frac{2}{p}}\\
   &\lesssim_d \dashint_{D_{2r}(x)}\Big(\frac{1}{r^2}\langle|
   \varepsilon\phi^\varepsilon|^{p}\rangle^{\frac{2}{p}}
   + \langle|
   \nabla\phi^\varepsilon|^{p}\rangle^{\frac{2}{p}}\Big)
   \lesssim \mu_d^2(R_0/\varepsilon)+1,
\end{aligned}
\end{equation}
where we use the condition $r\geq\varepsilon$.
On the other hand, it suffices to show
\begin{equation}\label{f:3.3}
   \esssup_{y\in D_{2r}(x)}\langle|z_\varepsilon(y)|^p\rangle^{\frac{1}{p}}
   \lesssim_{\lambda,\lambda_1,\lambda_2,d,m_0,R_0^{\tau}M_0,\tau,p} r \mu_d(R_0/\varepsilon).
\end{equation}

Let us admit the above estimate for a while, and plugging the estimates $\eqref{f:3.2}$ and $\eqref{f:3.3}$ back into
$\eqref{f:3.4}$, one can derive that
\begin{equation}\label{f:3.5}
\Big\langle\Big(\dashint_{D_r(x)}|\nabla z_\varepsilon|^2\Big)^{\frac{p}{2}}\Big\rangle^{\frac{2}{p}}
\lesssim \mu_d^2(R_0/\varepsilon)+1,
\end{equation}
and therefore we obtain that
\begin{equation*}
\begin{aligned}
&\Big\langle\Big(\dashint_{D_r(x)}|\nabla v_\varepsilon|^2\Big)^{\frac{p}{2}}\Big\rangle^{\frac{2}{p}}\\
&\leq
\Big\langle\Big(\dashint_{D_r(x)}|\nabla z_\varepsilon|^2\Big)^{\frac{p}{2}}\Big\rangle^{\frac{2}{p}}
+ \Big\langle\Big(\dashint_{D_r(x)}|\nabla g|^2\Big)^{\frac{p}{2}}\Big\rangle^{\frac{2}{p}}
\lesssim^{\eqref{f:3.5},\eqref{f:3.2}}
\mu_d^2(R_0/\varepsilon)+1.
\end{aligned}
\end{equation*}

\textbf{Step 2. } Show the estimate $\eqref{f:3.3}$.
By the advantage of Green's function,
in view of $\eqref{eq:2.1}$, for any $\tilde{x}
\in D_{2r}(x)$,
we have the representation
\begin{equation*}
  z_\varepsilon(\tilde{x}) = -\int_{\Omega}\nabla_yG_\varepsilon(\tilde{x},y)\cdot (a^\varepsilon\nabla g)(y)dy
  = -\int_{O_{2r}}\nabla_yG_\varepsilon(\tilde{x},y)\cdot (a^\varepsilon\nabla g)(y)dy,
\end{equation*}
where we note that $\text{supp}(g)\subset O_{2r}$. Also, for any
$p_2\geq p_1>p\geq 1$, it follows from the computations
analogous to $\eqref{f:3.2}$ and Proposition $\ref{P:2}$ that,
for any $\tilde{x}\in D_{2r}(x)$,
\begin{equation*}
\begin{aligned}
&\langle|z_\varepsilon(\tilde{x})|^p\rangle^{\frac{1}{p}}
\lesssim_{\lambda,d}
\int_{O_{2r}}\langle|\nabla_y G_\varepsilon(\tilde{x},y)|^{p_1}\rangle^{\frac{1}{p_1}}
\langle|\nabla g|^{p_2}\rangle^{\frac{1}{p_2}}\\
&\lesssim \mu_d(R_0/\varepsilon)
\int_{O_{2r}}\langle|\nabla_y G_\varepsilon(\tilde{x},y)|^{p_1}\rangle^{\frac{1}{p_1}}
\lesssim^{\eqref{pri:3.5}} \mu_d(R_0/\varepsilon)r^{1-\sigma}
[\delta(\tilde{x})]^{\sigma}
\lesssim \mu_d(R_0/\varepsilon)r,
\end{aligned}
\end{equation*}
where we also employ the fact that $\delta(\tilde{x})\leq 2r$ for any
$\tilde{x}\in D_{2r}(x)$ in the last inequality.
This further implies the stated estimate $\eqref{f:3.3}$,
and we are done.
\end{proof}

\textbf{Proof of Proposition $\ref{P:5}$.}
The proof is divided into four steps. By the definition of $Q_{\varepsilon,i}$ and the equation $\eqref{pde:10}$, there holds
\begin{equation}\label{pde:6}
\left\{\begin{aligned}
\nabla\cdot a^\varepsilon\nabla Q_{\varepsilon,i}
&= 0 \quad&\text{in}&\quad \Omega;\\
Q_{\varepsilon,i}
&= -\varepsilon\phi_i^\varepsilon \quad&\text{on}&\quad \partial\Omega,
\quad i=1,\cdots,d.
\end{aligned}\right.
\end{equation}
For the ease of the statement, we omit $i$ in the subscripts of $Q_{\varepsilon,i}$ throughout the proof.

\medskip
\textbf{Step 1.} Arguments for the estimate $\eqref{pri:21-1}$.
On account of $\tilde{\Phi}_{\varepsilon,i} = Q_{\varepsilon,i}+ \varepsilon\phi_i^\varepsilon$ and
the estimate $\eqref{pri:**3}$, it suffices
to establish
\begin{equation}\label{pri:2.8}
 \esssup_{x\in\Omega}\langle |Q_\varepsilon(x)|^p\rangle^{\frac{1}{p}}
 \lesssim \varepsilon\mu_d(R_0/\varepsilon).
\end{equation}

In view of the equations $\eqref{pde:6}$, we have the
Poisson formula,
\begin{equation}\label{eq:4.1}
Q_\varepsilon(x) = -\varepsilon\int_{\partial\Omega}
n(y)\cdot a^{*\varepsilon}(y)\nabla_y G_{\varepsilon}(x,y)\phi^\varepsilon(y)dS(y)
\quad \forall x\in\Omega,
\end{equation}
where $dS$ is the surface measure on $\partial\Omega$, and
$n$ is the outward unit normal vector to $\partial\Omega$.
Let $\Gamma_{N_0}(x_0)\subset\Omega$ be a cone with vertex $x_0$
and aperture $N_0$ (see Subsection $\ref{notation}$), and
$\bar{p}>p$. For any $x\in \Gamma_{N_0}(x_0)$, it follows from
$\eqref{eq:4.1}$ and Proposition $\ref{P:4}$ that
\begin{equation}\label{f:35}
\begin{aligned}
&\quad\langle |Q_\varepsilon(x)|^p\rangle^{\frac{1}{p}}\\
&\lesssim^{\eqref{pri:19}} \varepsilon\delta(x)
\int_{\partial\Omega}\frac{\langle|\phi^\varepsilon(y)|^{\bar{p}}
\rangle^{\frac{1}{\bar{p}}}}{|x-y|^d}dS(y)
\lesssim^{\eqref{pri:**3}} \varepsilon\mu_d(\frac{R_0}{\varepsilon})\delta(x)
\int_{\partial\Omega}\frac{dS(y)}{|x-y|^d}\\
&\lesssim \varepsilon\mu_d(R_0/\varepsilon)
\mathcal{M}_{\partial\Omega}(1)(x_0)\lesssim \varepsilon\mu_d(R_0/\varepsilon),
\end{aligned}
\end{equation}
where we also employ the fact that $\phi(0)=0$ in the second inequality,  and $\mathcal{M}_{\partial\Omega}(f)$ is the Hardy–Littlewood maximal function of $f$ on $\partial\Omega$, defined by
\begin{equation*}
\mathcal{M}_{\partial\Omega}(f)(x_0) := \sup_{0<r<R_0}\Big\{\dashint_{B_r(x_0)\cap\partial\Omega}
 |f(x)|dS(x)\Big\}.
\end{equation*}

Inspired by the definition $\eqref{mf}$, the nontangential maximal function of $\langle |Q_\varepsilon|^p\rangle^{\frac{1}{p}}$ is defined by
\begin{equation*}
\big(\langle |Q_\varepsilon|^p\rangle^{\frac{1}{p}}\big)^*(x_0)
:= \sup_{x\in\Gamma_{N_0}(x_0)}|\langle |Q_\varepsilon(x)|^p\rangle^{\frac{1}{p}}|.
\end{equation*}
%where $\Gamma_{N_0}(x_0):=\{x\in\Omega:|x-x_0|\leq N_0\text{dist}(x,\partial\Omega)\}$, and
%is the cone with vertex $x_0$
%and aperture $N_0$, and $N_0>1$ may be sufficiently large.
Therefore, on account of $\eqref{f:35}$,  one can further derive
that
\begin{equation}\label{f:2.21}
\big(\langle |Q_\varepsilon|^p\rangle^{\frac{1}{p}}\big)^*(x_0)
\lesssim_{\lambda,\lambda_1,\lambda_2,d,m_0,R_0^{\tau}M_0,\tau,p} \varepsilon\mu_d(R_0/\varepsilon).
\end{equation}

Consequently, for any $1<q<\infty$, we arrive at
\begin{equation*}
\begin{aligned}
\bigg(\dashint_{\partial\Omega}
|\big(\langle |Q_\varepsilon|^p\rangle^{\frac{1}{p}}\big)^*|^{q}
\bigg)^{\frac{1}{q}}
\lesssim^{\eqref{f:2.21}} \varepsilon\mu_d(R_0/\varepsilon),
~~\text{and}~
\big\|(\langle |Q_\varepsilon|^p\rangle^{\frac{1}{p}})^*\big\|_{L^\infty(\partial\Omega)}
\lesssim \varepsilon\mu_d(R_0/\varepsilon),
\end{aligned}
\end{equation*}
and the later one actually gives us the desired estimate $\eqref{pri:2.8}$ (by a simple contradiction argument).

\medskip
\textbf{Step 2.}
Arguments for the estimate $\eqref{pri:21-2}$.
Let $\bar{p}>p>\underline{p}>2$ and $\eta\in(0,1]$. We start
from the estimate $\eqref{pri:**3}$ and $\langle|\nabla\phi|^{\bar{p}}\rangle^{\frac{1}{\bar{p}}}\lesssim 1$
(see Lemma $\ref{lemma:*3}$) that
\begin{equation}\label{f:2.23*}
\big\langle [\phi^{\varepsilon}]_{C^{0,\eta}(B_{3r}(x))}^{\bar{p}}
\big\rangle^{\frac{1}{\bar{p}}} \lesssim \varepsilon^{-\eta}\mu_d(R_0/\varepsilon)
\quad \forall x\in\Omega.
\end{equation}

Then,
it is reduced to showing the estimates:
\begin{subequations}
\begin{align}
&\esssup_{x\in\Omega}\big\langle\big[Q_\varepsilon
\big]_{C^{0,1}(U_{\varepsilon}(x))}^{\underline{p}}
\big\rangle^{\frac{1}{\underline{p}}}
\lesssim \mu_d(R_0/\varepsilon); \label{pri:2.9}\\
& \esssup_{x\in\Omega}\big\langle\big[Q_\varepsilon
\big]_{C^{0,\eta'}(U_{\varepsilon}(x))}^p
\big\rangle^{\frac{1}{p}}
\lesssim \varepsilon^{1-\eta'}\mu_d(R_0/\varepsilon)
\quad\forall \eta'\in(0,1).
\label{pri:3.7}
\end{align}
\end{subequations}

If we first admit the estimate of $\eqref{pri:2.9}$, then it is not hard to apply the following interpolation inequality
\begin{equation*}
\big[Q_\varepsilon\big]_{C^{0,\eta}(U_\varepsilon(x))}
\lesssim_{\eta}
\|Q_\varepsilon\|_{L^\infty(U_\varepsilon(x))}^{1-\eta}[
Q_\varepsilon]_{C^{0,1}(U_\varepsilon(x))}^{\eta}
\qquad(\text{where~}\eta\in[0,1])
\end{equation*}
to obtain the stated estimate $\eqref{pri:3.7}$, where we require
the established estimates \eqref{pri:2.8} and \eqref{pri:2.9}.
Therefore, the key ingredient is reduced to showing
the desired estimate $\eqref{pri:2.9}$.

\textbf{Step 3.} Arguments for the estimate $\eqref{pri:2.9}$.
For any $x\in\Omega$, the proof will be split into two cases: (1)
$\delta(x)\geq 2\varepsilon$; (2) $\delta(x)\leq 2\varepsilon$.
We first handle
the case (1). For any $\tilde{x}\in B_\varepsilon(x)$,
the Poisson formula $\eqref{eq:4.1}$ leads to
\begin{equation}\label{f:2.22}
\begin{aligned}
Q_\varepsilon(x)-Q_\varepsilon(\tilde{x}) = \varepsilon\int_{\partial\Omega}
n(y)\cdot a^{*\varepsilon}(y)\big(\nabla_y G_{\varepsilon}(\tilde{x},y)-\nabla_yG_{\varepsilon}(x,y)\big)
\phi^\varepsilon(y)dS(y),
\end{aligned}
\end{equation}
and this implies that
\begin{equation}\label{f:34}
\begin{aligned}
\big[Q_\varepsilon\big]_{C^{0,1}(B_\varepsilon(x))}
\lesssim \varepsilon
\int_{\partial\Omega}
[\nabla_y G_{\varepsilon}(\cdot,y)]_{C^{0,1}(B_\varepsilon(x))}
|\phi^\varepsilon(y)|dS(y).
\end{aligned}
\end{equation}
Note that $\mathcal{L}_\varepsilon(\nabla_yG_\varepsilon^{*}(y,\cdot))=0$
in $B_{3\varepsilon/2}(x)\subset\Omega$, where the uniform interior ball condition can guarantee this, and we also have the fact that  $\nabla_yG_\varepsilon^{*}(y,\cdot) =
\nabla_yG_\varepsilon(\cdot,y)$. The classical Lipschitz estimate leads to
\begin{equation}\label{f:4.1}
\begin{aligned}
\big[\nabla_y G_\varepsilon(\cdot,y)\big]_{C^{0,1}(B_{\varepsilon}(x))}
\lesssim_{\lambda,d} \frac{C_*(x)}{\varepsilon}
\Big(\dashint_{B_{3\varepsilon/2}(x)}|\nabla_y G_\varepsilon(\cdot,y)|^2
\Big)^{\frac{1}{2}},
\end{aligned}
\end{equation}
where $C_*(x) := [a]_{C^{0,\sigma}(B_1(x/\varepsilon))}^{\frac{1}{\sigma}(\frac{d}{2}+1)}$ and we refer the reader to \cite[Lemma A.3]{Josien-Otto22} for the details.
Then, taking $\langle|\cdot|^{p}\rangle^{\frac{1}{p}}$ on the both sides of
the above estimate, by using
H\"older's inequality, the assumption $\eqref{a:3}$, and Proposition
$\ref{P:4}$ in order, we obtain that
\begin{equation*}
\begin{aligned}
\big\langle\big[\nabla_y G_\varepsilon(\cdot,y)\big]_{C^{0,1}(B_{\varepsilon}(x))}^p
\big\rangle^{\frac{1}{p}}
&\lesssim \frac{1}{\varepsilon}
\Big(\dashint_{B_{\frac{3\varepsilon}{2}}(x)}
\langle|\nabla_y G_\varepsilon(\cdot,y)|^{\bar{p}}\rangle^{\frac{2}{\bar{p}}}
\Big)^{\frac{1}{2}}
\lesssim^{\eqref{pri:19}} \frac{\delta(x)}{\varepsilon|x-y|^d}.
\end{aligned}
\end{equation*}
After taking $\langle(\cdot)^{\underline{p}}\rangle^{\frac{1}{\underline{p}}}$ on the both sides of $\eqref{f:34}$, we plug the above estimate back into
its right-hand side, and this yields
\begin{equation}\label{f:2.25}
\big\langle\big[Q_\varepsilon\big]_{C^{0,1}(B_\varepsilon(x))}^{\underline{p}}
\big\rangle^{\frac{1}{\underline{p}}}
\lesssim \delta(x)
\int_{\partial\Omega}\frac{\langle|\phi^\varepsilon(y)|^{\bar{p}}
\rangle^{\frac{1}{\bar{p}}}}{|x-y|^d}dS(y)
\lesssim_{\lambda,\lambda_1,\lambda_2,d,m_0,R_0^{\tau}M_0,\tau,\bar{p}}
\mu_d(R_0/\varepsilon),
\end{equation}
where the last inequality follows from an analogous computation given for
${\eqref{f:35}}$.

Then, we proceed to handle the case (2). It suffices to show that
for any $x_0\in\partial\Omega$,
\begin{equation}\label{pri:3.4}
\big\langle[Q_\varepsilon]_{C^{0,1}(D_{3\varepsilon}(x_0))}^p
\big\rangle^{\frac{1}{p}}
\lesssim_{\lambda,\lambda_1,\lambda_2,d,m_0,R_0^{\tau}M_0,\tau,p}
\mu_d(R_0/\varepsilon).
\end{equation}
For any $x\in\Omega$ with $\delta(x)\leq 2\varepsilon$, there exists  $\bar{x}\in\partial\Omega$ such that $(B_\varepsilon(x)\cap\Omega)
\subset (B_{3\varepsilon}(\bar{x})\cap\Omega)$, and we therefore derive that
\begin{equation*}
 \big\langle\big[Q_\varepsilon\big]_{C^{0,1}(U_{\varepsilon}(x))}^p
\big\rangle^{\frac{1}{p}}
\leq
 \big\langle\big[Q_\varepsilon\big]_{C^{0,1}(D_{3\varepsilon}(\bar{x}))}^p
\big\rangle^{\frac{1}{p}}
\lesssim^{\eqref{pri:3.4}}
\mu_d(R_0/\varepsilon),
\end{equation*}
where we recall the notations $U_\varepsilon(x)=B_\varepsilon(x)\cap\Omega$ and $D_{3\varepsilon}(\bar{x})=B_{3\varepsilon}(\bar{x})\cap\Omega$.
Thus, this together with $\eqref{f:2.25}$ leads to the stated estimate
$\eqref{pri:2.9}$.

Now, we show the estimate $\eqref{pri:3.4}$. Define
$\Phi_{\varepsilon,i}:=\tilde{\Phi}_{\varepsilon,i}+x_i$. Then, there holds
\begin{equation}\label{pde:9}
\left\{\begin{aligned}
\mathcal{L}_\varepsilon(\Phi_{\varepsilon,i})
 &= 0
 \quad &\text{in}&\quad \Omega;\\
 \Phi_{\varepsilon,i} &= x_i
 \quad &\text{on}&\quad \partial\Omega,
\end{aligned}\right.
\quad i=1\cdots,d.
\end{equation}
By the definition of $Q_{\varepsilon,i}$ in Proposition $\ref{P:4}$,
we have $Q_{\varepsilon,i}= \Phi_{\varepsilon,i}-x_i-\varepsilon\phi^{\varepsilon}_i$.
For any $x_0\in\partial\Omega$, in view of the equations $\eqref{pde:9}$ and $\eqref{corrector}$, the classical Lipschitz estimate leads to
\begin{equation*}
\begin{aligned}
\big[Q_\varepsilon\big]_{C^{0,1}(D_{3\varepsilon}(x_0))}
&\leq \|\nabla\Phi_\varepsilon\|_{L^{\infty
}(D_{3\varepsilon}(x_0))}
+ 1 +
\|\nabla\phi\|_{L^\infty(D_{3}(x_0/\varepsilon))}\\
&\lesssim_{\lambda,d,m_0,R_0^{\tau}M_0,\tau} C_*(x_0)\bigg\{\big(\dashint_{D_{6\varepsilon}(x_0)}
|\nabla\Phi_\varepsilon|^2\big)^{\frac{1}{2}}+1
+\big(\dashint_{B_{6}(x_0/\varepsilon)}|\nabla\phi|^2\big)^{\frac{1}{2}}\bigg\},
\end{aligned}
\end{equation*}
where $C_*(x_0)$ is given by the same expression as in $\eqref{f:4.1}$.
Taking $\langle(\cdot)^p\rangle^{\frac{1}{p}}$ on the both sides above,
and then employing Lemma $\ref{lemma:4.1}$, Minkowski's inequality, and the fact that
$\langle|\nabla\phi|^{\beta}\rangle^{\frac{1}{\beta}}\lesssim 1$
in order, we obtain that
\begin{equation*}
\begin{aligned}
\big\langle\big[Q_\varepsilon\big]_{C^{0,1}(D_{3\varepsilon}(x_0))}^p
\big\rangle^{\frac{1}{p}}
&\lesssim^{\eqref{a:3}} \Big\langle
\big(\dashint_{D_{6\varepsilon}(x_0)}
|\nabla Q_\varepsilon|^2\big)^{\frac{\bar{p}}{2}}
\Big\rangle^{\frac{1}{\bar{p}}}
+ \Big\langle
\big(\dashint_{B_{6}(x_0/\varepsilon)}
|\nabla \phi|^2\big)^{\frac{\bar{p}}{2}}
\Big\rangle^{\frac{1}{\bar{p}}} + 1\\
&\lesssim^{\eqref{pri:3.6}}
\mu_d(R_0/\varepsilon)+1,
\end{aligned}
\end{equation*}
which gives us the desired estimate $\eqref{pri:3.4}$.

\medskip
\textbf{Step 4.} Arguments for the estimate $\eqref{pri:18}$.
Note that the estimate $\eqref{pri:2.9}$ has already
shown the Lipschitz estimate in $\eqref{pri:18}$, while we mainly address
the rest part.
In view of Corollary $\ref{corollary:2.2}$ and
Caccioppoli's inequality,
%We show that there exists the stationary random fields $x\mapsto C_{*}^{(1)}(x)$, depends on $[a]_{C^{0,\eta}(B_1(x))}$ and $r_*(x/\varepsilon)$, such that,
for any $z\in\Omega$, we have
\begin{equation}\label{f:29}
 |\nabla Q_\varepsilon(z)|
 \lesssim
 \frac{\mathcal{C}_*(z)}{\delta(z)}
 \Big(\dashint_{B_{\delta(z)}(z)}|Q_\varepsilon|^2\Big)^{\frac{1}{2}}.
\end{equation}

Based upon the estimate $\eqref{pri:2.8}$ in \textbf{Step 1},
taking $\langle(\cdot)^p\rangle^{\frac{1}{p}}$ on the both sides of
$\eqref{f:29}$, we obtain that
\begin{equation*}
\begin{aligned}
 \big\langle|\nabla Q_\varepsilon(z)|^p\big\rangle^{\frac{1}{p}}
 \lesssim^{\eqref{f:29}}
 \frac{1}{\delta(z)}
 \esssup_{x\in\Omega}\langle|Q_\varepsilon(x)|^{\bar{p}}\rangle^{\frac{1}{\bar{p}}}
 \lesssim^{\eqref{pri:2.8}} \frac{\varepsilon\mu_d(R_0/\varepsilon)}{\delta(z)},
\end{aligned}
\end{equation*}
which implies the stated estimate in $\eqref{pri:18}$, and this ends the proof.
\qed

\medskip

\textbf{Proof of Theorem $\ref{thm:2}$.} The estimates $\eqref{pri:1.a},
\eqref{pri:1.b}$, and $\eqref{pri:1.c}$ follow from $\eqref{pri:21-1}$,  $\eqref{pri:21-2}$, and $\eqref{pri:18}$, respectively. We have completed
the whole proof.
\qed

\section{Decay estimates of Green function's expansion}\label{sec:5}
\noindent
Note that we adopt the notations and arguments mostly from \cite{Clozeau-Josien-Otto-Xu}.
Since we merely handle the first-order expansion, a direct duality argument should be accessible
as in \cite{Avellaneda-Lin91,Kenig-Lin-Shen14},
and we refer the reader to the work by Blanc, Josien, and Le Bris \cite{Blanc-Josien-LeBris-20} on this aspect\footnote{It is known from personal communication with Marc Josien that the error estimate for the asymptotic expansion of Green's function at the level of mixed derivatives for the non-periodic case is contained in his PhD thesis (in French), specifically Theorem 1.1.4 on page 23 therein.}.
Instead, we choose another way to carry out a proof, and it relies on the following ingredient.

\begin{lemma}[Bella-Giunti-Otto's lemma of boundary version]\label{lemma:4}
Let $\Omega\subset\mathbb{R}^d$ be a bounded Lipschitz domain
and $r>0$. Suppose that
$\langle\cdot\rangle$ is an ensemble of $\lambda$-uniformly elliptic coefficient fields and $u_\varepsilon$ is
a weak solution of $\nabla \cdot a^\varepsilon\nabla u_\varepsilon
=0$ in $D_r$ with $u_\varepsilon = 0$ on $\Delta_r$. Then,
for all $p\in[1,\infty)$, we obtain that
\begin{equation}\label{pri:27}
 \Big\langle\Big(\dashint_{D_{r/2}}|\nabla u_\varepsilon|^2\Big)^{\frac{p}{2}}
 \Big\rangle^{\frac{1}{p}}
 \lesssim_{\lambda,d,m_0}
 \sup_{w\in C^{\infty}_0(D_r)}\frac{\big\langle|\dashint_{D_r}u_\varepsilon
 w|^p
\big\rangle^{\frac{1}{p}}}{r^3\sup|\nabla^2 w|},
\end{equation}
where $m_0$ is the Lipschitz character of $\Omega$.
\end{lemma}

\begin{proof}
The main idea of the proof is based upon that given in \cite[Lemma 4]{Clozeau-Josien-Otto-Xu}, originally developed in \cite[Lemma 4]{Bella-Giunti-Otto17}, and we provide a proof for the reader's convenience.
The whole proof has been divided into three steps.

\medskip

\textbf{Step 1.} Outline the main steps.
It is fine to assume $r=1$ by a rescaling argument. Let $\eta\in C^1_0(B_1)$ be a cut-off function satisfying $\eta = 1$ in $B_{3/4}$
and $\eta = 0$ outside $B_{4/5}$ with $|\nabla \eta|\lesssim_d 1$.
The proof starts from the following interpolation inequality for any
$v\in H_0^{2k}(D_1)$ with $k\in\mathbb{N}$, i.e.,
\begin{equation}\label{f:53}
\|\eta^{4k}\Delta^{2k}v\|_{L^2(D_1)}
\lesssim
 \|\eta^{4k+1}\nabla\Delta^{2k} v\|_{L^2(D_1)}^{\frac{2k}{2k+1}}
 \|v\|_{L^2(D_1)}^{\frac{1}{2k+1}}
 + \|v\|_{L^2(D_1)}.
\end{equation}
(We refer the reader to \cite[Lemma 4]{Bella-Giunti-Otto17} for the proof of $\eqref{f:53}$.) Then, one can continue to construct
the zero-Dirichlet boundary value problem for polyharmonic operator $\Delta^{2k}$ (with $k\geq 1$ being fixed later) as follows:
\begin{equation}\label{pde:21}
\left\{\begin{aligned}
\Delta^{2k}v &= u_\varepsilon &\quad&\text{in}\quad D_1;\\
\frac{\partial^m v}{\partial\nu^m}[v] &= 0 &\quad&\text{on}\quad\partial D_1,
\qquad m=0,\cdots,2k-1,
\end{aligned}\right.
\end{equation}
where $\partial^m/\partial\nu^m:=
\frac{\partial}{\partial\nu}(
\frac{\partial^{m-1}}{\partial\nu^{m-1}})$ and $\partial/\partial\nu=n\cdot\nabla$ with $\partial^0v/\partial\nu^0: = v$.
Thus, plugging the first line of $\eqref{pde:21}$ back into $\eqref{f:53}$, we have
\begin{equation*}
\begin{aligned}
\|\eta^{4k}u_\varepsilon\|_{L^2(D_1)}
&\lesssim
 \|\eta^{4k+1}\nabla u_\varepsilon\|_{L^2(D_1)}^{\frac{4k}{4k+1}}
 \|v\|_{L^2(D_1)}^{\frac{1}{4k+1}}
 + \|v\|_{L^2(D_1)} \\
& \lesssim_{\lambda,d}
 \|\eta^{4k}u_\varepsilon\|_{L^2(D_1)}^{\frac{4k}{4k+1}}
 \|v\|_{L^2(D_1)}^{\frac{1}{4k+1}}
 + \|v\|_{L^2(D_1)},
\end{aligned}
\end{equation*}
where the second inequality follows from Caccioppoli's inequality of $u_\varepsilon$.
By Young's inequality, there holds
\begin{equation}\label{f:5.1}
\|\eta^{4k}u_\varepsilon\|_{L^2(D_1)}
 \lesssim_{\lambda,k,d}
 \|v\|_{L^2(D_1)}.
\end{equation}
Moreover, using Caccioppoli's inequality again, we arrive at
\begin{equation}\label{f:54}
\int_{D_{1/2}}|\nabla u_\varepsilon|^2
\leq \int_{D_{1}}|\eta^{4k+1}\nabla u_\varepsilon|^2
\lesssim_{\lambda,d} \int_{D_{1}}|\eta^{4k}u_\varepsilon|^2
\lesssim^{\eqref{f:5.1}}
\int_{D_1} |v|^2.
\end{equation}
On the other hand, thanks to the complete orthonormal system of  eigenfunctions $\{w_j\}\subset H_{0}^{2k}(D_1)$ and eigenvalues
$\{\lambda_j\}_{j=1}^\infty$, which are associated by
$\Delta^{2k} w_j = \lambda_j w_j$ in $D_1$ with zero-Dirichlet boundary conditions (it's convenient to assume $\|w_j\|_{L^2(D_1)} =1$), we have
\begin{equation}\label{f:56}
\begin{aligned}
\int_{D_1} |v|^2
=\sum_{j}\frac{1}{\lambda_j^2}\big(\int_{D_1}u_\varepsilon w_j \big)^2
=\sum_{j}\frac{1}{\lambda_j}\frac{\big(\int_{D_1}u_\varepsilon w_j
\big)^2}{\int_{D_1}|\nabla^{2k} w_j|^2}.
\end{aligned}
\end{equation}
This together with $\eqref{f:54}$ implies
\begin{equation*}
\int_{D_{1/2}}|\nabla u_\varepsilon|^2
\lesssim
\sum_{j}\frac{1}{\lambda_j}\frac{\big(\int_{D_1}u_\varepsilon w_j
\big)^2}{\int_{D_1}|\nabla^{2k} w_j|^2}.
\end{equation*}
Putting $\langle(\cdot)^{\frac{p}{2}}\rangle$ on the both sides above, and
using H\"older's inequality, we obtain that
\begin{equation}\label{f:55}
\begin{aligned}
\Big\langle\big(\int_{D_{1/2}}|\nabla u_\varepsilon|^2\big)^{\frac{p}{2}}
\Big\rangle
\lesssim \Big(\sum_{j'}\frac{1}{\lambda_{j'}}\Big)^{\frac{p}{2}-1}
\sum_{j}\frac{1}{\lambda_j}\frac{\big\langle|\int_{D_1}u_\varepsilon w_j|^p
\big\rangle}{\big(\int_{D_1}|\nabla^{2k} w_j|^2\big)^{\frac{p}{2}}}.
\end{aligned}
\end{equation}
In addition, due to the finite trace of the inverse of $\Delta^{2k}$
with zero-Dirichlet boundary conditions, one can derive that
\begin{equation}\label{f:5.2}
   \sum_{j} \frac{1}{\lambda_j} <\infty,
\end{equation}
provided $2k>(d/2)$.
Plugging $\eqref{f:5.2}$ back into $\eqref{f:55}$, and then applying for the Sobolev embedding theorem leads to
\begin{equation*}
\begin{aligned}
\Big\langle\big(\int_{D_{1/2}}|\nabla u_\varepsilon|^2\big)^{\frac{p}{2}}
\Big\rangle
\lesssim \sup_{w\in H^k_0(D_1)}\frac{\big\langle|\int_{D_1}u_\varepsilon w|^p
\big\rangle}{\big(\sup|\nabla^2 w|\big)^p},
\end{aligned}
\end{equation*}
which finally implies the desired estimate $\eqref{pri:27}$ after the   rescaling argument.

\textbf{Step 2.} Show the identity $\eqref{f:56}$. Due to the
complete orthonormal system of  eigenfunctions $\{w_j\}$ and the condition
$\|w_j\|_{L^2(D_1)} = 1$, we first have
\begin{equation}\label{f:6.8}
\begin{aligned}
\int_{D_1} |v|^2
= \sum_j \big(\int_{D_1} vw_j\big)^2
&= \sum_j \frac{1}{\lambda_j^2}\big(\int_{D_1} v\Delta^{2k} w_j\big)^2 \\
&= \sum_j \frac{1}{\lambda_j^2}\big(\int_{D_1}\Delta^{2k} v w_j\big)^2
=^{\eqref{pde:21}} \sum_j \frac{1}{\lambda_j^2}\big(\int_{D_1} u_\varepsilon w_j\big)^2,
\end{aligned}
\end{equation}
where we merely employ the equation $\Delta^{2k}w_j = \lambda_j w_j$ in $D_1$ for the second equality and integration by parts for the third one above. Then, by noting that
\begin{equation*}
\begin{aligned}
 \lambda_j = \int_{D_1}|\nabla^{2k}w_j|^2,
\end{aligned}
\end{equation*}
we can plug this back into $\eqref{f:6.8}$ to derive the desired identity
$\eqref{f:56}$.

\textbf{Step 3.} Show the estimate $\eqref{f:5.2}$.
Let $G(\cdot,x)$ be the Green function associated with $\Delta^{2k}$
with zero-Dirichlet boundary conditions, and the notation $(\cdot,\cdot)$
represent either the action of a distribution on a test function
or the scalar product in $L^2(D_1)$. By the definition of Green function
and eigenfunctions, we obtain that
\begin{equation}\label{f:5.3}
\begin{aligned}
\sum_{k=1}^\infty \frac{1}{\lambda_k}
& = \sum_{k=1}^\infty \big((\Delta^{2k})^{-1}w_k, w_k\big)
 = \sum_{k=1}^\infty \Big(\int_{D_1}G(\cdot,y)w_k(y)dy, w_k\Big)\\
%&= \Big(\int_{D_1}G(\cdot,y)\sum_{k=1}^\infty w_k(y) w_k(\cdot) dy\Big)
&= \int_{D_1} \big(G(\cdot,y),\delta_y\big) dy
 = \int_{D_1} G(y,y) dy.
\end{aligned}
\end{equation}
where we employ $\sum_{k=1}^{\infty}w_k(y)w_k=\delta_{y}$ in the third
equality. If the above calculation of $\eqref{f:5.3}$ is meaningful,
it is required $\delta_x\in H^{-2k}(D_1)$, and therefore we need
the precondition $2k>(d/2)$ to guarantee $\eqref{f:5.2}$.
This ends the whole proof.
\end{proof}

\begin{lemma}[annealed Calder\'on-Zygmund estimates]\label{P:7*}
Let $\Omega$ be a bounded (uniform) $C^1$ domain, and $\varepsilon\in(0,1]$.
Let $1<p,q<\infty$. Suppose that $\langle\cdot\rangle$ satisfies the spectral gap condition $\eqref{a:2}$, and
the admissible coefficients satisfy the smoothness condition $\eqref{a:3}$. Let $u_\varepsilon$ be a weak solution to
\begin{equation}\label{pde:20}
\left\{\begin{aligned}
-\nabla\cdot a^\varepsilon\nabla u_\varepsilon &= \nabla\cdot f
&\quad&\text{in}~~\Omega;\\
u_\varepsilon &= 0
&\quad&\text{on}~~\partial\Omega.
\end{aligned}\right.
\end{equation}
Then, for any $p<\bar{p}<\infty$ and $\omega\in A_q$\footnote{
The notation $A_q$ is known as the Muckenhoupt weight class, and its definition is now standard (see e.g. \cite[Definition 1.3]{Wang-Xu23-1}).}, there holds
\begin{equation}\label{pri:12}
\Big(\int_{\Omega}\big\langle|\nabla u_\varepsilon|^p
\big\rangle^{\frac{q}{p}}\omega\Big)^{\frac{1}{q}}
\lesssim \Big(\int_{\Omega}
\big\langle|f|^{\bar{p}}\big\rangle^{\frac{q}{\bar{p}}}\omega
\Big)^{\frac{1}{q}},
\end{equation}
where the multiplicative constant above depends on
$\lambda$, $\lambda_1$, $\lambda_2$, $d$,
$q$, $p$, $\bar{p}$, $m_0$ and $\omega(t)$ in $\eqref{boundary-2}$ at most.
\end{lemma}

\begin{proof}
See \cite[Theorem 1.4]{Wang-Xu23-1}.
\end{proof}

%\begin{proposition}[annealed errors]\label{P:6}
%Let $\Omega\ni \{0\}$ be a bounded domain of $C^{1,1}$ type with $\tau\in(0,1)$
%and $\varepsilon\in(0,1]$.
%Suppose that $\langle\cdot\rangle$ satisfies Spectral Gap conditions, and
%the coefficient $a$ additionally satisfies smoothness condition. Let
%$\{\Phi_{\varepsilon,i}\}_{i=1}^d$ be the solution of
%the equations $\eqref{pde:9}$ (we denote its adjoint one by
%$\Phi_{\varepsilon,i}^{*}$), and
%\begin{equation}\label{eq:6.1}
%\mathcal{E}_{\varepsilon}(x,y)
%:= \nabla\nabla G_\varepsilon(x,y)
%+\nabla\Phi_{\varepsilon,i}(x)
%\nabla\Phi_{\varepsilon,j}^{*}(y)\partial_{ij}\overline{G}(x-y),
%\end{equation}
%where $-\nabla\cdot\bar{a}\nabla \overline{G} = \delta_0$ in $\Omega$ with zero boundary conditions.
%Then there holds
%\begin{equation}
% |x-y|^{d+1}
% \big\langle|\mathcal{E}_{\varepsilon}(x,y)|^p\big\rangle^{\frac{1}{p}}
% \lesssim \varepsilon\mu_d^2(R_0/\varepsilon)\ln(|x-y|/\varepsilon+2)
%\end{equation}
%for any $p<\infty$ and any $x,y\in\Omega$ with $|x-y|\geq \varepsilon$.
%\end{proposition}

\textbf{Proof of Theorem $\ref{thm:3}$.}
Throughout the proof, we fix two ``base points'' $x_0,y_0\in\Omega$  with $|x_0-y_0|\geq 2\varepsilon$, and the whole proof will be split into five steps.

\textbf{Step 1.} The first reduction on the annealed error $\eqref{pri:6.1}$.
The introduced notation $\mathcal{E}_\varepsilon$ is merely a truncated version of the error of (two-scale) expansion of Green functions, on the level of the mixed derivatives, and we start from introducing the error of the expansion associated with boundary correctors as follows
\begin{equation}\label{eq:6.2}
  w_{\varepsilon}(x,y)
  := G_\varepsilon(x,y)
  -(1+\tilde{\Phi}_{\varepsilon,i}(x)\partial_i)
  (1+\tilde{\Phi}_{\varepsilon,j}^{*}(y)\partial_j)\overline{G}(x,y),
\end{equation}
where $\tilde{\Phi}_{\varepsilon,i}$ is defined in $\eqref{pde:10}$, and
we use the convention that the variable corresponding to the subscript of the derivatives is the same as the function with the same subscript (i.e.,
we have $\partial_i = \partial_{x_i}$ and
$\partial_j = \partial_{x_j}$ in $\eqref{eq:6.2}$ by this convention).
Recalling the expression of $\mathcal{E}_\varepsilon$ in $\eqref{eq:6.1}$,
a routine computation leads to the mixed derivatives of $w_{\varepsilon}(x,y)$, i.e.,
\begin{equation}\label{eq:6.3}
\begin{aligned}
&\nabla_x\nabla_y w_{\varepsilon}(x,y)
=\mathcal{E}_{\varepsilon}(x,y)
-\nabla\Phi_{\varepsilon,i}(x)
\tilde{\Phi}_{\varepsilon,j}^{*}(y)
e_n\partial_{ijn}\overline{G}(x,y)\\
&+\tilde{\Phi}_{\varepsilon,i}(x)
\nabla\Phi_{\varepsilon,j}^{*}(y)e_m\partial_{imj}\overline{G}(x,y)
+\tilde{\Phi}_{\varepsilon,i}^{}(x)
\tilde{\Phi}_{\varepsilon,j}^{*}(y)e_me_n\partial_{imjn}\overline{G}(x,y).
\end{aligned}
\end{equation}
Although forth-order derivatives of
$\overline{G}$ appeared above, only second-order derivatives are involved for variables $x$ or $y$\footnote{By the convention, we have
$\partial_{ijn}=\partial_{x_iy_jy_n}$,
$\partial_{imj}=\partial_{x_ix_my_j}$ and
$\partial_{imjn}=\partial_{x_ix_my_jy_n}$ in $\eqref{eq:6.3}$.}, and thus the assumption $\partial\Omega\in C^{2,\tau}$ is required.

Let $r:=|x_0-y_0|/2$. It is fine to assume $\delta(y_0)\leq r/2$, and  otherwise
the interior estimates are sufficient for the proof. Let $\bar{y}_0\in\partial\Omega$ be such that $\delta(y_0)=|y_0-\bar{y}_0|$.
Appealing to Proposition $\ref{P:5}$ and decays of Green function  $\overline{G}$, we have
\begin{equation}\label{pri:25}
\begin{aligned}
&\big\langle\big|\mathcal{E}_{\varepsilon}(x_0,y_0) -
\nabla_x\nabla_y w_{\varepsilon}(x_0,y_0)\big|^p
\big\rangle^{\frac{1}{p}}\\
&\lesssim^{\eqref{pri:21-1},\eqref{pri:21-2}} \varepsilon\mu_d^2(R_0/\varepsilon)r^{-d-1}
+ \varepsilon^2\mu_d^2(R_0/\varepsilon)r^{-d-2}
\lesssim \varepsilon\mu_d^2(R_0/\varepsilon)r^{-d-1}.
\end{aligned}
\end{equation}
In this regard, the desired estimate
$\eqref{pri:6.1}$ is reduced to showing
\begin{equation}\label{pri:6.2}
\big\langle |\nabla\nabla w_\varepsilon(x_0,y_0)|^p\big\rangle^{\frac{1}{p}}
\lesssim \mu_d(R_0/\varepsilon)\varepsilon\ln(1/\varepsilon)\ln(1+r)r^{-d-1}.
\end{equation}

For the ease of the statement, we introduce the following notation
\begin{equation}\label{eq:6.5}
\bar{u}(x,\cdot)
:=(1+\tilde{\Phi}_{\varepsilon,i}(x)\partial_i)\overline{G}(x,\cdot).
\end{equation}
Hence, we can rewrite $\eqref{eq:6.3}$ as
\begin{equation*}
 w_{\varepsilon}(x,y)
 = G_\varepsilon(x,y)-(1+\tilde{\Phi}_{\varepsilon,j}^{*}(y)\partial_j)
 \bar{u}(x,y).
\end{equation*}
Taking derivative with respect to $x_k$-variable ($k=1,\cdots,d$), and
then fix $x_0\in \Omega$, we denoted by
\begin{equation}\label{eq:6.6}
 V_{\varepsilon}^{k}(x_0,y)
 := \partial_{x_k}G_\varepsilon(x_0,y)-(1-\tilde{\Phi}_{\varepsilon,j}^{*}(y)
 \partial_j)\partial_{x_k}
 \bar{u}(x_0,y).
\end{equation}
By noting that $\partial_{x_k}
\bar{u}(x_0,\cdot) = 0$ on $\partial\Omega$,
there holds
\begin{equation}\label{pde:11}
\left\{\begin{aligned}
-\nabla\cdot a^{*\varepsilon}\nabla
V_{\varepsilon}^{k}(x_0,\cdot)
&=\nabla\cdot h_{\varepsilon}^k(x_0,\cdot)
+ g_{\varepsilon}^k(x_0,\cdot)
\quad &\text{in}&\quad \Omega;\\
V_{\varepsilon}^{k}(x_0,\cdot)
& =0
\quad &\text{on}&\quad \partial\Omega,
\end{aligned}\right.
\end{equation}
where
\begin{equation}\label{eq:6.4}
\begin{aligned}
h_{\varepsilon}^k(x_0,\cdot)
&:= \big(a^{*\varepsilon}\tilde{\Phi}_{\varepsilon,j}^{*}
-\varepsilon\sigma_j^{*\varepsilon}\big)
\nabla\partial_j\partial_{x_k}\bar{u}(x_0,\cdot);\\
g_{\varepsilon}^{k}(x_0,\cdot)
&:= a^{*\varepsilon}\nabla \tilde{Q}_{\varepsilon,j}^{*}
\cdot\nabla\partial_j\partial_{x_k}\bar{u}(x_0,\cdot).
\end{aligned}
\end{equation}

To fulfill our purpose, we plan to split the equation $\eqref{pde:11}$ into two parts. Let $\mathbf{1}_{D_r(\bar{y}_0)}$ be the indicator function of
$D_r(\bar{y}_0)$, and the near-field part $V_{\varepsilon}^{k,(1)}(x_0,\cdot)$ is  the solution to the following equations:
\begin{equation}\label{pde:12}
\left\{\begin{aligned}
-\nabla\cdot a^{*\varepsilon}\nabla
V_{\varepsilon}^{k,(1)}(x_0,\cdot)
&=\nabla\cdot
h_{\varepsilon}^{k,(1)}(x_0,\cdot)
+ g_{\varepsilon}^{k,(1)}(x_0,\cdot)
\quad &\text{in}&\quad \Omega;\\
V_{\varepsilon}^{k,(1)}(x_0,\cdot)
& =0
\quad &\text{on}&\quad \partial\Omega,
\end{aligned}\right.
\end{equation}
where $h_{\varepsilon}^{k,(1)}(x_0,\cdot):=
\mathbf{1}_{D_r(\bar{y}_0)} h_{\varepsilon}^k(x_0,\cdot)$ and
$g_{\varepsilon}^{k,(1)}(x_0,\cdot)
:=\mathbf{1}_{D_r(\bar{y}_0)}
g_{\varepsilon}^{k}(x_0,\cdot)$.
%\begin{equation*}
%\begin{aligned}
%h_{\varepsilon}^{k,(1)}(x_0,\cdot):=
%\mathbf{1}_{U_r(y_0)} h_{\varepsilon}^k(x_0,\cdot);
%\qquad
%g_{\varepsilon}^{k,(1)}(x_0,\cdot)
%:=\mathbf{1}_{U_r(y_0)}
%g_{\varepsilon}^{k}(x_0,\cdot).
%\end{aligned}
%\end{equation*}

Let $V_{\varepsilon}^{k,(2)}(x_0,\cdot)
:=
V_{\varepsilon}^{k}(x_0,\cdot)
-V_{\varepsilon}^{k,(1)}(x_0,\cdot)$ be far-field part,
which satisfies
\begin{equation}\label{pde:13}
-\nabla\cdot a^{*\varepsilon}\nabla
V_{\varepsilon}^{k,(2)}(x_0,\cdot)
= 0
\quad
\text{in} \quad D_r(\bar{y}_0).
\end{equation}

In terms of the equations $\eqref{pde:20}$, we study the near-field part and manage to show
\begin{equation}\label{f:51}
\begin{aligned}
\langle|\partial_m V_{\varepsilon}^{k,(1)}(x_0,y_0)|^p\rangle^{\frac{1}{p}}
\lesssim
\varepsilon\mu_d^2(R_0/\varepsilon)\ln(r/\varepsilon)
r^{-d-1}.
\end{aligned}
\end{equation}
Concerning the equation $\eqref{pde:13}$,
we proceed to study the far-field part and establish
\begin{equation}\label{f:52}
\langle|\nabla V_{\varepsilon}^{k,(2)}(x_0,y_0)|^p\rangle^{\frac{1}{p}}
\lesssim \varepsilon\mu_d^2(R_0/\varepsilon)\ln(r/\varepsilon)r^{-1-d}.
\end{equation}

Admitting the above two estimates for a while,
therefore, the stated estimate $\eqref{pri:6.2}$ can be reduced to
showing
\begin{equation*}
\begin{aligned}
\langle|\nabla V_{\varepsilon}^{k}(x_0,y_0)|^p\rangle^{\frac{1}{p}}
&\leq  \langle|\nabla V_{\varepsilon}^{k,(1)}(x_0,y_0)|^p\rangle^{\frac{1}{p}}
+\langle|\nabla V_{\varepsilon}^{k,(2)}(x_0,y_0)|^p\rangle^{\frac{1}{p}}\\
&\lesssim^{\eqref{f:51},\eqref{f:52}}
\mu_d^2(R_0/\varepsilon)
\varepsilon\ln(1/\varepsilon)
\ln(1+r)r^{-1-d}.
\end{aligned}
\end{equation*}

\textbf{Step 2.} Show the estimate $\eqref{f:51}$.
Appealing to the representation of the solution via Green's function,
for any $y\in D_r(\bar{y}_0)$ (and $m=1,\cdots,d$), there holds
\begin{equation}\label{f:6.1}
\begin{aligned}
&\partial_{y_{m}} V_{\varepsilon}^{k,(1)}
(x_0,y)
= \int_{U_\varepsilon(y)}
\nabla \partial_{y_{m}}G_\varepsilon^*(y,z)\cdot\big(h_{\varepsilon}^{k}(x_0,z)
-h_{\varepsilon}^{k}(x_0,y)\big) dz\\
&+  h_{\varepsilon}^{k}(x_0,y)\cdot \int_{\partial U_\varepsilon(y)}
n(z) \partial_{y_{m}} G_\varepsilon^*(y,z)dS(z)
+  \int_{D_r(\bar{y}_0)}
\partial_{y_{m}} G_\varepsilon^*(y,z)g_{\varepsilon}^{k}(x_0,z)dz\\
&+  \int_{D_r(\bar{y}_0)\setminus U_{\varepsilon}(y)}
\nabla\partial_{y_{m}} G_\varepsilon^*(y,z)\cdot
h_{\varepsilon}^{k}(x_0,z) dz:=I_1+I_2+I_3+I_4,
\end{aligned}
\end{equation}
where $n$ is the outward unit normal vector to $\partial U_\varepsilon(y)$.
For further calculations, we need a preparation. In terms of decays of Green function $\overline{G}$, Proposition $\ref{P:5}$, and $n=1,2$, by recalling the expression $\eqref{eq:6.5}$, we have
\begin{equation}\label{f:43}
\begin{aligned}
&\sup_{D_r(\bar{y}_0)}\big\langle
|\nabla^n\partial_{x_k}\bar{u}(x_0,\cdot)|^{p}\big\rangle^{\frac{1}{p}}
\lesssim \big\langle
|\partial_k\tilde{\Phi}_{\varepsilon,i}(x_0)|^p\big\rangle^{\frac{1}{p}}
\sup_{D_r(\bar{y}_0)}|\nabla^n\partial_i\overline{G}(x_0,\cdot)|\\
&\qquad\qquad\qquad+\big\langle
|\tilde{\Phi}_{\varepsilon,i}(x_0)|^p\big\rangle^{\frac{1}{p}}
\sup_{D_r(\bar{y}_0)}|\nabla^n\partial_{ki}\overline{G}(x_0,\cdot)|\\
&\lesssim^{\eqref{pri:21-1},\eqref{pri:18},\eqref{pri:8.2},\eqref{pri:8.4},\eqref{pri:8.5}}_{
\lambda,\lambda_1,\lambda_2,d,R_0^{\tau}M_0,R_0^{1+\tau}M_0,\tau,p}\mu_d(R_0/\varepsilon)r^{1-d-n}
+\varepsilon\mu_d(R_0/\varepsilon) r^{-d-n}
\lesssim \mu_d(R_0/\varepsilon)r^{1-d-n},
\end{aligned}
\end{equation}
where we note that the estimated constant is independent of the scale transformation\footnote{Although the present proof does not rely on this property,
it will be convenient in some applications.}.
From the expression of $h_{\varepsilon}^k$ in $\eqref{eq:6.4}$ we
further have
\begin{equation}\label{f:36}
\begin{aligned}
\esssup_{y\in D_r(\bar{y}_0)}
\Big\{\big\langle|h_{\varepsilon}^k(x_0,y)|^p
\big\rangle^{\frac{1}{p}}
&+\varepsilon^\eta\big\langle
[h_{\varepsilon}^k(x_0,\cdot)]_{C^{0,\eta}(U_\varepsilon(y))}^{p}
\big\rangle^{\frac{1}{p}}\Big\}\\
&\lesssim^{\eqref{pri:21-1},\eqref{pri:21-2},\eqref{f:43}}
\varepsilon\mu_d^2(R_0/\varepsilon)
r^{-d-1},
\end{aligned}
\end{equation}
and
\begin{equation}\label{f:36*}
%\esssup_{y\in D_r(\bar{y}_0)}
\big\langle |g_\varepsilon^{k}(x_0,y)|^p\big\rangle^{\frac{1}{p}}
\lesssim^{\eqref{pri:18},\eqref{f:43}} \mu_{d}^2(R_0/\varepsilon)r^{-1-d}
\min\Big\{\frac{\varepsilon}{\delta(y)},1\Big\}
\qquad\forall y\in D_r(\bar{y}_0).
\end{equation}

Now, we can compute the right-hand side of $\eqref{f:6.1}$ term by term.
On account of the estimates $\eqref{pri:19}$ and $\eqref{f:36}$, we can start from
\begin{equation}\label{f:6.2}
\begin{aligned}
&\langle I_1^p\rangle^{\frac{1}{p}}
\lesssim \varepsilon^{1-\eta}\mu_d^2(R_0/\varepsilon)
r^{-d-1}\int_{U_\varepsilon(y)}\frac{dz}{|y-z|^{d-\eta}}
\lesssim
\varepsilon\mu_d^2(R_0/\varepsilon)
r^{-d-1};\\
&\langle I_2^p\rangle^{\frac{1}{p}}
\lesssim
\varepsilon\mu_d^2(R_0/\varepsilon)
r^{-d-1}\int_{\partial U_\varepsilon(y)\setminus\partial\Omega}
\frac{dS(z)}{|z-y|^{d-1}}
\lesssim \varepsilon\mu_d^2(R_0/\varepsilon)
r^{-d-1};\\
&\langle I_4^p\rangle^{\frac{1}{p}}
\lesssim \varepsilon\mu_d^2(R_0/\varepsilon)
r^{-d-1}\int_{D_r(\bar{y}_0)\setminus U_\varepsilon(y)}
\frac{dz}{|z-y|^d}
\lesssim \varepsilon\mu_d^2(R_0/\varepsilon)
r^{-d-1}\ln(r/\varepsilon).
\end{aligned}
\end{equation}
The reminder term is $I_3$, and we can similarly divide it into two parts:
\begin{equation*}
I_3=\Big\{\int_{D_r(\bar{y}_0)\setminus D_\varepsilon(y)}+
\int_{D_\varepsilon(y)}\Big\}
\big|\partial_{y_m} G_\varepsilon^*(y,z)g_{\varepsilon}^{k}(x_0,z)\big|dz
=: I_{31} + I_{32}.
\end{equation*}
Therefore, it follows from the estimates $\eqref{pri:19}$ and $\eqref{f:36*}$ that
\begin{equation}\label{f:6.9}
\begin{aligned}
& \langle |I_{31}|^p\rangle^{\frac{1}{p}}
\lesssim
%\mu_d^2(R_0/\varepsilon)
%\varepsilon r^{-d-1}\int_{D_r(y_0)\setminus D_\varepsilon(y)}
%\langle|\nabla G_\varepsilon(y,z)|^{p_1}\rangle^{\frac{1}{p_1}}
%\frac{dz}{\delta(z)}
\varepsilon\mu_d^2(R_0/\varepsilon)
 r^{-d-1}\int_{D_r(\bar{y}_0)\setminus U_\varepsilon(y)}
\frac{dz}{|z-y|^d}
\lesssim \frac{\varepsilon\mu_d^2(R_0/\varepsilon)\ln(r/\varepsilon)}{r^{d+1}};\\
& \langle |I_{32}|^p\rangle^{\frac{1}{p}}
\lesssim \mu_{d}^2(R_0/\varepsilon)r^{-d-1}\int_{U_{\varepsilon}(y)}
\frac{dz}{|z-y|^{d-1}}
\lesssim \varepsilon\mu_d^2(R_0/\varepsilon)
 r^{-d-1}.
\end{aligned}
\end{equation}

As a result, combining the estimates $\eqref{f:6.1}$, $\eqref{f:6.2}$, and $\eqref{f:6.9}$ yields the following estimate slightly stronger than $\eqref{f:51}$, i.e.,
\begin{equation}\label{pri:6.3}
\sup_{D_r(\bar{y}_0)\atop m=1,\cdots,d}\langle|\partial_m V_{\varepsilon}^{k,(1)}(x_0,\cdot)|^p\rangle^{\frac{1}{p}}
\lesssim \varepsilon\mu_d^2(R_0/\varepsilon)\ln(r/\varepsilon)r^{-d-1},
\end{equation}
which offers us the desired result $\eqref{f:51}$ by simply noting
that $y_0\in D_r(\bar{y}_0)$.

\textbf{Step 3.}
Arguments for $\eqref{f:52}$.
In view of the equation $\eqref{pde:13}$, the
Lipschitz regularity estimate  yields
\begin{equation}\label{f:6.10}
\begin{aligned}
\langle|\nabla V_{\varepsilon}^{k,(2)}
(x_0,y_0)|^p\rangle^{\frac{1}{p}}
&\lesssim^{\eqref{pri:2.10}}
\Big\langle\big(\dashint_{B_{\delta(y_0)}(y_0)}|\nabla V_{\varepsilon}^{k,(2)}(x_0,\cdot)|^2
\big)^{\frac{p_1}{2}}\Big\rangle^{\frac{1}{p_1}}\\
&\lesssim
\Big\langle\big(\dashint_{D_{r}(\bar{y}_0)}|\nabla V_{\varepsilon}^{k,(2)}(x_0,\cdot)|^2
\big)^{\frac{p_2}{2}}\Big\rangle^{\frac{1}{p_2}}:=K,
\end{aligned}
\end{equation}
where $p_2\geq p_1>p$, and
one also employs a large-scale estimate in the second step.
Let $\mathbf{X}:=C^\infty_0(D_r(\bar{y}_0))$, and
appealing to Lemma $\ref{lemma:4}$ one can derive that
\begin{equation}\label{f:6.4}
\begin{aligned}
K
%\lesssim^{\eqref{pri:2.10}}
%\Big\langle\big(\dashint_{B_{\delta(y_0)}(y_0)}|\nabla V_{\varepsilon}^{k,(2)}(x_0,\cdot)|^2
%\big)^{\frac{\bar{p}}{2}}\Big\rangle^{\frac{1}{\bar{p}}}\\
\lesssim^{\eqref{pri:27}}
\sup_{g\in \mathbf{X}}
\frac{\big\langle\big|\dashint_{D_r(\bar{y}_0)}
V_{\varepsilon}^{k}(x_0,\cdot)g\big|^{p_2}
\big\rangle^{\frac{1}{p_2}}}{r^3\sup|\nabla^2 g|}
+\sup_{g\in \mathbf{X}}
\frac{\big\langle\big|\dashint_{D_r(\bar{y}_0)}
V_{\varepsilon}^{k,(1)}(x_0,\cdot)g\big|^{p_2}
\big\rangle^{\frac{1}{p_2}}}{r^3\sup|\nabla^2 g|}.
\end{aligned}
\end{equation}

Let us start by acknowledging that
\begin{equation}\label{pri:6.4}
\begin{aligned}
\Big\langle\big|\dashint_{D_r(\bar{y}_0)}
V_{\varepsilon}^{k}(x_0,\cdot)g\big|^{p_2}\Big\rangle^{\frac{1}{p_2}}
\lesssim \varepsilon\mu_d^2(R_0/\varepsilon) \ln(r/\varepsilon)
r^{2-d}\sup_{D_r(\bar{y}_0)}|\nabla^2 g|
\end{aligned}
\end{equation}
is true for a while, and then the first term in the right-hand side of $\eqref{f:6.4}$ offers us
\begin{equation}\label{f:49}
\sup_{g\in \mathbf{X}}
\frac{\big\langle\big|\dashint_{D_r(\bar{y}_0)}
V_{\varepsilon}^{k}(x_0,\cdot)g\big|^{p_2}
\big\rangle^{\frac{1}{p_2}}}{r^3\sup|\nabla^2 g|}
\lesssim^{\eqref{pri:6.4}} \varepsilon \mu_d^2(R_0/\varepsilon)
\ln(r/\varepsilon)
r^{-d-1}.
\end{equation}

In terms of the second term in the right-hand side of $\eqref{f:6.4}$, by using Minkowski's inequality and Poincar\'e's inequality, we have the following computations
\begin{equation*}
\begin{aligned}
\big\langle\big|\dashint_{D_r(\bar{y}_0)}
V_{\varepsilon}^{k,(1)}(x_0,\cdot)g\big|^{p_2}
\big\rangle^{\frac{1}{p_2}}
&
\lesssim r^3\sup_{D_r(\bar{y}_0)}
\big\langle|\nabla V_{\varepsilon}^{k,(1)}(x_0,\cdot)|^{p_2}\big\rangle^{\frac{1}{p_2}}
\sup_{D_r(\bar{y}_0)}|\nabla^2 g|\\
&\lesssim^{\eqref{pri:6.3}}
\varepsilon\mu_d^2(R_0/\varepsilon)\ln(r/\varepsilon)r^{2-d}
\sup_{D_r(\bar{y}_0)}|\nabla^2 g|,
\end{aligned}
\end{equation*}
which implies
\begin{equation}\label{f:50}
\sup_{g\in \mathbf{X}}
\frac{\big\langle\big|\dashint_{D_r(y_0)}
V_{\varepsilon}^{k,(1)}(x_0,\cdot)g\big|^{\bar{p}}
\big\rangle^{\frac{1}{\bar{p}}}}{r^3\sup|\nabla^2 g|}
\lesssim \varepsilon\mu_d^2(R_0/\varepsilon)\ln(r/\varepsilon)r^{-d-1}.
\end{equation}

Consequently, combining the estimates $\eqref{f:6.10}$, $\eqref{f:6.4}$, $\eqref{f:49}$, and $\eqref{f:50}$ we have the desired estimate $\eqref{f:52}$.
The rest of the proof is devoted to the estimate $\eqref{pri:6.4}$.

\textbf{Step 4.} Make a reduction on the estimate $\eqref{pri:6.4}$.
For any $g\in \mathbf{X}$, let us consider the following Dirichlet problems:
\begin{equation}\label{pde:14}
(1)\left\{\begin{aligned}
-\nabla\cdot a^\varepsilon\nabla
u_\varepsilon &= g  \quad\text{in}\quad\Omega;\\
u_\varepsilon &= 0  \quad\text{on}\quad\partial\Omega,
\end{aligned}\right.
\qquad
(2)\left\{\begin{aligned}
-\nabla\cdot \bar{a}\nabla
u_0 &= g  \quad\text{in}\quad\Omega;\\
u_0 &= 0  \quad\text{on}\quad\partial\Omega.
\end{aligned}\right.
\end{equation}
Obviously, (2) is the homogenized equations of (1). Therefore, we can investigate the homogenization error, which can be represented by
Green functions with the help of $\eqref{eq:6.5}$, i.e., for any $x\in\Omega$,
\begin{equation}\label{pde:7}
\begin{aligned}
  z_\varepsilon(x)  &:= u_\varepsilon(x) - (1+\tilde{\Phi}_{\varepsilon,i}(x)\partial_i)u_{0}(x)
  =\int_{\Omega}\big(G_\varepsilon(x,\cdot)-\bar{u}(x,\cdot)\big)g.
\end{aligned}
\end{equation}
On the one hand, this offers us $\partial_{x_k}z_\varepsilon(x)
=\int_{\Omega}\big(\partial_{x_k} G_\varepsilon(x,\cdot)-\partial_{x_k}\bar{u}(x,\cdot)\big)g$.
On the other hand, by recalling the expression of  $V_{\varepsilon}^{k}(x_0,\cdot)$ in $\eqref{eq:6.6}$, we have
the following formula:
\begin{equation}\label{f:48}
\begin{aligned}
\int_{\Omega} V_{\varepsilon}^{k}(x_0,\cdot)g
&=\int_{\Omega}
\big[\partial_{x_k} G_\varepsilon(x_0,\cdot)
 -(1-\tilde{\Phi}_{\varepsilon,j}^{*}\partial_j)
 \partial_{x_k}\bar{u}(x_0,\cdot)\big] g\\
&=\int_{\Omega}
\big[\partial_{x_k} G_\varepsilon(x_0,\cdot)
 -\partial_{x_k}\bar{u}(x_0,\cdot)
 +\tilde{\Phi}_{\varepsilon,j}^{*}\partial_j
 \partial_{x_k}\bar{u}(x_0,\cdot)\big] g\\
& = \partial_{x_k} z_\varepsilon(x_0)
+ \int_{\Omega}
 \tilde{\Phi}_{\varepsilon,j}^{*}\partial_j
 \partial_{x_k}\bar{u}(x_0,\cdot)g.
\end{aligned}
\end{equation}

Therefore, the desired estimate $\eqref{pri:6.4}$ is reduced to estimating
two parts in the right-hand side of $\eqref{f:48}$.
Again, let us acknowledge that
\begin{equation}\label{pri:23}
\langle|\nabla z_\varepsilon(x_0)|^{p_2}\rangle^{\frac{1}{p_2}}
\lesssim
\varepsilon\mu_d(R_0/\varepsilon) \ln(r/\varepsilon)
r^{2}\sup_{D_r(\bar{y}_0)}|\nabla^2 g|
\end{equation}
is true for a moment. With the aid of Propositions $\ref{P:4}$ and $\ref{P:5}$, the easier part is
\begin{equation}\label{pri:24}
\begin{aligned}
\Big\langle\big|\int_{\Omega}
 \tilde{\Phi}_{\varepsilon,j}^{*}(y)\partial_j
 \partial_{x_k}\bar{u}(x_0,\cdot)g\big|^{p_2}\Big\rangle^{\frac{1}{p_2}}
\lesssim^{\eqref{f:43},\eqref{pri:21-1}}
\varepsilon\mu_d^2(R_0/\varepsilon) r^{2}\sup_{D_r(\bar{y}_0)}|\nabla^2 g|.
\end{aligned}
\end{equation}
Thus, combining the estimates $\eqref{f:48},\eqref{pri:23}$, and $\eqref{pri:24}$ leads to the stated estimate $\eqref{pri:6.4}$.

\textbf{Step 5.} Show the estimate $\eqref{pri:23}$ to complete the whole proof.
From the equations $\eqref{pde:14}$ and $\eqref{pde:7}$, we find that $z_\varepsilon$ satisfies
\begin{equation}\label{pde:8}
\left\{\begin{aligned}
-\nabla\cdot a^{\varepsilon}\nabla
z_{\varepsilon}
&=\nabla\cdot h_{\varepsilon}
+ g_{\varepsilon}
\quad &\text{in}&\quad \Omega;\\
z_{\varepsilon}
& =0
\quad &\text{on}&\quad \partial\Omega,
\end{aligned}\right.
\end{equation}
where $h_{\varepsilon}
:= \big(a^{\varepsilon}\tilde{\Phi}_{\varepsilon,i}
-\varepsilon\sigma_i^{\varepsilon}\big)
\nabla\partial_i u_{0}$ and
$g_{\varepsilon}
:= a^{\varepsilon}\nabla Q_{\varepsilon,i}
\cdot\nabla\partial_i u_0$.
Then, we need to decompose the equation $\eqref{pde:8}$ into the following three equations:
\begin{subequations}
\begin{align}
& (1).~\left\{\begin{aligned}
-\nabla\cdot a^{\varepsilon}\nabla
z_{\varepsilon}^{(1)}
&=\nabla\cdot h_{\varepsilon}\mathbf{1}_{U_r(x_0)}
+ g_{\varepsilon}\mathbf{1}_{U_r(x_0)}
~ &\text{in}&\quad \Omega;\\
z_{\varepsilon}^{(1)}
& =0
~ &\text{on}&\quad \partial\Omega.
\end{aligned}\right.\label{pde:15.1}\\
& (2).~\left\{\begin{aligned}
-\nabla\cdot a^{\varepsilon}\nabla
z_{\varepsilon}^{(2)}
&=\nabla\cdot h_{\varepsilon}\mathbf{1}_{\Omega\setminus U_r(x_0)}
~ &\text{in}&\quad \Omega;\\
z_{\varepsilon}^{(2)}
& =0
~ &\text{on}&\quad \partial\Omega.
\end{aligned}\right.
\label{pde:15.2} \\
&(3).~\left\{\begin{aligned}
-\nabla\cdot a^{\varepsilon}\nabla
z_{\varepsilon}^{(3)}
&=g_{\varepsilon}\mathbf{1}_{\Omega\setminus U_r(x_0)}
~ &\text{in}&\quad \Omega;\\
z_{\varepsilon}^{(3)}
& =0
~ &\text{on}&\quad \partial\Omega.
\end{aligned}\right.
\label{pde:15.3}
\end{align}
\end{subequations}

We start from dealing with the solution $z_\varepsilon^{(1)}$ of
$\eqref{pde:15.1}$, and the main idea is similar to
that given in \textbf{Step 2}. With the help of the representation of
the solution by Green functions, we first have
\begin{equation}\label{eq:6.7}
\begin{aligned}
& \partial_m z_\varepsilon^{(1)}(x_0)
 =
\int_{U_\varepsilon(x_0)}\nabla \partial_{m}G_\varepsilon(x_0,z)\cdot\big(h_{\varepsilon}(z)-h_{\varepsilon}(x_0)\big)dz\\
&+  h_{\varepsilon}(x_0)\cdot\int_{\partial U_\varepsilon(x_0)}n(z)\partial_m G_\varepsilon(x_0,z)dS(z)
 +\int_{U_r(x_0)}\partial_m G_\varepsilon(x_0,z)g_{\varepsilon}(z)dz\\
&+\int_{U_r\setminus U_\varepsilon(x_0)}\nabla\partial_m G_\varepsilon(x_0,z)\cdot h_{\varepsilon}(z)dz=:J_1+J_2+J_3+J_4.
\end{aligned}
\end{equation}

By Proposition $\ref{P:5}$, for any $\beta\in[1,\infty)$, a routine computation leads to
\begin{equation}\label{pri:6.5}
\begin{aligned}
&\esssup_{U_r(x_0)}
\big\langle|h_{\varepsilon}|^{\beta}\big\rangle^{\frac{1}{\beta}}
+\varepsilon^\eta
\big\langle[h_\varepsilon]_{C^{0,\eta}(U_\varepsilon(x_0))}^{\beta}
\big\rangle^{\frac{1}{\beta}}
\lesssim^{\eqref{pri:21-1},\eqref{pri:21-2}} \varepsilon\mu_d(R_0/\varepsilon)
\dashint_{D_r(\bar{y}_0)}|g|;\\
& \big\langle|g_\varepsilon(z)|^{\beta}\big\rangle^{\frac{1}{\beta}}
\lesssim^{\eqref{pri:18}} \mu_d(R_0/\varepsilon)\min\Big\{1,\frac{\varepsilon}{\delta(z)}\Big\}
\dashint_{D_r(\bar{y}_0)}|g|
\qquad\forall z\in U_r(x_0).
\end{aligned}
\end{equation}
This together with Proposition $\ref{P:4}$ yields
\begin{subequations}
\begin{align}
& \langle|J_1|^{p_2}\rangle^{\frac{1}{p_2}}
+ \langle|J_2|^{p_2}\rangle^{\frac{1}{p_2}}
\lesssim \varepsilon\mu_d(R_0/\varepsilon)\dashint_{D_r(\bar{y}_0)}|g|;\label{f:46}\\
& \langle|J_3|^{p_2}\rangle^{\frac{1}{p_2}}
+\langle|J_4|^{p_2}\rangle^{\frac{1}{p_2}}
\lesssim \varepsilon \mu_d(R_0/\varepsilon)\ln(r/\varepsilon)\dashint_{D_r(\bar{y}_0)}|g|.\label{f:44}
\end{align}
\end{subequations}
As a result, collecting the above estimates we obtain
\begin{equation}\label{f:6.5}
\begin{aligned}
\langle|\nabla z_\varepsilon^{(1)}(x_0)|^p\rangle^{\frac{1}{p}}
\lesssim^{\eqref{eq:6.7},\eqref{f:46},\eqref{f:44}}
\varepsilon\mu_d(R_0/\varepsilon)\ln(r/\varepsilon)\sup_{D_r(\bar{y}_0)}|g|.
\end{aligned}
\end{equation}

Then we proceed to address the solution $z_\varepsilon^{(2)}$ of
$\eqref{pde:15.2}$. As a preparation, in view of the equation (2) in $\eqref{pde:14}$, we have
\begin{equation}\label{pri:6.6}
 \Big(\int_{\Omega}|\nabla^2 u_0|^2\Big)^{1/2}
 \lesssim_{\lambda,d,m_0,R_0M_0} \Big(\int_{D_r(\bar{y}_0)} |g|^2\Big)^{1/2}
 \lesssim r^{\frac{d}{2}}\sup_{D_r(\bar{y}_0)}|g|,
\end{equation}
where we require $\partial\Omega\in C^{1,1}$ at least for the above estimate.
Moreover, we claim that there holds
\begin{equation}\label{f:6.11}
\langle|\nabla z_\varepsilon^{(2)}(x_0)|^p\rangle^{\frac{1}{p}}
\lesssim
\frac{1}{r^{\frac{d}{2}}}
\Big(\int_{\Omega}
\big\langle|\nabla z_\varepsilon^{(2)}|^{p_1}
\big\rangle^{\frac{2}{p_1}}\Big)^{\frac{1}{2}}.
\end{equation}
Admitting for a while, this together with Lemma $\ref{P:7*}$ (taking $\omega=1$ therein) leads to
\begin{equation}\label{f:6.6}
\begin{aligned}
\langle|\nabla z_\varepsilon^{(2)}(x_0)|^p\rangle^{\frac{1}{p}}
&\lesssim^{\eqref{f:6.11},\eqref{pri:12}}
\lesssim
\frac{1}{r^{\frac{d}{2}}}
\Big(\int_{\Omega}
\big\langle|h_\varepsilon\mathbf{1}_{\Omega\setminus
U_r(x_0)}|^{p_2}
\big\rangle^{\frac{2}{p_2}}\Big)^{\frac{1}{2}}\\
&\lesssim^{\eqref{pri:21-1},\eqref{pri:6.6}} \varepsilon\mu_d(R_0/\varepsilon)
\sup_{D_r(\bar{y}_0)}
|g|.
\end{aligned}
\end{equation}

Concerning the estimate $\eqref{f:6.11}$, it involves two cases:
(1) $\delta(x)\leq (r/2)$; (2) $\delta(x)>(r/2)$. For the case (1), it follows
from the argument analogous to $\eqref{f:6.10}$ that
\begin{equation*}
\begin{aligned}
\langle|\nabla z_\varepsilon^{(2)}(x_0)|^p\rangle^{\frac{1}{p}}
&\lesssim \Big\langle\big(\dashint_{D_r(\bar{x}_0)}
|\nabla z_\varepsilon^{(2)}|^2\big)^{\frac{p_2}{2}}
\Big\rangle^{\frac{1}{p_2}}\\
&\lesssim
\Big(\dashint_{D_r(\bar{x}_0)}
\langle|\nabla z_\varepsilon^{(2)}|^{p_2}\rangle^{\frac{2}{p_2}}
\Big)^{\frac{1}{2}}
\lesssim  r^{-\frac{d}{2}}\Big(\int_{\Omega}
\langle|\nabla z_\varepsilon^{(2)}|^{p_2}\rangle^{\frac{2}{p_2}}
\Big)^{\frac{1}{2}},
\end{aligned}
\end{equation*}
where $\bar{x}_0\in\partial\Omega$ is
such that $\delta(\bar{x}_0)=|x_0-\bar{x}_0|$, and we employ Minkowski's inequality in the second step. For the case (2), the interior estimate leads to
\begin{equation*}
\begin{aligned}
\langle|\nabla z_\varepsilon^{(2)}(x_0)|^p\rangle^{\frac{1}{p}}
\lesssim^{\eqref{pri:2.10}} \Big\langle\big(\dashint_{B_{\frac{r}{4}}(x_0)}
|\nabla z_\varepsilon^{(2)}|^2\big)^{\frac{p_2}{2}}
\Big\rangle^{\frac{1}{p_2}}
\lesssim
r^{-\frac{d}{2}}\Big(\int_{\Omega}
\langle|\nabla z_\varepsilon^{(2)}|^{p_2}\rangle^{\frac{2}{p_2}}
\Big)^{\frac{1}{2}}.
\end{aligned}
\end{equation*}

We now turn to the last equation $\eqref{pde:15.3}$.
By using the representation of the solution via Green functions
and Propositions $\ref{P:4},\ref{P:5}$, we have
\begin{equation}\label{f:6.7}
\begin{aligned}
&\langle|\nabla z_\varepsilon^{(3)}(x_0)|^p\rangle^{\frac{1}{p}}
\lesssim^{\eqref{pri:18}} \varepsilon\mu_d(R_0/\varepsilon)
\int_{\Omega\setminus U_r(x_0)}
\langle|\nabla G_\varepsilon(x_0,y)|^{\bar{p}}\rangle^{\frac{1}{\bar{p}}}
|\nabla^2 u_0(y)|\frac{dy}{\delta(y)}\\
&\lesssim^{\eqref{pri:19}}
\varepsilon\mu_d(R_0/\varepsilon)
\Big(\int_{\Omega}|\nabla^2 u_0|^2\Big)^{\frac{1}{2}}
\Big(\int_{\Omega\setminus U_r(x_0)}\frac{dy}{|y-x_0|^{2d}}\Big)^{\frac{1}{2}}\\
&\lesssim^{\eqref{pri:6.6}} \varepsilon\mu_d(R_0/\varepsilon)
\sup_{D_r(\bar{y}_0)}|g|.
\end{aligned}
\end{equation}

Consequently, by noting $z_\varepsilon
=z_\varepsilon^{(1)}+z_\varepsilon^{(2)}+z_\varepsilon^{(3)}$ we have
\begin{equation*}
\begin{aligned}
\langle|\nabla z_\varepsilon(x_0)|^p\rangle^{\frac{1}{p}}
&\leq
\langle|\nabla z_\varepsilon^{(1)}(x_0)|^p\rangle^{\frac{1}{p}}
+\langle|\nabla z_\varepsilon^{(2)}(x_0)|^p\rangle^{\frac{1}{p}}
+\langle|\nabla z_\varepsilon^{(3)}(x_0)|^p\rangle^{\frac{1}{p}}\\
& \lesssim^{\eqref{f:6.5},\eqref{f:6.6},\eqref{f:6.7}}
\varepsilon\mu_d(R_0/\varepsilon)
 \ln(r/\varepsilon)
\sup_{D_r(\bar{y}_0)}|g|,
\end{aligned}
\end{equation*}
which together with Poincar\'e's inequality leads to the desired
estimate $\eqref{pri:23}$. This ends the whole proof.
\qed

\medskip

\section{A bootstrap argument}\label{sec:6}

\noindent
The proof of Theorem $\ref{thm:2*}$ relies on the following results, and
the main ideas are inspired by the work of Bella et al. \cite{Bella-Fischer-Josien-Raithel-24}.
The present contribution mainly lies in discussing the more general case of bounded domains and including the situation where the dimension $d = 2$.
Recall that
the notation
``$\lfloor x\rfloor$'' denotes the largest integer not more than $x$, and ``$\lceil x\rceil$'' represents the smallest integer not less than $x$.
\begin{proposition}[CLT-scaling type estimates]\label{P:8}
Let $\Omega\subset\mathbb{R}^d$ with $d\geq 2$ be a bounded $C^{\lfloor d/2\rfloor+1}$ domain.
Suppose that the ensemble $\langle\cdot\rangle$ satisfies
$\eqref{a:2}$ and $\eqref{a:3}$.
Let $1\leq p<\infty$.
Assume that $Q$ is the related boundary corrector\footnote{Recalling
the equations $\eqref{pde:6}$, in view of rescaling, one indeed has the relationship:  $Q:=\frac{1}{\varepsilon}Q_\varepsilon(\varepsilon\cdot)$.} associated with $\phi$ by the equations
$-\nabla\cdot a\nabla Q = 0$ in $\tilde{\Omega}$ and $Q=-\phi$
on $\partial\tilde{\Omega}$, where $\tilde{\Omega}:=\Omega/\varepsilon$
and $\varepsilon\in(0,1]$.  Then, there exists
$\kappa_d$, satisfying
$\kappa_2=0$ if $d=2$ and $\kappa_d=\kappa$ if $d\geq 3$ with
$0<\kappa\ll 1$, such that the following estimate
\begin{equation}\label{pri:7.4}
 \big\langle|\nabla Q(z)|^p\big\rangle^{\frac{1}{p}}
 \lesssim_{\lambda,\lambda_1,\lambda_2,d,p,m_0,
 R_0^{\lfloor d/2\rfloor+\tau}M_0,\kappa} \mu_d^{\lceil\frac{d}{2}\rceil}(R_0/\varepsilon)
 \big[\delta(z)
 \big]^{-\frac{d}{2}+\kappa_d}
\end{equation}
holds for any $z\in\tilde{\Omega}$ and $\delta(z):=\text{dist}(z,\partial\tilde{\Omega})\geq 2$.
\end{proposition}

We introduce a smoothing operator $S_{r}$
with $r>0$, defined as $S_{r}(u):=\int_{\mathbb{R}^d} \zeta_r(y)u(x-y)dy$,
where $\zeta_r(x)=r^{-d}
\zeta(x/r)$ and $\zeta\in C_0^{\infty}(B_{\frac{1}{2}})$ satisfies
$\int_{\mathbb{R}^d}\zeta = 1$ and $\zeta\geq 0$.
Without a proof (see e.g. \cite[pp.37]{Shen18}), for any $u\in H^1_{\text{loc}}(\mathbb{R}^d)$, there holds the following inequality
\begin{equation}\label{f:7.1}
\int_{B_{r}}|u-S_{r}(u)|^2
\lesssim r^2\int_{B_{2r}}|\nabla u|^2.
\end{equation}

\begin{lemma}[improved Caccioppoli's inequality]\label{lemma:7.1}
Let $0<\rho\ll 1$, and $R\geq 1$ (could be very large).
Suppose that $\langle\cdot\rangle$ is an ensemble of $\lambda$-uniformly elliptic coefficient fields, and
$u$ satisfies the equation $\nabla\cdot a\nabla u = 0$ in $B_{2R}$.
Then, there exist two integers $k_0:=\lfloor \rho\log_2 R \rfloor$
and $c_\rho := \lceil2^{\frac{d^2}{2\rho}}\rceil$, and a family of center points $\{x_m^j\}\subset B_R$, such that
\begin{equation}\label{pri:7.1}
 \dashint_{B_{R^{1-\rho}}}|\nabla u|^2
 \lesssim \frac{1}{R^d} \dashint_{B_R}|\nabla u|^2
 + c_\rho^{-1}\sum_{j=1}^{k_0}2^{-\frac{j}{\rho}}
 \sum_{m=1}^{c_\rho}\big|S_{2^{-k_0-\frac{d}{2\rho}+j}R}(\nabla u)(\hat{x}_m^j)\big|^2,
\end{equation}
where the multiplicative constant depends only on $\lambda$ and $d$.
\end{lemma}

\begin{proof}
The main idea may be found in \cite{Bella-Fischer-Josien-Raithel-24,Armstrong-Kuusi-Mourrat-16}, and
we provide a proof for the reader's convenience. The proof is divided into two steps.

\textbf{Step 1.}
Let $0<r\leq R$, and $\theta\in (0,1/2)$ will be fixed later on. Based upon Caccioppoli's inequality and the smoothing operator defined above, we show that there exist $c_\theta\in\mathbb{N}$ and a family of points $\{\hat{x}_m\}_{m=1}^{c_\theta}\subset B_{2r}$ such that
\begin{equation}\label{f:7.3}
\dashint_{B_r}|\nabla u|^2
\lesssim
\theta^2\dashint_{B_{2r}}|\nabla u|^2
+ c_\theta^{-1}
\sum_{m=1}^{c_{\theta}}|S_{\theta r}(\nabla u)(\hat{x}_m)|^2.
\end{equation}

We start from the following Caccioppoli's inequality
\begin{equation}\label{f:7.2}
\begin{aligned}
\dashint_{B_r}|\nabla u|^2
&\lesssim \frac{1}{r^2}\dashint_{B_{\frac{3}{2}r}}|u-c|^2
\lesssim \frac{1}{r^2}\dashint_{B_{\frac{3}{2}r}}|u-S_{\theta r}(u)|^2
+ \frac{1}{r^2}\dashint_{B_{\frac{3}{2}r}}|S_{\theta r}(u)- c|^2\\
&\lesssim^{\eqref{f:7.1}} \theta^2\dashint_{B_{2r}}|\nabla u|^2
+ \dashint_{B_{2r}}| S_{\theta r}(\nabla u)|^2,
\end{aligned}
\end{equation}
where we take $c=\dashint_{B_{2r}}S_{\theta r}(u)$ and use Poincare's inequality for the last inequality above.
To handle the last term in $\eqref{f:7.2}$, we introduce a good covering of
$B_{2r}$ by a family of balls with smaller radius, i.e., there exist $\{x_m\}\subset B_{2r}$ and
$B_{2r}\subset \cup_{m=1}^{c_\theta} B_{\theta r}(x_m)$
where $c_{\theta}:=\lceil\theta^{-d}\rceil$.
By the mean value theorem of integrals, for each $m=1\,\cdots,c_\theta$
there exists $\hat{x}_m\in B_{\theta r}(x_m)$, such that
\begin{equation*}
\int_{B_{2r}}|S_{\theta r}(\nabla u)|^2
\lesssim_d (\theta r)^d
\sum_{m=1}^{c_{\theta}}\dashint_{B_{\theta r}(x_m)}|S_{\theta r}(\nabla u)|^2
\lesssim_d (\theta r)^d
\sum_{m=1}^{c_{\theta}}|S_{\theta r}(\nabla u)(\hat{x}_m)|^2,
\end{equation*}
Plugging this back into $\eqref{f:7.2}$, one can derive the stated estimate $\eqref{f:7.3}$.

\textbf{Step 2.} Show the desired estimate $\eqref{pri:7.1}$ by iterating the estimate $\eqref{f:7.3}$. Fix $\rho\in (0,1)$ small, and let
$k_0\in\mathbb{N}$ be such that $2^{-k_0}R\geq R^{1-\rho}$, which finally\footnote{It is convenient to keep
$k_0 = \rho\log_2 R$ in the later computation of this lemma.} implies that one can choose
$k_0 = \lfloor \rho\log_2 R \rfloor$. Therefore, we start from taking
$r= 2^{-k_0}R$ in $\eqref{f:7.3}$ to implement the iterations, i.e.,
\begin{equation}\label{f:7.4}
\begin{aligned}
\dashint_{B_{R^{1-\rho}}}|\nabla u|^2
&\lesssim_d \dashint_{B_{2^{-k_0}R}}|\nabla u|^2
\lesssim
\theta^2\dashint_{B_{2^{-k_0+1}R}}|\nabla u|^2
+ c_\theta^{-1}
\sum_{m=1}^{c_{\theta}}|S_{\theta 2^{-k_0}R}(\nabla u)(\hat{x}_m^1)|^2\\
&\lesssim \cdots \cdots \\
&\lesssim \theta^{2k_0}
\dashint_{B_{R}}|\nabla u|^2
+ c_{\theta}^{-1}\sum_{j=1}^{k_0}\theta^{2(j-1)}
\sum_{m=1}^{c_{\theta}}|S_{\theta 2^{-k_0+j}R}(\nabla u)(\hat{x}_m^j)|^2,
\end{aligned}
\end{equation}
where $\{\hat{x}_m^j\}$ represent those points $\{\hat{x}_m\}$ selected in the ball $B_{2^{-k_0+j}R}$ during the $j$-th iteration, following the procedure described in \textbf{Step 1}. Since these points change in each iteration, we thus introduce the superscripts of $\{\hat{x}_m^j\}$
to mark this change.

Now, one can choose $\theta\in(0,1/2)$ such that $\theta^{2k_0}=R^{-d}$, which implies that $\theta = 2^{-\frac{d}{2\rho}}$. Then,
it follows from the estimate $\eqref{f:7.4}$ that
\begin{equation}
\begin{aligned}
\dashint_{B_{R^{1-\rho}}}|\nabla u|^2
\lesssim_{d,\mu,\rho} \frac{1}{R^d}
\dashint_{B_{R}}|\nabla u|^2
+ c_{\theta}^{-1}\sum_{j=1}^{k_0}2^{-\frac{j}{\rho}}
\sum_{m=1}^{c_{\theta}}|S_{2^{-k_0-\frac{d}{2\rho}+j}R}(\nabla u)(\hat{x}_m^j)|^2.
\end{aligned}
\end{equation}
By setting $c_\rho := \lceil2^{\frac{d^2}{2\rho}}\rceil$ we therefore have
$c_\rho = c_\theta$, and this have completed the whole proof.
\end{proof}

To state the following result, we introduce the following notation:
$\tilde{\Omega}:=\Omega/\varepsilon$ with $\varepsilon\in(0,1]$,
and therefore we correspondingly have $\delta(x):=\text{dist}(x,\partial\tilde{\Omega})$ without a confusion.
\begin{lemma}[one-order improvement]\label{lemma:7.2}
Let $\Omega\subset\mathbb{R}^d$ with $d\geq 2$ be a bounded $C^{\lfloor d/2\rfloor+1}$ domain.
Suppose that the ensemble $\langle\cdot\rangle$ satisfies
$\eqref{a:2}$ and $\eqref{a:3}$.
Let $\beta>0$.
Assume that $Q$ is the related boundary corrector associated with $\phi$ by the equations
$-\nabla\cdot a\nabla Q = 0$ in $\tilde{\Omega}$ and $Q=-\phi$
on $\partial\tilde{\Omega}$ with the following estimate:
\begin{equation}\label{pri:7.2condition}
 \big\langle|\nabla Q(z)|^p\big\rangle^{\frac{1}{p}}
 \lesssim_{c_0} \mu_d^{\lfloor\beta\rfloor}(R_0/\varepsilon)
 \big[\delta(z)\big]^{-\beta}
 \qquad \forall z\in\tilde{\Omega}\quad\text{with}\quad \delta(z)\geq 2.
\end{equation}
Then, for any deterministic vector field $g$ with the components $g_i:=\zeta_r(\cdot-\hat{x})e_i$ with
$\delta(\hat{x})\geq 2$ and $r = \delta(\hat{x})/2$, there holds
\begin{equation}\label{pri:7.2outcome}
\Big\langle\int_{\tilde{\Omega}}\nabla Q\cdot g \Big\rangle
\lesssim_{} \mu_d^{\lfloor\beta\rfloor+1}(R_0/\varepsilon)
\Big\{r^{-\beta-1}\ln r +r^{-\frac{d}{2}}\Big\},
\end{equation}
where the multiplicative constant depends on $\lambda,\lambda_1,\lambda_2,d,c_0,\tau,m_0,R_0^{\lfloor d/2\rfloor+\tau}M_0$, and $p$.
\end{lemma}

\begin{proof}
The main idea is inspired by
Bella et al. \cite[Proposition 6]{Bella-Fischer-Josien-Raithel-24}, but here we avoid using the stationarity of the boundary corrector in the tangential direction alone the boundary (which usually does not hold for general domains). The whole proof is divided into five steps.

\textbf{Step 1.} Outline the proof and reduction.
By a duality argument, for any deterministic vector field $g$, one can first establish the following identity:
\begin{equation}\label{f:7.5}
\Big\langle\int_{\tilde{\Omega}}\nabla Q\cdot g\Big\rangle
= \int_{\tilde{\Omega}}\Big\langle
(a^{*}\phi_j^{*}-\sigma_j^{*})\nabla\bar{\varphi}_j\cdot\nabla Q\Big\rangle
+\int_{\tilde{\Omega}}
\Big\langle(a^*-\bar{a}^{*})(\nabla\bar{v}-\bar{\varphi}_j)\cdot\nabla Q
\Big\rangle,
\end{equation}
where we consider the adjoint problems:
\begin{equation}\label{pde:7.2}
\left\{\begin{aligned}
-\nabla\cdot a^*\nabla v &= \nabla\cdot g &\quad&\text{in}\quad\tilde{\Omega};\\
v&=0 &\quad&\text{on}\quad\partial\tilde{\Omega},
\end{aligned}\right.
\qquad
\left\{\begin{aligned}
-\nabla\cdot \bar{a}^*\nabla \bar{v} &= \nabla\cdot g &\quad&\text{in}\quad\tilde{\Omega};\\
\bar{v}&=0 &\quad&\text{on}\quad\partial\tilde{\Omega},
\end{aligned}\right.
\end{equation}
and choosing $\bar{\varphi}_j\in H_0^1(\tilde{\Omega})$, which is determined by $\partial_j\bar{v}$, will be explicitly shown later on.

Then, in view of the precondition $\eqref{pri:7.2condition}$ and Lemma $\ref{lemma:*3}$, one can derive that
\begin{equation}\label{f:7.12}
\Big\langle\int_{\tilde{\Omega}}\nabla Q\cdot g \Big\rangle
\lesssim
\mu_d^{\lfloor\beta\rfloor+1}(R_0/\varepsilon)
\Big\{\int_{O_2}\big(|\nabla\bar{v}|+|\nabla^2\bar{v}|\big)
+\int_{\tilde{\Omega}\setminus O_2}|\nabla^2\bar{v}|\delta^{-\beta}\Big\},
\end{equation}
where $O_r:=\{x\in\tilde{\Omega}:\text{dist}(x,\partial\tilde{\Omega})\leq r\}$ represents the layer part of $\tilde{\Omega}$.
By means of the Green's function, one can establish the following estimates:
\begin{subequations}
\begin{align}
&\int_{O_2}\big(|\nabla \bar{v}| + |\nabla^2\bar{v}|\big)
\lesssim r^{-\frac{d}{2}};  \label{f:7.24}\\
&\int_{\tilde{\Omega}\setminus O_2}|\nabla^2\bar{v}|\delta^{-\beta}
\lesssim r^{-\beta-1}\ln r
+ r^{-\frac{d}{2}}.
\label{f:7.25}
\end{align}
\end{subequations}

Consequently, plugging $\eqref{f:7.24}$ and $\eqref{f:7.25}$
into $\eqref{f:7.12}$, we can derive the desired estimate
$\eqref{pri:7.2outcome}$.

%\begin{equation}\label{f:7.8}
%\Big\langle\int_{\tilde{\Omega}}\nabla Q\cdot g \Big\rangle
%\lesssim
%\mu_d^2(R_0/\varepsilon)
%\Big\{\int_{O_2}\big(|\nabla\bar{v}|+|\nabla^2\bar{v}|\big)
%+\int_{\tilde{\Omega}\setminus O_2}|\nabla^2\bar{v}|\delta^{-\beta}\Big\}
%\end{equation}

\textbf{Step 2.} Arguments for $\eqref{f:7.5}$.
The basic idea is appealing to the duality by using the given data $g$.
In view of $\eqref{pde:7.2}$,
consider the error of the two-scale expansion  $z:=v-\bar{v}-\phi_j^*\varphi_j$, where one can choose $\varphi_j :=\eta\partial_j\bar{v} $ to make $z\in H_0^1(\tilde{\Omega})$,
where $\eta\in C_0^1(\tilde{\Omega})$ is a cut-off function satisfying
\begin{equation*}
\eta = 0 ~\text{on}~ O_{1}:=\{x\in\tilde{\Omega}:\text{dist}(x,\partial\tilde{\Omega})\leq 1\};
\quad
\eta =1 ~\text{in}~\tilde{\Omega}\setminus O_2;
\quad
|\nabla \eta|\lesssim 1.
\end{equation*}
This error satisfies the following equation:
\begin{equation}\label{pde:7.1}
-\nabla\cdot a^*\nabla z
= \nabla\cdot\big[(a^*\phi_j^{*}-\sigma_j^{*})\nabla\bar{\varphi}_j
+(a^*-\bar{a}^*)(\nabla\bar{v}-\varphi)\big]
\quad\text{in}\quad\tilde{\Omega}.
\end{equation}
Applying $Q$ to the both sides of $\eqref{pde:7.1}$, we introduce two quantities for further computations:
\begin{equation*}
\begin{aligned}
&\mathfrak{L}:= -\int_{\tilde{\Omega}}\nabla\cdot a^*\nabla z Q;\\
&\mathfrak{R}:=\int_{\tilde{\Omega}}
\nabla\cdot\big[(a^*\phi_j^{*}-\sigma_j^{*})\nabla\bar{\varphi}_j
+(a^*-\bar{a}^*)(\nabla\bar{v}-\varphi)\big]Q.
\end{aligned}
\end{equation*}

The following computation is routine (it is merely integration by parts),
and we provide a proof for the reader's convenience.
We start from the right-hand side,
\begin{equation}\label{f:7.7}
\begin{aligned}
\mathfrak{R} &=
-\int_{\tilde{\Omega}}
\big[(a^*\phi_j^{*}-\sigma_j^{*})\nabla\bar{\varphi}_j
+(a^*-\bar{a}^*)(\nabla\bar{v}-\varphi)\big]\cdot\nabla Q
- \int_{\partial\tilde{\Omega}}n\cdot\big(a^*-\bar{a}^*\big)
\nabla\bar{v}\phi dS,
\end{aligned}
\end{equation}
where we note that $\nabla\bar{\varphi}_j =0$ and $\bar{\varphi}_j=0$ on
$\partial\tilde{\Omega}$. Then, the left-hand side is as follows:
\begin{equation}\label{f:7.6}
\begin{aligned}
\mathfrak{L} &= \int_{\tilde{\Omega}}\nabla\cdot g Q
+ \int_{\tilde{\Omega}}\nabla\cdot a^{*}\nabla
\big(\bar{v}+\phi_j^{*}\bar{\varphi}_j\big)Q \\
&=-\int_{\tilde{\Omega}}\nabla Q\cdot g
- \int_{\tilde{\Omega}}\nabla\big(\bar{v}+\phi_j^*\bar{\varphi}_j\big)
\cdot a\nabla Q -
\int_{\partial\tilde{\Omega}} n\cdot a^*\nabla\bar{v}\phi dS,
\end{aligned}
\end{equation}
where we employ the fact that $\text{supp}(g)\cap\partial\tilde{\Omega}=\emptyset$ and $\bar{\varphi}_j=0$ on $\partial\tilde{\Omega}$. Moreover,
noting that $\nabla\cdot a\nabla Q = 0$ in $\tilde{\Omega}$, and
$\bar{v}$ with $\bar{\varphi}_j$ vanishes on $\partial\tilde{\Omega}$,
we continue to apply integration by parts to the second term above,
\begin{equation*}
\int_{\tilde{\Omega}}\nabla\big(\bar{v}+\phi_j^*\bar{\varphi}_j\big)
\cdot a\nabla Q = 0.
\end{equation*}
Plugging this back into $\eqref{f:7.6}$, we have
\begin{equation*}
\mathfrak{L}
=-\int_{\tilde{\Omega}}\nabla Q\cdot g
 -\int_{\partial\tilde{\Omega}} n\cdot a^*\nabla\bar{v}\phi dS.
\end{equation*}

By using $\mathfrak{L} = \mathfrak{R}$ and combining the equalities
$\eqref{f:7.6}$ and $\eqref{f:7.7}$, we obtain that
\begin{equation*}
\begin{aligned}
&\int_{\tilde{\Omega}}\nabla Q\cdot g
=  \int_{\tilde{\Omega}}
\big[(a^*\phi_j^{*}-\sigma_j^{*})\nabla\bar{\varphi}_j
+(a^*-\bar{a}^*)(\nabla\bar{v}-\varphi)\big]\cdot\nabla Q
+ \int_{\partial\tilde{\Omega}}n\cdot \bar{a}^* \nabla\bar{v}\phi dS\\
&= \int_{\tilde{\Omega}}
\big[(a^*\phi_j^{*}-\sigma_j^{*})\nabla\bar{\varphi}_j
+(a^*-\bar{a}^*)(\nabla\bar{v}-\varphi)\big]\cdot\nabla Q
+ \int_{\tilde{\Omega}}\big(g+\bar{a}^*\nabla\bar{v}\big)\cdot\nabla\phi.
\end{aligned}
\end{equation*}
Therefore,
taking $\langle\cdot\rangle$ on the both sides above leads to
the stated equality $\eqref{f:7.5}$, where we utilize the stationarity of
$\nabla\phi$ to note that
\begin{equation*}
\Big\langle\int_{\tilde{\Omega}}\big(g+\bar{a}^*\nabla\bar{v}\big)\cdot
  \nabla\phi \Big\rangle
=\int_{\tilde{\Omega}}\big(g+\bar{a}^*\nabla\bar{v}\big)\cdot
 \big\langle\nabla\phi \big\rangle = 0.
\end{equation*}

\textbf{Step 3.} Show the estimate $\eqref{f:7.12}$. Let $1<p<\infty$. We start from
the first term in the right-hand side of $\eqref{f:7.5}$, and
appealing to the precondition $\eqref{pri:7.2condition}$, Lemma $\ref{lemma:*3}$, and Theorem $\ref{thm:2}$ in order we have
\begin{equation}\label{f:7.9}
\begin{aligned}
\int_{\tilde{\Omega}}\Big\langle
(a^{*}\phi_j^{*}
&-\sigma_j^{*})\nabla\bar{\varphi}_j
\cdot\nabla Q\Big\rangle
\lesssim^{\eqref{pri:**3}} \mu_d(R_0/\varepsilon)
\bigg\{\int_{O_2\setminus O_1}|\nabla\bar{v}|\big\langle
|\nabla Q|^p\big\rangle^{1/p} + \int_{\tilde{\Omega}\setminus O_1}
|\nabla^2\bar{v}|\big\langle
|\nabla Q|^p\big\rangle^{1/p}\bigg\}\\
&\lesssim^{\eqref{pri:1.b},\eqref{pri:7.2condition}} \mu_d^{\lfloor\beta\rfloor+1}(R_0/\varepsilon)
\bigg\{\int_{O_2}\big(|\nabla \bar{v}| + |\nabla^2\bar{v}|\big)
+ \int_{\tilde{\Omega}\setminus O_2}|\nabla^2\bar{v}|\delta^{-\beta}
\bigg\}.
\end{aligned}
\end{equation}
By the same token, we have the estimates on the second term in the right-hand side of $\eqref{f:7.5}$,
\begin{equation}\label{f:7.10}
\begin{aligned}
\int_{\tilde{\Omega}}
\Big\langle(a^*-\bar{a}^{*})(\nabla\bar{v}-\bar{\varphi}_j)\cdot\nabla Q
\Big\rangle
\lesssim \mu_d^{\lfloor\beta\rfloor}(R_0/\varepsilon)
\int_{O_2}|\nabla \bar{v}|.
\end{aligned}
\end{equation}
Combining the estimates $\eqref{f:7.9}$ and $\eqref{f:7.10}$ leads to
the stated estimate $\eqref{f:7.12}$.

\textbf{Step 4.} Show the estimate $\eqref{f:7.24}$.
Let $\alpha\in\mathbb{N}^d$ be a multi-index satisfying
$|\alpha|:= \lceil d/2\rceil$. Let $\{O_2^i\}$ be a decomposition of
$O_2$, i.e., $O_2 = \cup_i O_2^i$, where $|O_2^i|\sim 2^d$.
We split the proof into two cases: (1) $d\geq 3$; (2) d=2.
We start from the case $d\geq 3$. It follows from the interpolation inequality (see e.g. \cite[Theorem 5.2]{Adams-03}) that
\begin{equation}\label{f:7.15}
\begin{aligned}
\int_{O_2}\big(|\nabla \bar{v}| + |\nabla^2\bar{v}|\big)
\lesssim  \sum_{i}\int_{O_2^i}\big(|\nabla^\alpha\bar{v}|+ |\bar{v}|\big)
\lesssim \int_{O_2}|\nabla^\alpha\bar{v}| +
 \sum_{i}\int_{O_2^i}|\bar{v}|.
\end{aligned}
\end{equation}
Let $q:=d/(d-|\alpha|)$ and $1/q+1/{q'}=1$.
For any $s\in[1,\infty)$, there exists the extension operator
$T^i$ such that
$T^i(\bar{v})=\bar{v}$ in $O_2^i$ and
$\|\nabla^{\alpha}T^i(\bar{v})\|_{L^s(\mathbb{R}^d)}
\leq C_s^i\|\nabla^{\alpha}\bar{v}\|_{L^s(O_2^i)}$
(see \cite[Theorem 1.2]{Chua-92}).
By H\"older's inequality and Sobolev's inequality (see e.g. \cite[Theorem 4.31]{Adams-03}),
one can derive that
\begin{equation}\label{f:7.16}
\begin{aligned}
\sum_{i}\int_{O_2^i}|\bar{v}|
\leq \sum_{i}|O_2^i|^{\frac{1}{q'}}\|\bar{v}\|_{L^q(O_2^i)}
&\lesssim_{d,q'} \sum_{i}\|\bar{v}\|_{L^q(O_2^i)}
\lesssim \sum_{i}\|T^i(\bar{v})\|_{L^q(\mathbb{R}^d)}
\lesssim \sum_{i}\|\nabla^{\alpha} T^i(\bar{v})\|_{L^1(\mathbb{R}^d)}\\
&\lesssim \sum_{i}C_1^i\|\nabla^{\alpha} \bar{v}\|_{L^1(O_2^i)}
\lesssim \max_{i}C_1^{i} \int_{O_2}|\nabla^{\alpha}\bar{v}|
\lesssim \int_{O_2}|\nabla^{\alpha}\bar{v}|,
\end{aligned}
\end{equation}
where $C_1^i$ is uniformly bounded since the geometric shapes of $O_2^i$ are similar. Also, we mention that $\bar{v}$ cannot be a constant function unless $\bar{v} = 0$.
Combining the estimates $\eqref{f:7.15}$ and  $\eqref{f:7.16}$, we obtain that
\begin{equation}\label{f:7.23}
\begin{aligned}
\int_{O_2}\big(|\nabla \bar{v}| + |\nabla^2\bar{v}|\big)
\lesssim \int_{O_2}|\nabla^{\alpha}\bar{v}|.
\end{aligned}
\end{equation}

Moreover,
let $\hat{B}:=B_r(\hat{x})$ with $\text{supp}(g)\subset \hat{B}$
and $r=\delta(\hat{x})$, and the sets $\{\mathcal{R}_{j}\}$
constitute a ring-shaped decomposition of $O_{2}$, i.e.,
$\mathcal{R}_0:=B_{2r}(\hat{x}_0)\cap O_{2}$ with $\hat{x}_0\in\partial\tilde{\Omega}$ being the projection point of $\hat{x}$ on the boundary $\partial\tilde{\Omega}$;
$\mathcal{R}_{j} := \big(B_{2^{j+1}r}(\hat{x}_0)\setminus B_{2^{j}r}(\hat{x}_0)\big)\cap O_{2}$. Also, there hold
$|\mathcal{R}_{j}|=c_d(2^jr)^{d-1}$ with $N:=\log_2(R_0/\varepsilon)$ and
$\text{dist}(\mathcal{R}_j,\hat{B})\sim 2^jr$.
Let $\bar{G}(\cdot,\cdot)$ be Green's function associated with the
elliptic operator $-\nabla\cdot\bar{a}\nabla$ and $\tilde{\Omega}$.
For any integer $k\geq 0$, there hold the decay estimates
\begin{subequations}
\begin{align}
& |\nabla_x\nabla_y \bar{G}(x,y)|
 \lesssim_{\lambda,d,m_0,R_0M_0}^{\eqref{pri:8.2}}|x-y|^{-d}; \label{f:7.27} \\
& |\nabla_x^k\nabla_y \bar{G}(x,y)|
 +  |\nabla_x\nabla_y^k \bar{G}(x,y)|
 \lesssim_{\lambda,d,\tau,k,m_0,R_0^{k-1+\tau}M_0}^{\eqref{pri:8.4}}
 |x-y|^{1-k-d}
 \qquad \forall x,y\in\Omega.
 \label{f:7.27-h}
\end{align}
\end{subequations}
To see this, let $G(\cdot,\cdot)$ be Green's function associated with the
elliptic operator $-\nabla\cdot\bar{a}\nabla$ and $\Omega$. By the uniqueness of the solution we have $G(y,y')=\varepsilon^{2-d}\bar{G}(y/\varepsilon,y'/\varepsilon)$ with
$y,y'\in\Omega$. By setting $x=y/\varepsilon$ and $x'=y'/\varepsilon$, one can derive that
$|\nabla_x\nabla_{x'}\bar{G}(x,x')|=\varepsilon^{d}|\nabla_y\nabla_{y'}G(y,y')|
\lesssim^{\eqref{pri:8.2}}_{\lambda,d,R_0M_0}\varepsilon^{d}|y-y'|^{-d}
=|x-x'|^{-d}$
(where we employ Lemma $\ref{lemma:8.1}$ and take $\tau=1$ therein), which
leads to the stated estimate $\eqref{f:7.27}$. By the same token, we can obtain the estimate $\eqref{f:7.27-h}$.
Then,
it follows from the representation of Green's function and its decay estimates that
\begin{equation}\label{f:7.26}
\begin{aligned}
\int_{O_2}|\nabla^{\alpha}\bar{v}|
&\leq \sum_{j=0}^{N}\int_{\mathcal{R}_j}|\nabla^{\alpha}\bar{v}|
=\sum_{j=0}^{N}
\int_{\mathcal{R}_j}\big|\int_{\hat{B}}\nabla_x^{\alpha}\nabla_y \bar{G}(x,y)g(y)dy\big|dx\\
&\lesssim^{\eqref{f:7.27-h}}
\sum_{j=0}^{N}\int_{\hat{B}}|g(y)|
\int_{\mathcal{R}_j}\frac{dx}{|x-y|^{d+|\alpha|-1}} dy
\lesssim \sum_{j=0}^{N}\frac{1}{r^{|\alpha|-1}}\frac{(2^jr)^{d-1}}{(2^jr)^d}
\lesssim \sum_{j=0}^{N}2^{-j}r^{-|\alpha|}
\lesssim r^{-|\alpha|}.
\end{aligned}
\end{equation}
This together with $\eqref{f:7.23}$ leads to
the stated estimate $\eqref{f:7.24}$ in the case of $d\geq 3$.
For the case $d=2$, by using the same calculation as in
$\eqref{f:7.26}$, we can obtain
\begin{equation*}
\int_{O_2}\big(|\nabla \bar{v}| + |\nabla^2\bar{v}|\big)
\lesssim^{\eqref{f:7.27}}_{\eqref{f:7.27-h}}
\sum_{j=0}^{N}\int_{\hat{B}}|g(y)|\Big\{
\int_{\mathcal{R}_j}\frac{dx}{|x-y|^d}
+\int_{\mathcal{R}_j}\frac{dx}{|x-y|^{d+1}}\Big\} dy
\lesssim
\sum_{j=0}^{N}2^{-j}\big(r^{-1}+r^{-2}\big)
\lesssim r^{-1},
\end{equation*}
and this implies the stated estimate $\eqref{f:7.24}$
in the case of $d=2$.

%Moreover, we decompose $O_2$ as follows:
%\begin{equation}\label{}
% O_2 \subset \cup_{k=1}^{k_0} O_{2}^k;
% \qquad \mathcal{H}^d(O_{2}^k) \approx c_d 2^{k(d-1)},
%\end{equation}
%where $k_0:=\log_2(R_0/\varepsilon)$, and $c_d>0$ is a constant depending only on $d$.  One can construct the following
%auxiliary equations:
%\begin{equation}\label{pde:7.2}
%-\nabla \cdot\bar{a}^*\nabla z = \nabla\cdot g\quad\text{in}\quad\mathbb{R}^d;
%\qquad
%\left\{\begin{aligned}
%\nabla \cdot\bar{a}^*\nabla w&=0 &\quad&\text{in}\quad \tilde{\Omega}, \\
%w &=-z &\quad&\text{on}\quad \partial\tilde{\Omega},
%\end{aligned}\right.
%\end{equation}
%and there obviously holds $\bar{v} = z+w$. Let $1<p<\infty$.
%Thus, from the first equation stated above, it follows from singular integral estimates that
%\begin{equation}\label{}
% \sum_{|\beta|=n}\|\nabla^\beta z\|_{L^p(\mathbb{R}^d)}
% \lesssim_{p,d}\sum_{|\beta'|=n-1}\|\nabla^{\beta'} g\|_{L^p(\mathbb{R}^d)}
%\end{equation}
%For the second equations stated in $\eqref{pde:7.2}$, one has that
%\begin{equation}\label{}
% \|(\nabla^{\beta}w)^{*}\|_{L^p(\partial\tilde{\Omega})}
% \lesssim_{p,d}\|\nabla^{\beta}z\|_{L^p(\partial\tilde{\Omega})}
%\end{equation}

\textbf{Step 5}.  Show the estimate $\eqref{f:7.25}$.
To do so, let $k\in\mathbb{N}$ with $k\in[1,\lceil d/2\rceil]$, and we first claim the following inequality\footnote{In fact, we are only concerned to investigate the situation where $r$ is relatively large, i.e, $r\gg 1$.}
\begin{equation}\label{f:7.34}
\int_{O_{r/4}\setminus O_3} |\bar{v}(x)|\delta^{-k}(x)dx
\lesssim_{k,d}
\int_{O_{r/4}} |\nabla^{k}\bar{v}(x)|dx
 + \ln r\int_{S_3}|\nabla^{k-1}\bar{v}(y)|dS(y),
\end{equation}
where we note that $S_R:=\{x\in\tilde{\Omega}:\text{dist}(x,\partial\tilde{\Omega})=R\}$ with $R>0$ and the reader can find its proof later on. Let $\theta_\beta:=\lceil\beta\rceil-\beta$, and we recall that
$\lceil\beta\rceil$ represents the smallest integer not less than $\beta$.
We then have the following decomposition that
\begin{equation}\label{f:7.20}
\begin{aligned}
&\int_{\tilde{\Omega}\setminus O_2}|\nabla^2\bar{v}|\delta^{-\beta}
\leq \int_{O_{3}}|\nabla^2\bar{v}|
+ r^{\theta_\beta}\int_{O_{r/4}\setminus O_{3}}|\nabla^2\bar{v}|\delta^{-\lceil\beta\rceil}
+ \int_{\tilde{\Omega}\setminus O_{r/4}}|\nabla^2\bar{v}|\delta^{-\beta}\\
&\lesssim^{\eqref{f:7.23}}_{\eqref{f:7.34}} \int_{O_{3}}
|\nabla^\alpha\bar{v}|
+r^{\theta_\beta}
\bigg\{\int_{O_{r/4}}|\nabla^{\lceil\beta\rceil+2}\bar{v}|
+\ln r\int_{S_3}|\nabla^{\lceil\beta\rceil+1}\bar{v}|\bigg\}
+r^{-\beta}\int_{\tilde{\Omega}\setminus O_{r/4}}|\nabla^2\bar{v}|.
\end{aligned}
\end{equation}

Suppose that the sets $\{\mathcal{O}_{j}\}$
constitute a ring-shaped decomposition of $O_{r/4}$, similar to
those given in \textbf{Step 4}. Let $\hat{x}_0\in\partial\tilde{\Omega}$ be  the projection point of $\hat{x}$ on the boundary $\partial\tilde{\Omega}$.
$\mathcal{O}_0:=B_{2r}(\hat{x}_0)\cap O_{r/4}$ and
$\mathcal{O}_{j} := \big(B_{2^{j+1}r}(\hat{x}_0)\setminus B_{2^{j}r}(\hat{x}_0)\big)\cap O_{r/4}$. Moreover, we have
$|\mathcal{O}_{j}|=c_dr(2^jr)^{d-1}$ and $N:=\log_2(R_0/\varepsilon)$. Let
$0<p-1\ll 1$ and $1/p+1/p'=1$, and it follows that
\begin{equation}\label{f:7.19-1}
\begin{aligned}
\int_{O_{r/4}}|\nabla^{\lceil\beta\rceil+2}\bar{v}|
&\leq \sum_{j=0}^{N}
\int_{\mathcal{O}_{j}}|\nabla^{\lceil\beta\rceil+2}\bar{v}|
\lesssim \sum_{j=0}^{N}
  \int_{\hat{B}}|g(y)|\Big(\int_{\mathcal{O}_j}
  |\nabla_x^{\lceil\beta\rceil+2}\nabla_y\bar{G}(x,y)|dx\Big)dy \\
&\lesssim^{\eqref{f:7.27-h}} \sum_{j=0}^{N} \frac{1}{r^{\lceil\beta\rceil+1}}\frac{|\mathcal{O}_j|}{(2^jr)^{d}}
\lesssim r^{-1-\lceil\beta\rceil}.
\end{aligned}
\end{equation}
By the same token, we obtain that
\begin{equation}\label{f:7.19}
\begin{aligned}
\int_{S_3}|\nabla^{\lceil\beta\rceil+1}\bar{v}|
&  \lesssim \sum_{j=0}^{N}
  \int_{\hat{B}}|g(y)|\Big(\int_{\mathcal{O}_j\cap S_3}
  |\nabla_x^{\lceil\beta\rceil+1}\nabla_y\bar{G}(x,y)|dx\Big)dy\\
&  \lesssim^{\eqref{f:7.27-h}}
\sum_{j=0}^{N}\frac{|\mathcal{O}_j\cap S_3|}{(2^jr)^{d+\lceil\beta\rceil}}
\lesssim r^{-1-\lceil\beta\rceil}.
\end{aligned}
\end{equation}

%Combining the estimates $\eqref{f:7.17}$, $\eqref{f:7.18}$, $\eqref{f:7.18**}$, and $\eqref{f:7.19}$, we have
%\begin{equation}\label{f:7.21}
%\begin{aligned}
%\int_{\mathcal{O}_{j}}|\eta\nabla^2\bar{v}|\delta^{-\lceil\beta\rceil}
%\lesssim
%2^{-j}r^{-1-\lceil\beta\rceil}
%+ r^{-(|\alpha|\vee 2)+\frac{d}{p'}}(2^j)^{\frac{d-1}{p'}-1},
%\end{aligned}
%\end{equation}
%where we note that $(|\alpha|\vee 2)-\frac{d}{p'}>\frac{d}{2}$
%and $\frac{d-1}{p'}-1<0$ since $p'\gg 1$.

Let $U:=B_{r+1}(\hat{x})$ and $U\supseteq \hat{B}$. We now handle the last term in $\eqref{f:7.20}$. By using the interior $H^2$-estimate and
the decay of Green's function, as well as the energy estimate,
we obtain that
\begin{equation}\label{f:7.22}
\begin{aligned}
&\int_{\tilde{\Omega}\setminus O_{r/4}}|\nabla^2\bar{v}|
\leq \int_{U}|\nabla^2\bar{v}|
+ \int_{(\tilde{\Omega}\setminus O_{r/4})\setminus U}|\nabla^2\bar{v}|\\
&\leq |U|^{\frac{1}{2}}\Big(\int_{U}|\nabla^2\bar{v}|^2\Big)^{\frac{1}{2}}
+ \int_{(\tilde{\Omega}\setminus O_{r/4})\setminus U}
\big|\int_{\hat{B}}\nabla_x^2\nabla_y\bar{G}(x,y)g(y)dy\big| dx \\
&\lesssim^{\eqref{f:7.27}} |U|^{\frac{1}{2}}\bigg\{\frac{1}{r}
\Big(\int_{\tilde{\Omega}}|\nabla\bar{v}|^2\Big)^{\frac{1}{2}}
+\Big(\int_{2U}|\nabla g|^2\Big)^{\frac{1}{2}}\bigg\}
+ \int_{\hat{B}}|g(y)|\int_{(\tilde{\Omega}\setminus O_{r/4})\setminus U}\frac{dx}{|x-y|^{d+1}}dy\\
&\lesssim
|U|^{\frac{1}{2}}\bigg\{\frac{1}{r}
\Big(\int_{\hat{B}}|g|^2\Big)^{\frac{1}{2}}
+\Big(\int_{\hat{B}}|\nabla g|^2\Big)^{\frac{1}{2}}\bigg\}
+ \frac{1}{r}\int_{1}^{R_0/\varepsilon}\frac{ds}{s^2}
\lesssim r^{-1}.
\end{aligned}
\end{equation}

Plugging the estimates $\eqref{f:7.26}$, $\eqref{f:7.19-1}$, $\eqref{f:7.19}$, and $\eqref{f:7.22}$ back into
$\eqref{f:7.20}$ leads to the desired estimate $\eqref{f:7.25}$.
We now turn to the arguments for the estimate $\eqref{f:7.34}$.
Let $\{Q^{i}_{1,r}\}$ be a family of ``cubes'' such that
$O_{r/4}\setminus O_3 \subset \cup_{i} Q^{i}_{1,r}$ and
$Q^{i}_{1,r}:= \big([0,r/4]\times\tilde{\Delta}_1(z_i)\big)\setminus \big([0,3]\times\tilde{\Delta}_1(z_i)\big)$ with $z_i\in S_3$ and
$|Q^{i}_{1,r}|\sim r$, where $\tilde{\Delta}_1(z_i):=B_1(z_i)\cap S_3$. For any $y\in \tilde{\Delta}_1(z_i)$ and $x\in Q^{i}_{1,r}$,
there hold
\begin{equation*}
 \bar{v}(x) = -\int_{0}^{1}\nabla \bar{v}(x+t(y-x))\cdot(y-x) dt + \bar{v}(y),
 \quad\text{and}\quad |x-y|\leq 2\delta(x).
\end{equation*}
For any $k\in\mathbb{N}$ with $k\in [1,\lceil d/2\rceil]$, this implies
\begin{equation}\label{f:7.36}
\begin{aligned}
& \int_{Q^{i}_{1,r}}|\bar{v}(x)|\delta^{-k}(x)dx
 \lesssim \int_{Q^{i}_{1,r}}|\nabla \bar{v}(x)|\delta^{1-k}(x)dx
 +\int_{\tilde{\Delta}_1(z_i)}|\bar{v}(y)|dS(y)\int_{Q^{i}_{1,r}}\frac{dx}{\delta^k(x)}\\
& \lesssim \int_{Q^{i}_{1,r}}|\nabla^k \bar{v}(x)|dx
+\sum_{j=0}^{(k-2)\vee 0}\int_{\tilde{\Delta}_1(z_i)}|\nabla^j\bar{v}(y)|dS(y) + \ln r\int_{\tilde{\Delta}_1(z_i)}|\nabla^{k-1}\bar{v}(y)|dS(y)\\
&  \lesssim \int_{Q^{i}_{1,r}}|\nabla^k \bar{v}(x)|dx
+\int_{\tilde{\Delta}_1(z_i)}|\bar{v}(y)|dS(y) + \ln r\int_{\tilde{\Delta}_1(z_i)}|\nabla^{k-1}\bar{v}(y)|dS(y),
\end{aligned}
\end{equation}
where we employ the interpolation inequality in the last inequality (see e.g. \cite[Theorem 5.2]{Adams-03}).
Let $z_i'\in\partial\tilde{\Omega}$ be the projection of $z_i$ onto the boundary $\partial\tilde{\Omega}$, and $\Xi_5^i:=D_5(z_i')\cap O_3$ is such that $\partial\Xi_5^i\supseteq \tilde{\Delta}_1(z_i)$.
By using the trace theorem and the interpolation inequality, we have
\begin{equation}\label{f:7.35}
\begin{aligned}
\int_{\tilde{\Delta}_1(z_i)}|\bar{v}(y)|dS(y)
&\leq \int_{\partial\Xi_5^i}|\bar{v}(y)|dS(y)
\lesssim_{d} \int_{\Xi_5^i}\big(|\bar{v}(x)| + |\nabla \bar{v}(x)|\big)dx\\
&\lesssim\int_{D_5(z_i')}\big(|\nabla^{k}\bar{v}(x)|+|\bar{v}(x)|\big)dx
\lesssim\int_{D_5(z_i')}|\nabla^{k}\bar{v}(x)|dx,
\end{aligned}
\end{equation}
where we also employ the same argument given for $\eqref{f:7.16}$ in the last inequality above. Therefore, plugging the estimate $\eqref{f:7.35}$
back into $\eqref{f:7.36}$, it follows from the decomposition
$O_{r/4}\setminus O_3 \subset \cup_{i} Q^{i}_{1,r}$ that
\begin{equation*}
\begin{aligned}
\int_{O_{r/4}\setminus O_3} |\bar{v}(x)|\delta^{-k}(x)dx
&\leq \sum_{i}
\int_{Q^{i}_{1,r}}|\bar{v}(x)|\delta^{-k}(x)dx \\
&\lesssim \int_{O_{r/4}} |\nabla^k\bar{v}(x)|dx
+ \ln r\int_{S_3}|\nabla^{k-1}\bar{v}(y)|dS(y),
\end{aligned}
\end{equation*}
which gives the stated estimate $\eqref{f:7.34}$.
This completes the whole proof.
\end{proof}

\begin{lemma}[fluctuation estimates]\label{lemma:7.3}
Let $\Omega\subset\mathbb{R}^d$ (with $d\geq 2$) be a bounded $C^{2}$ domain and $\Omega\ni\{0\}$. Suppose that $\langle\cdot\rangle$ satisfies the spectral gap condition $\eqref{a:2}$, and
the admissible coefficients satisfy the smoothness condition $\eqref{a:3}$.
Let $\tilde{\Phi}_{\varepsilon}=
\{\tilde{\Phi}_{\varepsilon,i}\}_{i=1}^d$ be the solution of
the equations $\eqref{pde:10}$, and we set $\tilde{\Omega}:=\Omega/\varepsilon$ with $\varepsilon\in(0,1]$; and   $\Psi_{i}:=\frac{1}{\varepsilon}\tilde{\Phi}_{\varepsilon,i}(\varepsilon\cdot)$.
Then, for any deterministic vector field $g\in C_0^1(\tilde{\Omega})$
and $\alpha\in(0,1)$,
there holds
\begin{equation}\label{pri:7.2}
\bigg\langle\Big(\int_{\tilde{\Omega}}\nabla\Psi\cdot g
-\big\langle\int_{\tilde{\Omega}}\nabla\Psi\cdot g \big\rangle
\Big)^{2p}\bigg\rangle^{1/p}
\lesssim_{\lambda,\lambda_1,\lambda_2,d,m_0,R_0^{\tau}M_0,\tau,p} \int_{\tilde{\Omega}}|g|^2 + \Lambda_{d}\int_{\tilde{\Omega}}|\nabla g|^2\delta^{2\alpha},
\end{equation}
where $\Lambda_{2}:= \mu_2^2(R_0/\varepsilon)$ if $d=2$; and $\Lambda_{d}:= 0$ if $d\geq 3$.
Moreover, let $Q$ be given as in Lemma $\ref{lemma:7.2}$, and we have
\begin{equation}\label{pri:7.3}
\bigg\langle\Big(\int_{\tilde{\Omega}}\nabla Q\cdot g
-\big\langle\int_{\tilde{\Omega}}\nabla Q\cdot g \big\rangle
\Big)^{2p}\bigg\rangle^{1/p}
\lesssim_{\lambda,\lambda_1,\lambda_2,d,m_0,R_0^{\tau}M_0,\tau,p} \int_{\tilde{\Omega}}|g|^2
+\Lambda_{d}\int_{\tilde{\Omega}}|\nabla g|^2\delta^{2\alpha}.
\end{equation}
\end{lemma}

\begin{proof}
The basic idea is to use functional inequalities, and the proof is divided into three steps.

\textbf{Step 1.} Outline the arguments for $\eqref{pri:7.2}$.
For the ease of the statement, we set
$F_i(g):=\int_{\tilde{\Omega}}\nabla\Psi_{i}\cdot g$, where $g$ is
associated with $v\in H_0^1(\tilde{\Omega})$ by $-\nabla\cdot a^{*}\nabla v=\nabla\cdot g$ in $\tilde{\Omega}$. To address the estimate $\eqref{pri:7.2}$,
The key ingredient is
appealing to
$L^p$-version spectral gap inequality (which derived from $\eqref{a:2}$ and see \cite[pp.17-18]{Josien-Otto22} for a proof), i.e.,
\begin{equation}\label{f:5.5}
 \big\langle\big(F_i-\langle F_i\rangle\big)^{2p}\big\rangle^{\frac{1}{p}}
 \lesssim_{\lambda_1,p}
 \Big\langle\Big(\int_{\mathbb{R}^d} \big(\dashint_{B_1(x)}\big|\frac{\partial F_i}{\partial a}\big|\big)^2 dx\Big)^{p}\Big\rangle^{\frac{1}{p}}.
\end{equation}

Then, we claim that
\begin{equation}\label{f:5.4}
  \int_{\mathbb{R}^d}\frac{\partial F}{\partial a(x)}:(\delta a)(x) dx
  = \int_{\tilde{\Omega}}\nabla v \cdot\delta a(\nabla\Psi_{i}+e_i),
\end{equation}
where the notation ``:'' represents the tensor’s inner product of second order, and we can set $(\delta a)(x) = 0$ for $x\not\in\tilde{\Omega}$. Admitting $\eqref{f:5.4}$ for a while, the right-hand side of $\eqref{f:5.5}$ can be further controlled by
\begin{equation}\label{f:5.6}
\begin{aligned}
& \Big\langle\big(\int_{\mathbb{R}^d} \big|\frac{\partial F_i}{\partial a(x)}\big|^2 dx\big)^{p}\Big\rangle
 \leq \int_{\tilde{\Omega}}
 \big\langle\big|\nabla v\otimes(\nabla\Psi_{i}+e_i)\big|^{2p}\big\rangle^{\frac{1}{p}} =:I \\
&\leq \esssup_{x\in\tilde{\Omega}}
\big\langle|\nabla\Psi_{i}+e_i|^{4p}\big\rangle^{\frac{1}{2p}}
\int_{\tilde{\Omega}} \langle|\nabla v|^{4p}\rangle^{\frac{1}{2p}}
\lesssim^{\eqref{pri:1.b},\eqref{pri:12}}
\mu_d^{2}(R_0/\varepsilon)\int_{\tilde{\Omega}}|g|^2
\end{aligned}
\end{equation}
where we use Minkowski's inequality in the first inequality. In fact,
if $d>2$, we have already obtained optimal CLT-scaling (not optimal for
$d=2$) in the form of
\begin{equation}\label{f:5.8}
\big\langle\big(F_i-\langle F_i\rangle\big)^{2p}\big\rangle^{\frac{1}{p}}
 \lesssim^{\eqref{f:5.5},\eqref{f:5.6}} \int_{\tilde{\Omega}}|g|^2.
\end{equation}
%By taking  (note that $\int_{\tilde{\Omega}}\nabla\Psi_{i} =0$), the stated estimate $\eqref{pri:4}$ follows from
%a rescaling.
Now, we handle the case of $d=2$ by modifying
the estimate $\eqref{f:5.6}$ as follows. We recall that   $\delta:=\text{dist}(\cdot,\partial\tilde\Omega)$, $\tilde{R}_0:=R_0/\varepsilon$,
and  $1/\gamma+1/\gamma'=1$ with $0<\gamma'-1\ll 1$.
Let $\vartheta\in(0,1)$ and $\alpha:=1-\vartheta$.
Then, it follows from Lemma $\ref{lemma:*3}$, Theorem $\ref{thm:2}$,
and Lemma $\ref{P:7*}$ that
\begin{equation}\label{f:5.7}
\begin{aligned}
 I
&\lesssim^{\eqref{pri:1.c}} \int_{\tilde{\Omega}}
 \langle|\nabla v|^{4p}\rangle^{\frac{1}{2p}}
 +\mu_2^2(\tilde{R}_0)\bigg\{
 \int_{O_1}\langle|\nabla v|^{4p}\rangle^{\frac{1}{2p}}
 + \int_{\tilde{\Omega}\setminus O_1}
 \langle|\nabla v|^{4p}\rangle^{\frac{1}{2p}}\delta^{-2}\bigg\}\\
&\lesssim \int_{\tilde{\Omega}}
 \langle|\nabla v|^{4p}\rangle^{\frac{1}{2p}}
 +\mu_2^2(\tilde{R}_0)
\int_{\tilde{\Omega}}
 \langle|\nabla v|^{4p}\rangle^{\frac{1}{2p}}\delta^{-2\vartheta}\\
&\lesssim^{\eqref{pri:12}}\int_{\tilde{\Omega}}|g|^2
+ \mu_2^2(\tilde{R}_0)
\int_{\tilde{\Omega}}|g|^{2}\delta^{-2\vartheta}
\lesssim \int_{\tilde{\Omega}}|g|^2
+\mu_2^2(\tilde{R}_0)
\int_{\tilde{\Omega}}|\nabla g|^{2}\delta^{2\alpha},
\end{aligned}
\end{equation}
where we employ the weighted Hardy's inequality (see e.g. \cite[Theorem 1.1]{Lehrback-14}) in the last step. This gives the counterpart of  $\eqref{f:5.8}$ in the case of $d=2$, and together with the case $d\geq 3$ shows the desired estimate $\eqref{pri:7.2}$.

%The key ingredient is reduced to establishing
%\begin{equation}\label{f:5.9}
% \big\langle |F_i-\langle F_i\rangle|^{2p}\big\rangle^{\frac{1}{p}}
% \lesssim \int_{\tilde{\Omega}} |g|^2,
%\end{equation}
%and the desired estimate $\eqref{pri:4}$ follows by taking
%$g=\frac{1}{|B_{\tilde{R}_0/4}|}I_{B_{\tilde{R}_0/4}}$ in
%the above estimate. To estimate the left-hand side of $\eqref{f:5.9}$,
%one has
%\begin{equation}\label{f:5.10}
%\big\langle |F_i|^{2p}\big\rangle^{\frac{1}{p}}
%\lesssim
%\big\langle\big(F_i-\langle F_i\rangle\big)^{2p}\big\rangle^{\frac{1}{p}}
%+ \langle F_i\rangle^2,
%\end{equation}
%where we note that $\langle F_i\rangle \not=0$ is due to the non-stationarity of $\nabla\Psi_i$. We first claim that
%\begin{equation}\label{f:5.11}
% \langle F_i\rangle^2 \lesssim \int_{\tilde{\Omega}}|g|^2.
%\end{equation}
%To see the first term in the left-hand side of $\eqref{f:5.10}$,

\textbf{Step 2.} Show the estimate $\eqref{f:5.4}$.
Let $\Sigma$ be the probability space governed by $\langle\cdot\rangle$,
and then we set the space after suitable perturbations by $\Sigma':=\{a_0+\delta a:
a_0\in\Sigma,~ (\delta a)(x)=0~\text{if}~x\not\in\tilde{\Omega}\}$. Then,
for any $a\in\Sigma'$, one can understand $\Psi_i(a):=\Psi_i(a_0)
+\delta\Psi_i(a_0,\delta a)$ as an extension on $\Sigma'$, where the perturbation part
$\delta\Psi(a_0,\delta a)$ is determined by the following equations
\begin{equation}\label{pde:3}
\left\{\begin{aligned}
-\nabla\cdot a\nabla\delta\Psi(a_0,\delta a)&= \nabla\cdot \delta a
(\nabla\Psi_i(a_0)+e_i) &\quad&\text{in}\quad\tilde{\Omega};\\
\delta\Psi(a_0,\delta a) &=0
&\quad&\text{on}\quad\partial\tilde{\Omega}.
\end{aligned}\right.
\end{equation}
Therefore, with the help of the auxiliary function $v$ defined as in \textbf{Step 1} with the new replacement $a\in\Sigma'$,  we obtain
\begin{equation*}
\begin{aligned}
\delta F
&= \int_{\tilde{\Omega}}g\cdot\nabla \delta\Psi_i
= \int_{\tilde{\Omega}}\nabla\cdot a^{*}\nabla v \delta\Psi_i\\
&=\int_{\tilde{\Omega}} v\nabla\cdot a\nabla\delta\Psi_{i}
=^{\eqref{pde:3}}-\int_{\tilde{\Omega}}
v\nabla\cdot \delta a(\Psi_i+e_i)
=\int_{\tilde{\Omega}}
\nabla v\cdot \delta a(\Psi_i+e_i),
\end{aligned}
\end{equation*}
which gives the stated formula $\eqref{f:5.4}$.

\textbf{Step 3}. Arguments for $\eqref{pri:7.3}$.
Let $\tilde{g}:=g\textbf{1}_{\tilde{\Omega}}$. We note that
$\nabla Q = \nabla\Psi - \nabla\phi$, and therefore
the idea is based upon
the stated estimate $\eqref{pri:7.2}$.
\begin{equation*}
\begin{aligned}
\bigg\langle\Big(\int_{\tilde{\Omega}}\nabla Q\cdot g
-\big\langle\int_{\tilde{\Omega}}\nabla Q\cdot g \big\rangle
\Big)^{2p}\bigg\rangle^{\frac{1}{p}}
&\lesssim
\bigg\langle\Big(\int_{\tilde{\Omega}}\nabla\Psi\cdot g
-\big\langle\int_{\tilde{\Omega}}\nabla\Psi\cdot g \big\rangle
\Big)^{2p}\bigg\rangle^{\frac{1}{p}}
+\bigg\langle\big|
\int_{\mathbb{R}^d}\nabla\phi\cdot \tilde{g}\big|^{2p}\bigg\rangle^{\frac{1}{p}}\\
&\lesssim^{\eqref{pri:7.2},\eqref{pri:*2}}
\int_{\tilde{\Omega}}|g|^2,
\end{aligned}
\end{equation*}
which consequently leads to the desired estimate $\eqref{pri:7.3}$. This completes the whole proof.
\end{proof}

\textbf{Proof of Proposition $\ref{P:8}$.}
The case of $d=2$ follows from Theorem $\ref{thm:2}$.
If $d\geq 3$ and $0<\beta<(d/2)$, we can use an induction argument
to get an iteration formula. To do so, one can assume that there holds
\begin{equation}\label{pri:7.2condition**}
 \big\langle|\nabla Q(z)|^p\big\rangle^{\frac{1}{p}}
 \lesssim \mu_d^{\lfloor\beta\rfloor}(R_0/\varepsilon)
 \big[\delta(z)\big]^{-\beta}
 \qquad \forall z\in\tilde{\Omega}\quad\text{with}\quad \delta(z)\geq 2.
\end{equation}
Let $R:=\text{dist}(x_0,\partial\tilde{\Omega})/4$ with
$\text{dist}(x_0,\partial\tilde{\Omega})\geq 2$, and $1\leq p_0<p<\infty$. Set $F_{m,j}:=S_{2^{-k_0-\frac{1}{\rho}+j}R}(\nabla Q)(\hat{x}_m^j)$, where
$\hat{x}_m^j$ is given as in Lemma $\ref{lemma:7.1}$.
By uniform Lipschitz estimates, one derives that
\begin{equation}\label{f:7.30}
\begin{aligned}
& \big\langle\big|\nabla Q(x_0)\big|^{p_0}\big\rangle^{1/{p_0}}
 \lesssim^{\eqref{pri:2.10}} \bigg\langle\Big(\dashint_{B_{R^{1-\rho}}(x_0)}|\nabla Q|^2\Big)^{\frac{p}{2}}\bigg\rangle^{\frac{1}{p}}\\
 &\lesssim^{\eqref{pri:7.1}}
 \frac{1}{R^{\frac{d}{2}}}\bigg\langle\Big(\dashint_{B_{R}(x_0)}|\nabla Q|^2\Big)^{\frac{p}{2}}\bigg\rangle^{\frac{1}{p}}
 +  c_{\theta}^{-1}\sum_{j=1}^{k_0}2^{-\frac{j}{\rho}}
\sum_{m=1}^{c_{\theta}}\big\langle |F_{m,j}|^p\big\rangle^{\frac{1}{p}}\\
&\lesssim
\frac{1}{R^{\frac{d}{2}}}\bigg\langle\Big(\dashint_{B_{R}(x_0)}|\nabla Q|^2\Big)^{\frac{p}{2}}\bigg\rangle^{\frac{1}{p}}
 +  c_{\theta}^{-1}\sum_{j=1}^{k_0}2^{-\frac{j}{\rho}}
\sum_{m=1}^{c_{\theta}}\Big\{\big\langle |F_{m,j}-\langle F_{m,j}\rangle|^p\big\rangle^{\frac{1}{p}} + \langle F_{m,j}\rangle
\Big\},
\end{aligned}
\end{equation}
where we mention that $B_{4R}(x_0)\subset\tilde{\Omega}$, and therefore
the multiplicative constant depends only on $\lambda,\lambda_1,\lambda_2,d$.

Then, apply Lemmas $\ref{lemma:7.2}$ and $\ref{lemma:7.3}$ established above to addressing the right-hand side of $\eqref{f:7.30}$. In view of Lemma $\ref{lemma:7.3}$, taking the deterministic vector field $g$
with the components $g_i:=\zeta_r(\cdot-\hat{x})e_i$ with
$r=2^{j}\theta R^{1-\rho}$,
we first have
\begin{equation}\label{f:7.28}
\big\langle |F_{m,j}-\langle F_{m,j}\rangle|^p\big\rangle^{\frac{1}{p}}
\lesssim^{\eqref{pri:7.3}} \bigg(\int_{\mathbb{R}^d} \zeta_{r}^2(\hat{x}_m^j-\cdot)\bigg)^{\frac{1}{2}}
\lesssim 2^{-jd/2}\theta^{-d/2} R^{-(1-\rho)d/2},
\end{equation}
and it follows from Lemma $\ref{lemma:7.2}$ together with $\eqref{pri:7.2condition**}$ that
\begin{equation}\label{f:7.29}
\big\langle F_{m,j} \big\rangle
\lesssim \mu_d^{\lfloor\beta\rfloor+1}(R_0/\varepsilon)
\Big\{(2^j\theta R^{1-\rho})^{-\beta-1}\ln(2^j\theta R^{1-\rho})+
(2^j\theta R^{1-\rho})^{-\frac{d}{2}}\Big\}.
\end{equation}

Plugging the estimates $\eqref{f:7.28}$ and $\eqref{f:7.29}$
back into $\eqref{f:7.30}$, we obtain
\begin{equation}\label{f:7.31}
\begin{aligned}
\big\langle\big|\nabla Q(x_0)\big|^{p_0}\big\rangle^{1/{p_0}}
& \lesssim \mu_d(R_0/\varepsilon)R^{-\frac{d}{2}}
 + \mu_d^{\lfloor\beta\rfloor+1}(R_0/\varepsilon)
\Big\{
 R^{-1-\beta+\kappa}+R^{-\frac{d}{2}+\kappa}\Big\}\\
&\lesssim\mu_d^{\lfloor\beta\rfloor+1}(R_0/\varepsilon)
\Big\{
 R^{-1-\beta+\kappa}+R^{-\frac{d}{2}+\kappa}\Big\}
\lesssim
\mu_d^{\lfloor\beta\rfloor+1}(R_0/\varepsilon)
R^{-[(1+\beta)\wedge (d/2)]+\kappa},
\end{aligned}
\end{equation}
where $\kappa:=\big(\rho(1+\beta)+\theta\big)\vee\big(\rho(d/2)\big)$ could be very small with $0<\theta\ll1$. Obviously, if $(1+\beta)<(d/2)$,
we can take $\eqref{f:7.31}$ as the starting point of the iteration and repeat the above process until $(1+\beta)\geq (d/2)$ to stop the iteration.
This leads to the desired estimate $\eqref{pri:7.4}$.

The remaining proof will explain the first step. Due to Theorem $\ref{thm:2}$, the estimate $\eqref{pri:1.c}$ plays a role
as the initial step\footnote{One can take $\beta=1$ in
$\eqref{pri:7.2condition**}$ therein.} for this iteration.
For the case of $d\geq 3$, we can verify the first step starting from $\beta=1$ that \begin{equation}\label{f:7.32}
\begin{aligned}
\big\langle\big|\nabla Q(x_0)\big|^{p_0}\big\rangle^{1/{p_0}}
\lesssim
R^{-1-\beta}
+R^{-1-\beta+\kappa}+R^{-\frac{d}{2}+\kappa}
\lesssim
\left\{\begin{aligned}
&R^{-\frac{3}{2}+\kappa},&\quad&\text{if}~d=3; \\
&R^{-2+\kappa},&\quad&\text{if}~d\geq 4.
\end{aligned}\right.
\end{aligned}
\end{equation}
It is not difficult to see that the above iteration can be further carried out only when $d > 4$. In this case, $\eqref{f:7.32}$ will serve as the starting point $\eqref{pri:7.2condition**}$ for the further iteration.
This ends the whole proof.
\qed

\textbf{Proof of Theorem $\ref{thm:2*}$.}
Based up on the Proposition $\ref{P:8}$,
the estimate $\eqref{pri:1.3}$ simply follows from $\eqref{pri:7.4}$ by
the rescaling argument (where we note that $Q=\frac{1}{\varepsilon}Q_\varepsilon(\varepsilon\cdot)$ and
$\Psi = \frac{1}{\varepsilon}\tilde{\Phi}_\varepsilon(\varepsilon\cdot)$ on $\tilde\Omega$). To see the stated estimate $\eqref{pri:1.4}$, let
$g\in C_0^\infty(2B)$ be a test function satisfying $g=\textbf{1}_B$ in $B$ and
$|\nabla^k g|\lesssim 1/r_B^{k+d}$ for any $k\in\mathbb{N}$, where $B$
is given as in Theorem $\ref{thm:2*}$. By rescaling,
set $\tilde{g}:=g(\varepsilon\cdot)$ on $\tilde{\Omega}$ and $R:=r_B/\varepsilon$ (and therefore $|\nabla^k\tilde{g}|\lesssim R^{-k-d}$).
In view of Lemma $\ref{lemma:7.2}$, one can use
the estimate $\eqref{pri:7.4}$ as the precondition therein, and then we can derive that
\begin{equation*}
\begin{aligned}
\Big\langle\big|\int_{\tilde{\Omega}}
\nabla\Psi\cdot\tilde{g}\big|^p\Big\rangle^{\frac{1}{p}}
&\leq \Big\langle\big|\int_{\tilde{\Omega}}
\nabla\Psi\cdot\tilde{g}-
\big\langle\int_{\tilde{\Omega}}
\nabla\Psi\cdot\tilde{g}\big\rangle\big|^p\Big\rangle^{\frac{1}{p}}
+ \big\langle\int_{\tilde{\Omega}}
\nabla\Psi\cdot\tilde{g}\big\rangle \\
&\lesssim^{\eqref{pri:7.2}}
\Big(\int_{\tilde{\Omega}}|\tilde{g}|^2\Big)^{\frac{1}{2}}
+ \big\langle\int_{\tilde{\Omega}}
\nabla Q\cdot\tilde{g}\big\rangle + \big\langle\int_{\tilde{\Omega}}
\nabla \phi\cdot\tilde{g}\big\rangle\\
&\lesssim^{\eqref{pri:*2},\eqref{pri:7.2outcome}}
\Big(\int_{\tilde{\Omega}}|\tilde{g}|^2\Big)^{\frac{1}{2}}
+ \mu_d^{\lceil\frac{d}{2}\rceil+1}(R_0/\varepsilon)
\Big\{
R^{\kappa-\frac{d}{2}-1}+R^{-\frac{d}{2}}\Big\}\\
&\lesssim R^{-\frac{d}{2}} +
\mu_d^{\lceil\frac{d}{2}\rceil+1}(R_0/\varepsilon)R^{-\frac{d}{2}}
\lesssim \mu_d^{\lceil\frac{d}{2}\rceil+1}(R_0/\varepsilon)
R^{-\frac{d}{2}},
\end{aligned}
\end{equation*}
which leads to the stated estimate $\eqref{pri:1.4}$ by rescaling
the left-hand side above. This ends the whole proof.
\qed

\begin{flushleft}
\textbf{Declaration of competing interests}
\end{flushleft}

The two authors declare that none of them has any competing interests.

\begin{flushleft}
\textbf{Acknowledgements}
\end{flushleft}

The authors appreciate Prof. Armstrong Scott for kindly pointing out some improper statements in the previous version of the manuscript (arXiv:2402.18907v1). Also, the authors are grateful to Dr. Claudia Raithel for sharing the details and ideas in her previous papers, especially regarding to the boundary correctors.
This work was supported by the
Young Scientists Fund of the National Natural Science Foundation of China (Grant No. 12501280).

\bibliographystyle{siamplain}
%\bibliography{references}

\begin{thebibliography}{000}

\bibitem{Adams-03}
R.-A. Adams, J.J.F. Fournier,
Sobolev spaces. Second edition. Pure and Applied Mathematics (Amsterdam), 140. Elsevier/Academic Press, Amsterdam, 2003.

\bibitem{Armstrong-Kuusi22}
S. Armstrong, T. Kuusi,
Elliptic homogenization from qualitative to quantitative,
arXiv:2210.06488v1. (2022).

\bibitem{Armstrong-Kuusi-Mourrat19}
S. Armstrong, T. Kuusi, J.-C. Mourrat,
%Armstrong, S., Kuusi, T.,  Mourrat, J.-C.:
Quantitative Stochastic Homogenization and
Large-scale Regularity. Grundlehren der mathematischen Wissenschaften
[Fundamental Principles of Mathematical Sciences], 352. Springer, Cham,
2019.

\bibitem{Armstrong-Kuusi-Mourrat-16}
S. Armstrong, T. Kuusi, J.-C. Mourrat,
Mesoscopic higher regularity and subadditivity in elliptic homogenization. Comm. Math. Phys. 347 (2016), no. 2, 315-361.

\bibitem{Armstrong-Kuusi-Mourrat-Prange17}
%S. Armstrong, T. Kuusi, J.-C. Mourrat, C. Prange,
S. Armstrong, T. Kuusi, J.-C. Mourrat, C. Prange,
Quantitative analysis of boundary layers in periodic homogenization,
Arch. Ration. Mech. Anal. 226 (2017), no. 2, 695-741.


%\bibitem{Armstrong-Shen16}
%S. Armstrong, Z. Shen, Lipschitz estimates in almost-periodic homogenization, Comm. Pure Appl. Math. 69 (2016), no. 10, 1882-1923.

\bibitem{Armstrong-Smart16}
%S. Armstrong, C. Smart,
S. Armstrong, C. Smart,
Quantitative stochastic homogenization of convex integral functionals,
Ann. Sci. \'Ec. Norm. Sup\'er. (4) 49 (2016), no. 2, 423-481.


\bibitem{Avellaneda-Lin87}
%M. Avellaneda, F. Lin,
M. Avellaneda, F. Lin,
Compactness methods in the theory of homogenization.
Comm. Pure Appl. Math. 40 (1987), no. 6, 803-847.


\bibitem{Avellaneda-Lin91}
M. Avellaneda, F. Lin,
%Avellaneda, M., Lin, F. x.
$L^p$  bounds on singular integrals in homogenization,
Comm. Pure Appl. Math. 44 (1991), no. 8-9, 897-910.

\bibitem{Bella-Fischer-Josien-Raithel-24}
P. Bella,  J. Fischer, M. Josien, C. Raithel,
Boundary layer estimates in stochastic homogenization,
arXiv:2403.12911v1 (2024).


\bibitem{Bella-Giunti-Otto17}
%P. Bella, A. Giunti, F. Otto,
P. Bella, A. Giunti, F. Otto,
Quantitative stochastic homogenization:
local control of homogenization error through corrector. Mathematics and materials, 301-327,
IAS/Park City Math. Ser., 23, Amer. Math. Soc., Providence, RI,
2017.

\bibitem{Blanc-Josien-LeBris-20}
X. Blanc, M. Josien, C. Le Bris,
Precised approximations in elliptic homogenization beyond the periodic setting. Asymptot. Anal. 116 (2020), no. 2, 93-137.


\bibitem{Clozeau-Josien-Otto-Xu}
N. Clozeau, M. Josien, F. Otto, Q. Xu,
%Clozeau, N., Josien, M., Otto, F.,  Xu, Q.:
Bias in the representative volume element method:
periodize the ensemble instead of its realization.
Found. Comput. Math. 24 (2024), no. 4, 1305–1387.

\bibitem{Chua-92}
S.-K. Chua,
Extension theorems on weighted Sobolev spaces. Indiana Univ. Math. J. 41 (1992), no. 4, 1027-1076.

\bibitem{Conlon-Giunti-Otto17}
%J. Conlon, A. Giunti, F. Otto,
J. Conlon, A. Giunti, F. Otto, Green's function for elliptic systems: existence and Delmotte-Deuschel bounds. Calc. Var. Partial Differential Equations 56 (2017), no. 6, Paper No. 163, 51 pp.

\bibitem{E-M-Zhang05}
%W. E, P. Ming, P. Zhang,
W. E, P. Ming, P. Zhang,
Analysis of the heterogeneous multiscale method for elliptic homogenization problems. J. Amer. Math. Soc. 18 (2005), no. 1, 121-156.


%\bibitem{Egert-15}
%M. Egert, R. Haller-Dintelmann, J. Rehberg,
%Hardy's inequality for functions vanishing on a part of the boundary. Potential Anal. 43 (2015), no. 1, 49-78.

\bibitem{Dong-Kim09}
%H. Dong, S. Kim,
H. Dong, S. Kim,
Green's matrices of second order elliptic systems with measurable coefficients in two dimensional domains.
Trans. Amer. Math. Soc. 361 (2009), no. 6, 3303-3323.


\bibitem{Fischer-Raithel17}
%J. Fischer, C. Raithel,
J. Fischer, C. Raithel,
Liouville principles and a large-scale regularity theory for random elliptic operators on the half-space. SIAM J. Math. Anal. 49 (2017), no. 1, 82-114.



\bibitem{Gerard-Masmoudi12} %Gérard-Varet, David; Masmoudi, Nader
%D. G\'erard-Varet, N.  Masmoudi,
D. G\'erard-Varet, N.  Masmoudi,
Homogenization and boundary layers. Acta Math. 209 (2012), no.1, 133-178.


\bibitem{Giaquinta-Martinazzi12}
%M. Giaquinta, L.  Martinazzi,
M. Giaquinta, L.  Martinazzi,
An Introduction to the Regularity Theory
for Elliptic Systems, Harmonic Maps and Minimal Graphs. Second edition. Appunti.
Scuola Normale Superiore di Pisa (Nuova Serie)
[Lecture Notes. Scuola Normale Superiore di Pisa (New Series)],
11. Edizioni della Normale, Pisa, 2012.

\bibitem{Gilbarg-Trudinger-01}
D. Gilbarg, N.S. Trudinger,
Elliptic partial differential equations of second order. Reprint of the 1998 edition. Classics in Mathematics. Springer-Verlag, Berlin, 2001.


\bibitem{Giunti-Otto22}
%A. Giunti, F. Otto,
A. Giunti, F. Otto,
On the existence of the Green function for elliptic systems in divergence form. Manuscripta Math. 167 (2022), no. 1-2, 385-402.

\bibitem{Gloria-Otto11}
A. Gloria, F. Otto, An optimal variance estimate in stochastic homogenization of discrete elliptic equations. Ann. Probab. 39 (2011), no. 3, 779-856.

%\bibitem{Gloria-Otto15}
%%A. Gloria, F. Otto,
%A. Gloria, F. Otto,
%The corrector in stochastic homogenization: optimal rates, stochastic integrability, and fluctuations, arXiv:1510.08290v3 (2015).

\bibitem{Gloria-Neukamm-Otto15}
%A. Gloria, S. Neukamm, F. Otto,
A. Gloria, S. Neukamm, F. Otto,
Quantification of ergodicity in stochastic homogenization: optimal bounds via spectral gap on Glauber dynamics, Invent. Math. 199 (2015), no. 2, 455-515.


\bibitem{Gloria-Neukamm-Otto20}
%A. Gloria, S. Neukamm, F. Otto,
A. Gloria, S. Neukamm, F. Otto,
A regularity theory for random elliptic operators.
Milan J. Math. 88 (2020), no. 1, 99-170.

\bibitem{Gloria-Neukamm-Otto21}
%Gloria, Antoine; Neukamm, Stefan; Otto, Felix
%A. Gloria, S. Neukamm, F. Otto,
A. Gloria, S. Neukamm, F. Otto,
Quantitative estimates in stochastic homogenization for correlated coefficient fields. Anal. PDE 14 (2021), no. 8, 2497-2537.

%\bibitem{Gloria-Otto14}
%Gloria, A., Otto, F.:
%The corrector in stochastic homogenization: the corrector in
%stochastic homogenization: optimal rates, stochastic integrability and
%fluctuations, arXiv:1510.08290v3 (2015)



\bibitem{Hofmann-Kim07}
S. Hofmann, S. Kim,  The Green function estimates for strongly elliptic systems of second order, Manuscripta Math. 124
(2007) 139-172.




\bibitem{Josien-Otto22}
M. Josien, F. Otto, The annealed Calder\'on-Zygmund estimate as convenient tool in quantitative stochastic homogenization. J. Funct. Anal. 283 (2022), no. 7, Paper No. 109594, 74 pp.

\bibitem{Jikov-Kozlov-Oleinik94}
V. Jikov, S. Kozlov, O. Oleinik,
 Homogenization of differential operators and integral functionals. Translated from the Russian by G. A. Yosifian. Springer-Verlag, Berlin, 1994.

\bibitem{Kenig-Lin-Shen14}
%C. Kenig, F. Lin, Z. Shen,
C. Kenig, F. Lin, Z. Shen,
Periodic homogenization of Green and Neumann functions,
Comm. Pure Appl. Math. 67 (2014), no. 8, 1219–1262.

\bibitem{Lehrback-14}
Lehrbäck, J.: Weighted Hardy inequalities beyond Lipschitz domains. Proc. Am. Math. Soc. 142 (2014), no. 5, 1705–1715.


\bibitem{Lu-Otto21}
%J. Lu, F. Otto,
J. Lu, F. Otto,
Optimal artificial boundary condition for random elliptic media. Found. Comput. Math. 21 (2021), no. 6, 1643-1702.

\bibitem{Lu-Otto-Wang21}
J. Lu, F. Otto, L. Wang,
Optimal artificial boundary conditions based on second-order correctors for three dimensional random elliptic media. Comm. Partial Differential Equations 49 (2024), no. 7-8, 609-670.


\bibitem{Shen-Zhuge18}
Z. Shen, J. Zhuge,
Boundary layers in periodic homogenization of Neumann problems,
Comm. Pure Appl. Math. 71 (2018), no. 11, 2163-2219.

\bibitem{Shen16}
Z. Shen,
Boundary estimates in elliptic homogenization, Anal. PDE 10 (2017), 653-694.

\bibitem{Shen18}
Z. Shen,
Periodic Homogenization of Elliptic Systems.
Operator Theory: Advances and Applications, 269.
Advances in Partial Differential Equations (Basel). Birkhäuser/Springer, Cham, 2018.

\bibitem{Taylor-Kim-Brow13}
%J. Taylor, S. Kim, R. Brown,
J. Taylor, S. Kim, R. Brown,
The Green function for elliptic systems in two dimensions, Comm. Partial Differential Equations 38 (2013), no. 9, 1574-1600.


%\bibitem{Torquato02}
%S. Torquato, Random heterogeneous materials. Microstructure and macroscopic
%properties, volume 16 of Interdisciplinary Applied Mathematics. Springer-Verlag, New York (2002).

\bibitem{Wang-Xu22}
L. Wang, Q. Xu,
Calder\'on-Zygmund estimates for stochastic elliptic systems on bounded Lipschitz domains. J. Differential Equations 432 (2025), Paper No. 113200, 61 pp.



\bibitem{Wang-Xu23-1}
L. Wang, Q. Xu,
Annealed Calder\'on-Zygmund
estimates for  elliptic operator with random coefficients on $C^{1}$ domains. arXiv:2405.19102v1 (2024).


%\bibitem{Wang-Xu23-2}
%L. Wang, Q. Xu,
%Quenched estimates on nonhomogeneous elliptic operators with random coefficients. (In preparation).

\bibitem{Wang-Zhang23}
W. Wang, T. Zhang,
Homogenization theory of elliptic system with lower order terms for dimension two. Commun. Pure Appl. Anal. 22 (2023), no. 3, 787-824.

\bibitem{Xu16}
Q. Xu,
Uniform regularity estimates in homogenization theory of elliptic system with lower order terms. J. Math. Anal. Appl. 438 (2016), no. 2, 1066-1107.

\end{thebibliography}

\appendix

\section{Appendix}

\setcounter{lemma}{0}
\setcounter{equation}{0}

\renewcommand\thelemma{A.\arabic{lemma}}
\renewcommand\theequation{A.\arabic{equation}}

\begin{lemma}[scaling-invariant estimates]\label{lemma:8.1}
Let $\Omega\subset\mathbb{R}^d$ (with $d\geq 2$) be a bounded uniform $C^{1,\tau}$  domain with $\tau\in(0,1]$. Assume $v$ is a weak solution to $-\nabla\cdot \bar{a}\nabla v = \nabla\cdot f$ in $2D$ and $v=0$ on $2\Delta$
with $x_B\in\partial\Omega$. Then, we have the scaling-invariant estimate
\begin{subequations}\label{}
\begin{align}
& \|\nabla v\|_{L^\infty(D)}
 \lesssim_{\lambda,d,m_0,\tau}\big(R_0^{\tau} M_0\big)^{\frac{1}{\tau}(d+1)} \dashint_{2D}|\nabla v|
 + \big(R_0^{\tau} M_0\big)^{\frac{1}{\tau}-1}r_B^{\tau}\big[f\big]_{C^{0,\tau}(2D)}
 +\big(R_0^{\tau} M_0\big)^{\frac{1}{\tau}}\|f\|_{L^{\infty}(2D)};\label{pri:8.1}\\
&  r_{B}^{\tau}[\nabla v]_{C^{0,\tau}(D)}
 \lesssim_{\lambda,d,m_0,\tau}\big(R_0^{\tau} M_0\big)^{\frac{1}{\tau}(d+1)+1} \dashint_{2D}|\nabla v|
+ \big(R_0^{\tau} M_0\big)^{\frac{1}{\tau}}r_B^{\tau}\big[f\big]_{C^{0,\tau}(2D)}
 +\big(R_0^{\tau} M_0\big)^{\frac{1}{\tau}+1}\|f\|_{L^{\infty}(2D)}.\label{pri:8.1*}
\end{align}
\end{subequations}
Moreover, let $\bar{G}(x,\cdot)$ be Green's function associated with the
elliptic operator $-\nabla\cdot\bar{a}\nabla$. Then, there holds the decay estimate
\begin{equation}\label{pri:8.2}
 |\nabla_x\nabla_y \bar{G}(x,y)|
 \lesssim_{\lambda,d,R_0^{\tau}M_0,\tau}|x-y|^{-d}
 \qquad \forall x,y\in\Omega.
\end{equation}
\end{lemma}

\begin{proof}
A similar result of $\eqref{pri:8.1}$ and $\eqref{pri:8.1*}$ may be found in \cite[Lemma 6.5]{Gilbarg-Trudinger-01}, and note the remark on page 98
regarding the dependence of the estimated constant.
Based on the estimate $\eqref{pri:8.1}$ and
the argument presented in $\eqref{pri:2.5-d}$, we can similarly derive estimate $\eqref{pri:8.2}$, and therefore we do not repeat the demonstration here. Instead, we will provide a proof of $\eqref{pri:8.1}$ for the reader's convenience. Let $B_1(x)$ be such that $x\in\partial\Omega$ or $B_1(x)\subset D$. There holds the interior estimate:
\begin{equation}\label{f:8.1}
 \|\nabla v\|_{C^{0,\tau}(B_{1/2}(x))}
 \lesssim_{\lambda,d} \dashint_{B_{1}(x)}|\nabla v|
 +\big[f\big]_{C^{0,\tau}(B_1(x))}
 +\|f\|_{L^{\infty}(B_1(x))},
\end{equation}
and the boundary estimate:
\begin{equation}\label{f:8.2}
\begin{aligned}
 \|\nabla v\|_{L^{\infty}(D_{1}(x))}
 +
\|\nabla\psi\|_{C^{0,\tau}(\mathbb{R}^{d-1})}^{-1}
& [\nabla v]_{C^{0,\tau}(D_{1}(x))}
\lesssim_{\lambda,d,\tau,m_0}[\nabla\psi]^{\frac{1}{\tau}(d+1)
 }_{C^{0,\tau}(\mathbb{R}^{d-1})}
 \dashint_{D_{2}(x)}|\nabla v| \\
& + [\nabla\psi]_{C^{0,\tau}(\mathbb{R}^{d-1})}^{\frac{1}{\tau}-1}
\big[f\big]_{C^{0,\tau}(D_2(x))}
 +[\nabla\psi]_{C^{0,\tau}(\mathbb{R}^{d-1})}^{\frac{1}{\tau}}\|f\|_{L^{\infty}(D_2(x))},
\end{aligned}
\end{equation}
where we recall that $\psi$ is the local boundary function introduced in
\eqref{boundary-1}, and the dependence of the constant on $[\nabla\psi]_{C^{0,\tau}(\mathbb{R}^{d-1})}$ in the right-hand side above can be referred to
\cite[Lemma A.3]{Josien-Otto22}. To see the dependence on $m_0$, set
$x=\Psi(x')$ with $y=\Psi(y')$ and $x,y\in D_1$, where $\Psi$ is a $C^{1,\tau}$ diffeomorphism (determined by $\psi$). There holds
$\frac{1}{m_0}|x'-y'|\leq |x-y|\leq m_0|x'-y'|$ for any $x,y\in D_1$.
Moreover, let $u(x)=u(\Psi(x'))=:\tilde{u}(x')$, and we can derive that
$\frac{1}{m_0}\|\nabla u\|_{L^{\infty}(D_1)}
\leq \|\nabla \tilde{u}\|_{L^\infty(D_1')}\leq m_0 \|\nabla u\|_{L^{\infty}(D_1)}$, where $D_1':=\Psi^{-1}(D_1)$ and
$\Delta_1' :=\Psi^{-1}(\Delta_1)$ is correspondingly the flat boundary. This leads to the
dependence on $m_0$ in $\eqref{f:8.2}$, and
the appearance of the product factor $\|\nabla\psi\|_{C^{0,\tau}(\mathbb{R}^{d-1})}^{-1}$ in the second term on the left-hand side of $\eqref{f:8.2}$.

The proof is divided into two cases: (1) $0<r_B\leq 1$;
(2) $1<r_B\leq R_0$.

\textbf{Step 1.} Arguments for the case of $0<r_B\leq 1$. The idea is a rescaling argument. By translation it is fine to assume $x_B=0$.
Let $\tilde{v}(y) = \frac{1}{r_B}v(r_B y)$, and $x=r_By$. Hence, there holds $\nabla_y\cdot \bar{a}\nabla_y \tilde{v} = 0$ in $2D_1$ and $v=0$ on $2\Delta_1$. Let $\tilde{\psi}(y_1,\cdots,y_{d-1}):=\frac{1}{r_B}\psi(r_By_1,\cdots,r_By_{d-1})$.
By applying the estimate $\eqref{f:8.2}$ to $\tilde{v}$, one can derive that
\begin{equation*}
\begin{aligned}
\|\nabla v\|_{L^\infty(D)}
=\|\nabla \tilde{v}\|_{L^\infty(D_{1}(0))}
&\lesssim_{\lambda,d,m_0\tau}[\nabla\tilde{\psi}]^{\frac{1}{\tau}(d+1)
 }_{C^{0,\tau}(\mathbb{R}^{d-1})}
 \dashint_{D_{2}(0)}|\nabla \tilde{v}| \\
&\lesssim
\big(r_B^{\tau}[\nabla\psi]_{C^{0,\tau}(\mathbb{R}^{d-1})}\big)^{\frac{1}{\tau}(d+1)
 }
 \dashint_{D}|\nabla v|
 \lesssim M_0^{\frac{1}{\tau}(d+1)
 }
 \dashint_{D}|\nabla v|,
\end{aligned}
\end{equation*}
which implies the stated estimate $\eqref{pri:8.1}$ in the case of $0<r_B\leq 1$.

\textbf{Step 2.} Arguments for the case of $1<r_B\leq R_0$.
The idea is a covering argument. Collect such balls that make estimates $\eqref{f:8.1}$ and $\eqref{f:8.2}$ hold
to form a covering of $D$, i.e.,
$D\subset\cup_{i=1}^N U_1(x_i)$. Then, we have
\begin{equation*}
\begin{aligned}
 \|\nabla v\|_{L^\infty(D)}
& \leq \max_{1\leq i\leq N}\|\nabla v\|_{L^\infty(U_1(x_i))}\\
&\lesssim_{\lambda,d,m_0,\tau}
\max_{1\leq i\leq N}[\nabla\psi_i]^{\frac{1}{\tau}(d+1)
 }_{C^{0,\tau}(\mathbb{R}^{d-1})}
 \int_{2D}|\nabla v|
\lesssim^{\eqref{boundary-1}} \big(R_0^{\tau}M_0\big)^{\frac{1}{\tau}(d+1)}
\dashint_{2D}|\nabla v|.
\end{aligned}
\end{equation*}
This together with the case (1) gives the desired estimate $\eqref{pri:8.1}$. Finally, we point out that $R_0^{\tau}M_0$ is an invariant quantity under the scaling transformation. To see so, let $\delta>0$, $x\in \Omega$, and $y=x/\delta$. Therefore, it concludes
$y\in\hat{\Omega}:=\Omega/\delta$ and $\hat{R}_0:=R_0/\delta$ represents the diameter of $\hat{\Omega}$. Let $\hat{\psi}_{\delta}(y_1,\cdots,y_{d-1}):=\delta^{-1}\psi(\delta y_1,
\cdots,\delta y_{d-1})$ denote the boundary function associated with
$\hat{\Delta}:=\Delta/\delta$, and $\hat{M}_0$ on $\hat{\Omega}$
is the counterpart of $M_0$ shown in $\eqref{boundary-1}$. Therefore,
from the relationship $[\nabla\hat{\psi}_{\delta}]_{C^{0,\tau}(\hat{\Delta})}=\delta^{\tau}
[\nabla\psi]_{C^{0,\tau}(\Delta)}$, it follows that $\hat{M}_0 = \delta^{\tau} M_0$, and one further observe that $R_0^{\tau}M_0
= \hat{R}_0^{\tau}\hat{M}_0$ (which is the scaling-invariant property). This completes the whole proof.
\end{proof}

\begin{lemma}[higher-order counterparts]\label{lemma:8.2}
Let $\Omega\subset\mathbb{R}^d$ (with $d\geq 2$) be a bounded $C^{k,\tau}$ domain with $k\geq 2$ and $\tau\in(0,1]$. Assume $v$ is a weak solution to $-\nabla\cdot \bar{a}\nabla v = \nabla\cdot f$ in $2D$ and $v=0$ on $2\Delta$
with $x_B\in\partial\Omega$. Then, we have the scaling-invariant estimate
\begin{equation}\label{pri:8.3}
r_B^k\|\nabla^k v\|_{L^\infty(D)}
+ r_B^{k+\tau}[\nabla^k v]_{C^{0,\tau}(D)}
 \lesssim
 \dashint_{2D}|v|
 + \sum_{n=0}^{k-1}\Big\{r_B^{n+\tau}\big[f\big]_{C^{n,\tau}(2D)}
 +r_B^{n}\|f\|_{C^{n}(2D)}\Big\},
\end{equation}
where the multiplicative constant depends on
$\lambda$, $d$, $m_0$, $R_0^{\tau}M_0$, $R_0^{k-1+\tau}M_0$, $\tau$, and $k$ at most.
Moreover, let $\bar{G}(x,\cdot)$ be Green's function associated with the
elliptic operator $-\nabla\cdot\bar{a}\nabla$. Then, there hold the decay estimates:
\begin{subequations}
\begin{align}
&  |\nabla_x^k\nabla_y \bar{G}(x,y)|
 +  |\nabla_x\nabla_y^k \bar{G}(x,y)|
 \lesssim_{\lambda,d,\tau,k,R_0^{\tau}M_0,R_0^{k-1+\tau}M_0}|x-y|^{1-k-d}; \label{pri:8.4}\\
& |\nabla_x^k\nabla_y^2 \bar{G}(x,y)|
 +  |\nabla_x^2\nabla_y^k \bar{G}(x,y)|
 \lesssim_{\lambda,d,\tau,k,R_0^{\tau}M_0,R_0^{k-1+\tau}M_0}|x-y|^{-k-d}, \qquad \forall x,y\in\Omega. \label{pri:8.5}
\end{align}
\end{subequations}
\end{lemma}

\begin{proof}
We carry out an induction argument to demonstrate the stated estimate $\eqref{pri:8.3}$, while the estimates $\eqref{pri:8.4}$
and $\eqref{pri:8.5}$ are merely a corollary of $\eqref{pri:8.3}$ and therefore is left to the reader
(referring to the proof of $\eqref{pri:2.5-d}$ for hints). The proof on
$\eqref{pri:8.3}$ is divided into two steps.

\textbf{Step 1.} Show the estimate $\eqref{pri:8.3}$ in the case of $k=2$.
Set $x'=\Phi(x)$ with $x\in D_1$ and therefore $x=\Psi(x')$, where $\Phi$ is a $C^{2,\tau}$ diffeomorphism (determined by $\psi$) straightening out the boundary, and $\Psi=\Phi^{-1}$. Let $v(x)=v(\Psi(x')):=\tilde{v}(x')$,
$D_1'=\Phi(D_1)$ with $\Delta_1' = \Phi(\Delta_1)$ being the flat boundary.
For the ease of the statement, we introduce $\tilde{\nabla}:=\nabla_{x'}$
and $\tilde{\partial}_i := \partial_{x'_i}$. We have
\begin{equation}\label{pde:8.1}
 -\tilde{\nabla} \cdot A\tilde{\nabla} \tilde{v} = \tilde{\nabla} \cdot\tilde{f}
 \quad\text{in}\quad D_1';
 \qquad \tilde{v} = 0 \quad\text{on}\quad\Delta_1',
\end{equation}
where $\tilde{f}(x'):=f(\Psi(x'))$, and $A:=(\nabla\Phi) \bar{a}(\nabla\Phi)^{T}$ is the new elliptic type coefficients satisfying $\frac{\lambda}{m_0^2}|\xi|^2\leq \xi\cdot A\xi\leq\frac{m_0^2}{\lambda}$ for any $\xi\in\mathbb{R}^d$.
Then, we can take the derivative of $\eqref{pde:8.1}$ along the tangent direction. Let $u_n:=\tilde{\partial}_n\tilde{v}$ and $n = 1, ..., d-1$,
and there holds
\begin{equation}\label{pde:8.2}
 -\tilde{\nabla} \cdot A\tilde{\nabla}  u_n = \tilde{\nabla} \cdot\big(\underbrace{\tilde{\partial}_{n}\tilde{f} + \tilde{\partial}_{n}A\nabla\tilde{v}}_{\tilde{f}_1}\big)
 \quad\text{in}\quad D_1';
 \qquad u_n = 0 \quad\text{on}\quad\Delta_1'.
\end{equation}
Let $C_{\psi}^1:=[\nabla\psi]_{C^{0,\tau}}$ represent the first-order scaling-invariant candidate, and the corresponding scaling-invariant
quantity is given by $C_{\psi}^1(R_0):=
R_0^{\tau}[\nabla\psi]_{C^{0,\tau}}$, where we adopt the notation $[\nabla\psi]_{C^{n,\tau}}$ as
the abbreviation of $[\nabla\psi]_{C^{n,\tau}(\mathbb{R}^{d-1})}$
(for $n=0,\cdots,k-1$) throughout the proof.
Then, using the estimate $\eqref{f:8.2}$, we have
\begin{equation}\label{f:8.3}
\begin{aligned}
 \|\nabla u_n\|_{C^{0,\tau}(D_{1}'(x'))}
& \lesssim_{\lambda,d,\tau,m_0,C_{\psi}^1}
 \dashint_{D_{2}'(x')}|\tilde{v}|
 + \big[\tilde{f}_1\big]_{C^{0,\tau}(D_{3/2}'(x'))}
 +\|\tilde{f}_1\|_{L^{\infty}(D_{3/2}'(x'))}\\
&\leq  \dashint_{D_{2}'(x')}|\tilde{v}|
 + \|\tilde{f}\|_{C^{1,\tau}(D_{3/2}'(x'))}
 +\|\nabla\psi\|_{C^{1,\tau}}
 \|\nabla\tilde{v}\|_{C^{0,\tau}(D_{3/2}'(x'))}\\
&\lesssim^{\eqref{f:8.2}}_{\lambda,d,\tau,m_0,C_{\psi}^1}[\nabla\psi]_{C^{1,\tau}}
\bigg\{\dashint_{D_{2}'(x')}|\tilde{v}|
+ \|\tilde{f}\|_{C^{0,\tau}(D_{2}'(x'))}\bigg\}
 + \big[\tilde{f}\big]_{C^{1,\tau}(D_{2}'(x'))},
\end{aligned}
\end{equation}
where we employ the interpolation inequality in the last inequality.
By the equation $\eqref{pde:8.1}$, we have
$|\tilde{\partial}^2_d\tilde{v}|\lesssim_{d,m_0}|\tilde{\nabla} \tilde{f}|+|\tilde{\nabla} A||\nabla\tilde{v}|+|\nabla u|$, where we note that $u=(u_1,\cdots,u_{n-1})$. This together with $\eqref{f:8.3}$ leads to
\begin{equation*}
\|\nabla^2 \tilde{v}\|_{C^{0,\tau}(D_{1}'(x'))}
\lesssim_{\lambda,d,\tau,m_0,C_{\psi}^1}[\nabla\psi]_{C^{1,\tau}}
\bigg\{\dashint_{D_{2}'(x')}|\tilde{v}|
+ \|\tilde{f}\|_{C^{0,\tau}(D_{2}'(x'))}\bigg\}
 + \big[\tilde{f}\big]_{C^{1,\tau}(D_{2}'(x'))}
\end{equation*}
and we consequently arrive at
\begin{equation}\label{f:8.4}
\|\nabla^2 v\|_{C^{0,\tau}(D_{1}(x))}
\lesssim_{\lambda,d,\tau,m_0,C_{\psi}^1} [\nabla\psi]_{C^{1,\tau}}^{\frac{d+3}{1+\tau}}
\bigg\{\dashint_{D_{2}(x)}|v|
+ \|f\|_{C^{1,\tau}(D_{2}(x))}\bigg\},
\end{equation}
where we employ a similar idea that given for \cite[Lemma A.3]{Josien-Otto22} to find out the well dependent power on $\|\nabla\psi\|_{C^{1,\tau}}$ in the above estimate. Let
$C_{\psi}^2:=[\nabla\psi]_{C^{1,\tau}}$ and
$C_{\psi}^2(R_0):=R_0^{1+\tau}[\nabla\psi]_{C^{1,\tau}}$.
Repeating the arguments shown in Lemma $\ref{lemma:8.1}$,
the estimate $\eqref{f:8.4}$ together with the following interior estimate
\begin{equation*}
 \|\nabla^2 v\|_{C^{0,\tau}(B_{1/2}(x))}
 \lesssim_{\lambda,d} \dashint_{B_{1}(x)}|v|
 +\big\|f\big\|_{C^{1,\tau}(B_1(x))}
\end{equation*}
implies the following scaling-invariant estimate
\begin{equation*}
r_B^2\|\nabla^2 v\|_{L^\infty(D)}
+ r_B^{2+\tau}[\nabla^2 v]_{C^{0,\tau}(D)}
 \lesssim_{\lambda,d,\tau,m_0,C_{\psi}^1(R_0),C_{\psi}^2(R_0)}
 \dashint_{2D}|v| +
 + \sum_{n=0}^{1}\Big\{r_B^{n+\tau}\big[f\big]_{C^{n,\tau}(2D)}
 +r_B^{n}\|f\|_{C^{n}(2D)}\Big\},
\end{equation*}
which is the desired estimate $\eqref{pri:8.3}$ in the case of $k=2$,
where we note that $C_{\psi}^1(R_0)\leq C_{\psi}^2(R_0)\leq R_0^{1+\tau}M_0$.

\textbf{Step 2.}
Suppose that the estimate $\eqref{pri:8.3}$ holds for the  $(k-1)^{\text{th}}$ order estimates. We will verify below that the $k^{\text{th}}$ order estimate holds.
Let $u_{\alpha}:=\tilde{\partial}_{\alpha}\tilde{v}$ with
$|\alpha| = k-1$ and $\alpha=(\alpha_1,\cdots,\alpha_{k-1})\in\{1,2,\cdots,d-1\}^{k-1}$,
and we have
\begin{equation}\label{pde:8.3}
 -\tilde{\nabla} \cdot A\tilde{\nabla}  u_\alpha = \tilde{\nabla} \cdot\big(\underbrace{\tilde{\partial}_{\alpha}\tilde{f} + \sum_{j=1}^{k-1}\sum_{|\beta|=j}C_{k-1}^j
 \tilde{\partial}_{\beta}A\tilde{\partial}_{\alpha\setminus\beta}
 \nabla\tilde{v}}_{\tilde{f}_{k-1}}\big)
 \quad\text{in}\quad D_1';
 \qquad u_{\alpha} = 0 \quad\text{on}\quad\Delta_1'.
\end{equation}
Then, using the estimate $\eqref{f:8.2}$, we have
\begin{equation}
\begin{aligned}
& \|\nabla u_\alpha\|_{C^{0,\tau}(D_{1}'(x'))}
 \lesssim_{\lambda,d,\tau,m_0,C_{\psi}^1}
 \dashint_{D_{4/3}'(x')}|u_{\alpha}|
 + \|\tilde{f}_{k-1}\|_{C^{0,\tau}(D_{4/3}'(x'))}\\
&\lesssim  \dashint_{D_{4/3}'(x')}|\nabla^{k-1}\tilde{v}|
 + \big\|\nabla^{k-1}\tilde{f}\big\|_{C^{0,\tau}(D_{4/3}'(x'))}
 +\sum_{|\beta|=1}^{k-1}\|\nabla^{|\beta|}\psi\|_{C^{1,\tau}}
 \|\nabla^{|\alpha-\beta|}\tilde{v}\|_{C^{1,\tau}(D_{3/2}'(x'))}\\
&\lesssim_{\lambda,d,\tau,m_0,C_{\psi}^1,C_{\psi}^{k-1}} \|\nabla\psi\|_{C^{k-1,\tau}}
\bigg\{\dashint_{D_{2}'(x')}|\tilde{v}|
+ \|\tilde{f}\|_{C^{k-1,\tau}(D_{2}'(x'))}\bigg\},
\end{aligned}
\end{equation}
in which we note that $C_{\psi}^{k-1}:=[\nabla\psi]_{C^{k-2,\tau}}$, and we use the inductive hypothesis to obtain the last inequality in the above estimate. Moreover, we claim that
%there exists a constant\footnote{Here, it is not necessary to pursue the magnitude of this constant. Later, one can obtain the size of this constant by a similar method as that employed in
%\cite[Lemma A.3]{Josien-Otto22}.} $\kappa\geq 1$ such that
\begin{equation}\label{f:8.5}
\|\nabla^k \tilde{v}\|_{C^{0,\tau}(D_{1}'(x'))}
\lesssim_{\lambda,d,\tau,m_0,C_{\psi}^1,C_{\psi}^{k-1},k}
[\nabla\psi]_{C^{k-1,\tau}}
\bigg\{\dashint_{D_{2}'(x')}|\tilde{v}|
+ \|\tilde{f}\|_{C^{k-1,\tau}(D_{2}'(x'))}\bigg\}.
\end{equation}
To see this, let $u_{\alpha^{(1)}}:=\tilde{\partial}_{\alpha^{(1)}}\tilde{v}$
with $|\alpha^{(1)}| = k-2$ and $\alpha^{(1)}=(\alpha_1,\cdots,\alpha_{k-2})\in\{1,2,\cdots,d-1\}^{k-2}$,
there holds
\begin{equation*}
 -\tilde{\nabla} \cdot A\tilde{\nabla}  u_{\alpha^{(1)}} = \tilde{\nabla} \cdot\big(\tilde{\partial}_{\alpha^{(1)}}\tilde{f} + \sum_{j=1}^{k-2}\sum_{|\beta|=j}C_{k-2}^j
 \tilde{\partial}_{\beta}A\tilde{\partial}_{\alpha^{(1)}\setminus\beta}
 \nabla\tilde{v}\big)=:\tilde{\nabla}\cdot \tilde{f}_{k-2}
 \quad\text{in}\quad D_1'.
\end{equation*}
This gives that $|\tilde{\partial}^2_{d}\tilde{\partial}_{\alpha^{(1)}}\tilde{v}|
\lesssim_{m_0,d}|\tilde{\nabla} \tilde{f}_{k-2}|+|\tilde{\nabla} A||\nabla\tilde{\partial}_{\alpha^{(1)}}\tilde{v}|+|\nabla u_{\alpha}|$,
and we can estimate $\big|\tilde{\nabla}^2\tilde{\partial}_{\alpha^{(1)}}\tilde{v}\big|$
by the right-hand side of $\eqref{f:8.5}$. Next,
set $u_{\alpha^{(2)}}:=\tilde{\partial}_{\alpha^{(2)}}\tilde{v}$
with $|\alpha^{(2)}| = k-3$ and $\alpha^{(2)}=(\alpha_1,\cdots,\alpha_{k-3})\in\{1,2,\cdots,d-1\}^{k-3}$,
and there correspondingly holds $ -\tilde{\nabla} \cdot A\tilde{\nabla}  u_{\alpha^{(2)}} = \tilde{\nabla}\cdot \tilde{f}_{k-3}$ in $D_1'$, and therefore we have
$ -\tilde{\partial}_d\big[\tilde{\nabla} \cdot A\tilde{\nabla}  u_{\alpha^{(2)}}\big] = \tilde{\nabla}\cdot \tilde{\partial}_{d}\tilde{f}_{k-3}$. This further leads to
$|\tilde{\partial}^3_{d}\tilde{\partial}_{\alpha^{(2)}}\tilde{v}|
\lesssim_{m_0,d}|\tilde{\nabla}^2 \tilde{f}_{k-3}|
+|\tilde{\nabla}^2A||\nabla\tilde{\partial}_{\alpha^{(2)}}\tilde{v}|
+|\tilde{\nabla}A||\nabla^2\tilde{\partial}_{\alpha^{(1)}}\tilde{v}|+|\nabla u_{\alpha}|$, so we can now estimate $\big|\tilde{\nabla}^3\tilde{\partial}_{\alpha^{(2)}}\tilde{v}\big|$
by the right-hand side of $\eqref{f:8.5}$. To get the final result,
we continue to adopt the method of induction, and it is
assumed that $\big|\tilde{\nabla}^{m+1}\tilde{\partial}_{\alpha^{(m)}}\tilde{v}\big|$
with $|\alpha^{(m)}| = k-m-1$ (where  $\alpha^{(m)}=(\alpha_1,\cdots,\alpha_{k-m-1})\in\{1,2,\cdots,d-1\}^{k-m-1}$
and $m=1,\cdots, k-2$) is controlled by the right-hand side of $\eqref{f:8.5}$. From the identity $ -\tilde{\partial}_d^{k-2}\big[\tilde{\nabla} \cdot A\tilde{\nabla}  \tilde{v}\big] = \tilde{\nabla}\cdot \tilde{\partial}_{d}^{k-2}\tilde{f}$,
it follows that
$|\tilde{\partial}^{k}_{d}\tilde{v}|
\lesssim_{m_0,d}|\tilde{\nabla}^{k-1} \tilde{f}|
+\sum_{m=1}^{k-1}|\tilde{\nabla}^mA|
|\tilde{\nabla}^{k-m}\tilde{v}|
+\sum_{m=1}^{k-2}|\tilde{\nabla}^{k-2-m}A|
|\tilde{\nabla}^{m+1}\tilde{\partial}_{\alpha^{(m)}}\tilde{v}|$,
and this together with the inductive hypothesis
and the interpolation inequality
on $\tilde{\nabla}^m A$ consequently leads to
the desired estimate $\eqref{f:8.5}$. Again,
by a similar method as that employed in \cite[Lemma A.3]{Josien-Otto22},
we have
\begin{equation*}
\|\nabla^k v\|_{C^{0,\tau}(D_{1}(x))}
\lesssim_{\lambda,d,\tau,m_0,k,C_{\psi}^1,C_{\psi}^{k-1}} [\nabla\psi]_{C^{k-1,\tau}}^{\frac{k+1+d}{k-1+\tau}}
\bigg\{\dashint_{D_{2}(x)}|v|
+ \|f\|_{C^{k-1,\tau}(D_{2}(x))}\bigg\},
\end{equation*}
and this together the corresponding interior estimates
gives the desired estimate $\eqref{pri:8.3}$, where we set
$C_{\psi}^{k}:=[\nabla\psi]_{C^{k-1,\tau}}$ and
$C_{\psi}^{k}(R_0):=R_0^{k-1+\tau}[\nabla\psi]_{C^{k-1,\tau}}$.
Note that $C_{\psi}^{k-1}(R_0)\lesssim_{k,d}
C_{\psi}^{1}(R_0) + C_{\psi}^{k}(R_0)$, and therefore
the estimated constant in $\eqref{pri:8.3}$ depends only on
$\lambda$, $d$, $m_0$, $\tau$, $k$, $C_{\psi}^{1}(R_0)$, and
$C_{\psi}^{k}(R_0)$. This ends the whole proof.
\end{proof}

\end{document}